\documentclass[english]{article}

\usepackage[utf8]{inputenc} % allow utf-8 input
\usepackage[T1]{fontenc}    % use 8-bit T1 fonts
\usepackage{hyperref}       % hyperlinks
\usepackage{url}            % simple URL typesetting
\usepackage{booktabs}       % professional-quality tables
\usepackage{amsfonts}       % blackboard math symbols
\usepackage{nicefrac}       % compact symbols for 1/2, etc.
\usepackage{microtype}      % microtypography
\usepackage{amsmath}
\usepackage{amsthm}
\usepackage{amssymb}
\usepackage{babel}
\usepackage{framed}
\usepackage{graphicx}
\usepackage{mathtools}
\newtheorem{theorem}{Theorem}
\newtheorem{lemma}{Lemma}

\usepackage{hyperref}
\usepackage{bbm}
\usepackage{subcaption}
\usepackage{clipboard}
\usepackage{floatrow}
\hypersetup{colorlinks,linkcolor={blue},citecolor={blue},urlcolor={red}}  
\makeatletter

\newcommand{\mr}{\mathrm}
\renewcommand{\mathbf}{\boldsymbol}
\newcommand{\mb}{\mathbf}
\newcommand{\bb}{\mathbb}
\newcommand{\mc}{\mathcal} 
\newcommand{\reals}{\bb R}
\newcommand{\hess}[2]{\mr{Hess}\left[#1 \right]\left(#2\right)}
\newcommand{\Hess}[2]{\mr{Hess}[#1](#2)}
\newcommand{\grad}[2]{\mr{grad}\left[#1 \right]\left(#2 \right)}
\newcommand{\Grad}[2]{\mr{grad}[#1](#2)}

\newcommand{\Exp}{\mr{Exp}}

\newcommand{\fs}{f_{\mr{Sep}}}
\newcommand{\set}[1]{\left\{ #1 \right\}}
\newcommand{\sech}{\mathrm{sech}}
\DeclarePairedDelimiter\abs{\lvert}{\rvert}%
\DeclarePairedDelimiter\norm{\lVert}{\rVert}%
\DeclarePairedDelimiter\ip{\langle}{\rangle}%
\newcommand{\V}{\boldsymbol}
\newcommand{\lbl}{\label}
\newcommand{\cdlmu}{c_1}
\newcommand{\cetadlpop}{c_2}
\newcommand{\cdistfrommin}{c_3}
\newcommand{\cmudlpop}{c_4}
\newcommand{\cetadl}{c_5}
\newcommand{\cmudl}{c_6}
\newcommand{\cz}{c_7}
\newcommand{\cb}{c_8}
\newcommand{\Cdlpop}{C_1}
\newcommand{\Cdl}{C_2}

\newcommand{\cdl}{c_{\theta}}

\usepackage[T3,T1]{fontenc}
\DeclareSymbolFont{tipa}{T3}{cmr}{m}{n}
\DeclareMathAccent{\invbreve}{\mathalpha}{tipa}{16}

\newcommand{\Anc}{\invbreve{A}}

\title{Efficient Dictionary Learning with Gradient Descent}

	\author{Dar Gilboa \thanks{Department of Neuroscience, Columbia University} \thanks{Data Science Institute, Columbia University}, Sam Buchanan \thanks{Department of Electrical Engineering, Columbia University} \footnotemark[2], John Wright \footnotemark[3] \footnotemark[2]
}
\begin{document}
	\setlength{\abovedisplayskip}{4pt}
	\setlength{\belowdisplayskip}{4pt}
\maketitle
%Dispersive Saddle Functions: \\ Efficient Nonconvex Optimization with Gradient Descent
\begin{abstract}
Randomly initialized first-order optimization algorithms are the method of choice for solving many high-dimensional nonconvex problems in machine learning, yet general theoretical guarantees cannot rule out convergence to critical points of poor objective value. For some highly structured nonconvex problems however, the success of gradient descent can be understood by studying the geometry of the objective. We study one such problem -- complete orthogonal dictionary learning, and provide converge guarantees for randomly initialized gradient descent to the neighborhood of a global optimum. The resulting rates scale as low order polynomials in the dimension even though the objective possesses an exponential number of saddle points. This efficient convergence can be viewed as a consequence of negative curvature normal to the stable manifolds associated with saddle points, and we provide evidence that this feature is shared by other nonconvex problems of importance as well. 
\end{abstract}
% !TEX root = dispersive-full.tex

\section{Introduction}

Many central problems in machine learning and signal processing are most naturally formulated as optimization problems. These problems are often both nonconvex and high-dimensional. High dimensionality makes the evaluation of second-order information prohibitively expensive, and thus randomly initialized first-order methods are usually employed instead. This has prompted great interest in recent years in understanding the behavior of gradient descent on nonconvex objectives \cite{hardt2015train,ge2015escaping,hardt2016gradient,dauphin2014identifying}. General analysis of first- and second-order methods on such problems can provide guarantees for convergence to critical points but these may be highly suboptimal, since nonconvex optimization is in general an NP-hard probem
\cite{bertsekas1999nonlinear}. Outside of a convex setting \cite{nesterov2013introductory} one must assume additional structure in order to make statements about convergence to optimal or high quality solutions. It is a curious fact that for certain classes of problems such as ones that involve sparsification \cite{lee2013near,bronstein2005blind} or matrix/tensor recovery \cite{keshavan2010matrix,jain2013low,anandkumar2014guaranteed} first-order methods can be used effectively. Even for some highly nonconvex problems where there is no ground truth available such as the training of neural networks first-order methods converge to high-quality solutions \cite{zhang2016understanding}. 

Dictionary learning is a problem of inferring a sparse representation of data that was originally developed in the neuroscience literature \cite{olshausen1996emergence}, and has since seen a number of important applications including image denoising, compressive signal acquisition and signal classification \cite{elad2006image,mairal2014sparse}. In this work we study a formulation of the dictionary learning problem that can be solved efficiently using randomly initialized gradient descent despite possessing a number of saddle points exponential in the dimension. A feature that appears to enable efficient optimization is the existence of sufficient negative curvature in the directions normal to the stable manifolds of all critical points that are not global minima \footnote{As well as a lack of spurious local minimizers, and the existence of large gradients or strong convexity in the remaining parts of the space}. %We refer to such functions as \textit{dispersive saddle} functions.% and choose to abstain from providing a more precise definition of this class at the present time
This property ensures that the regions of the space that feed into small gradient regions under gradient flow do not dominate the parameter space. Figure \ref{sadsfig} illustrates the value of this property: negative curvature prevents measure from concentrating about the stable manifold. As a consequence randomly initialized gradient methods avoid the ``slow region'' of around the saddle point. 
%\footnote{While in the proofs we found it more straightforward to avoid using the curvature properties directly but instead rely on the effects of curvature to bound the gradient, we believe the true reason that such problems are amenable to this analysis stems from their second-order properties.}. 

\begin{figure} 
	%  \includegraphics[height=2in,width=3.25in]{../experiments/compare_saddles2.pdf}
	%  \caption{{\bf Negative curvature helps gradient descent.} Red: ``slow region'' of small gradient around a saddle point. Green: stable manifold associated with the saddle point. Black: points that flow to the slow region. Left: global negative curvature normal to the stable manifold. Right: positive curvature normal to the stable manifold -- randomly initialized gradient descent is more likely to encounter the slow region.} \label{sadsfig}
	
	\floatbox[{\capbeside\thisfloatsetup{capbesideposition={right,top},capbesidewidth=5.3cm}}]{figure}[\FBwidth]
	{\caption{{\bf Negative curvature helps gradient descent.} Red: ``slow region'' of small gradient around a saddle point. Green: stable manifold associated with the saddle point. Black: points that flow to the slow region. Left: global negative curvature normal to the stable manifold. Right: positive curvature normal to the stable manifold -- randomly initialized gradient descent is more likely to encounter the slow region.}\label{sadsfig}}
	{ \includegraphics[height=2in,width=3.25in]{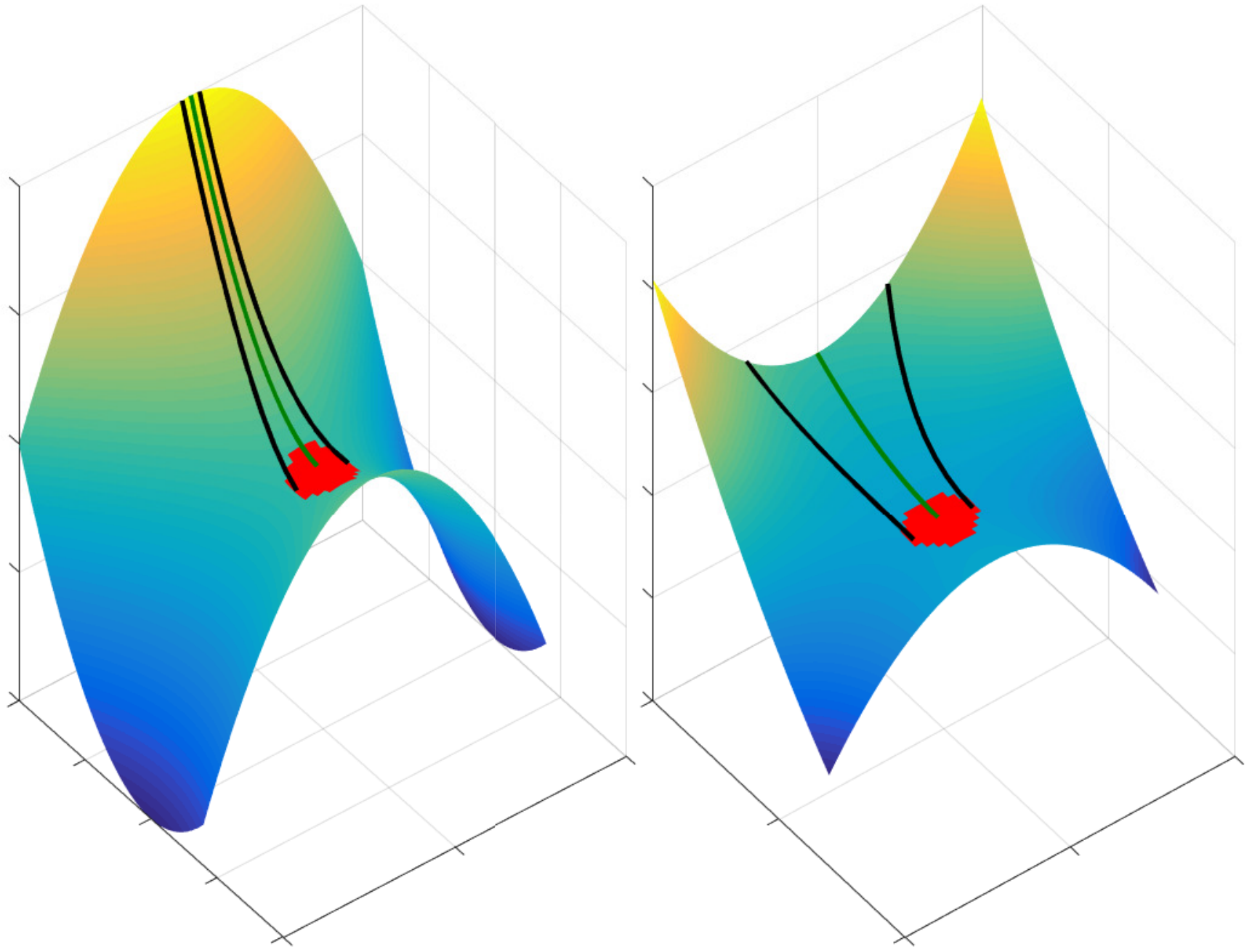}} 
\end{figure}

The main results of this work is a  convergence rate for randomly initialized gradient descent for complete orthogonal dictionary learning to the neighborhood of a global minimum of the objective. Our results are probabilistic since they rely on initialization in certain regions of the parameter space, yet they allow one to flexibly trade off between the maximal number of iterations in the bound and the probability of the bound holding. %(with the amount of data required depending on the desired success probability as well). 

While our focus is on dictionary learning, it has been recently shown that for other important nonconvex problems such as phase retrieval \cite{chen2018gradient} performance guarantees for randomly initialized gradient descent can be obtained as well. In fact, in Appendix \ref{appgpr} we show that negative curvature normal to the stable manifolds of saddle points (illustrated in Figure \ref{sadsfig}) is also a feature of the population objective of generalized phase retrieval, and can be used to obtain an efficient convergence rate. 

%In fact, one can obtain convergence rates using an analysis very similar to the one presented here (see Appendix ?/remove this sentence?). %This may lead one to conjecture that other nonconvex problems where local methods perform well may possess similar geometric properties. 

%It is possible that other sparsification problems such as phase retrieval \cite{sun2016geometric} and perhaps orthogonal tensor decomposition \cite{ge2015escaping} might be amenable to similar analysis.

% !TEX root = dispersive-full.tex

\section{Related Work} \label{relsec}

\paragraph{Easy nonconvex problems.} There are two basic impediments to solving nonconvex problems globally: {\bf (i) spurious local minimizers}, and {\bf (ii) flat saddle points}, which can cause methods to stagnate in the vicinity of critical points that are not minimizers. The latter difficulty has motivated the study of {\em strict saddle functions} \cite{sun2015nonconvex,ge2015escaping}, which have the property that at every point in the domain of optimization, there is a large gradient, a direction of strict negative curvature, or the function is strongly convex. By leveraging this curvature information, it is possible to escape saddle points and obtain a local minimizer in polynomial time.\footnote{This statement is nontrivial: finding a local minimum of a smooth function is NP-hard.} Perhaps more surprisingly, many known strict saddle functions also have the property that every local minimizer is global; for these problems, this implies that efficient methods find global solutions. Examples of problems with this property include variants of sparse dictionary learning \cite{sun2017complete}, phase retrieval \cite{sun2016geometric}, tensor decomposition \cite{ge2015escaping}, community detection \cite{bandeira2016low} and phase synchronization \cite{boumal2016nonconvex}. 

\paragraph{Minimizing strict saddle functions.} Strict saddle functions have the property that at every saddle point there is a direction of strict negative curvature. A natural approach to escape such saddle points is to use second order methods (e.g., trust region \cite{conn2000trust} or curvilinear search \cite{goldfarb1980curvilinear}) that explicitly leverage curvature information. Alternatively, one can attempt to escape saddle points using first order information only. However, some care is needed: canonical first order methods such as gradient descent will not obtain minimizers if initialized at a saddle point (or at a point that flows to one) -- at any critical point, gradient descent simply stops. A natural remedy is to randomly perturb the iterate whenever needed. A line of recent works shows that noisy gradient methods of this form efficiently optimize strict saddle functions \cite{lee2016gradient,du2017gradient,jin2017escape}. For example, \cite{jin2017escape} obtains rates on strict saddle functions that match the optimal rates for smooth convex programs up to a polylogarithmic dependence on dimension.\footnote{This work also proves convergence to a second-order stationary point under more general smoothness assumptions.}

%A recent line of work studies the performance of gradient descent with noise injection \cite{lee2016gradient,du2017gradient,jin2017escape}, and shows that these methods efficiently minimize strict saddle functions. 

\paragraph{Randomly initialized gradient descent?} The aforementioned results are broad, and nearly optimal. Nevertheless, important questions about the behavior of first order methods for nonconvex optimization remain unanswered. For example: {\em in every one of the aforemented benign nonconvex optimization problems, randomly initialized gradient descent rapidly obtains a minimizer.} This may seem unsurprising: general considerations indicate that the stable manifolds associated with non-minimizing critical points have measure zero \cite{nicolaescu2011invitation}, this implies that a variety of small-stepping first order methods converge to minimizers in the large-time limit \cite{lee2017first}. However, it is not difficult to construct strict saddle problems that {\em are not} amenable to efficient optimization by randomly initialized gradient descent -- see \cite{du2017gradient} for an example. This contrast between the excellent empirical performance of randomly initialized first order methods and worst case examples suggests that there are important geometric and/or topological properties of ``easy nonconvex problems'' that are not captured by the strict saddle hypothesis. Hence, the motivation of this paper is twofold: (i) to provide theoretical corroboration (in certain specific situations) for what is arguably the simplest, most natural, and most widely used first order method, and (ii) to contribute to the ongoing effort to identify conditions which make nonconvex problems amenable to efficient optimization. 

\section{Dictionary Learning over the Sphere} \label{objsec}

Suppose we are given data matrix $\mb Y = \left[ \mb y_1, \dots \mb y_p \right] \in \reals^{n \times p}$. The {\em dictionary learning} problem asks us to find a concise representation of the data \cite{elad2006image}, of the form $\mb Y \approx \mb A \mb X$, where $\mb X$ is a sparse matrix. In the {\em complete, orthogonal dictionary learning} problem, we restrict the matrix $\mb A$ to have orthonormal columns ($\mb A \in O(n)$). This variation of dictionary learning is useful for finding concise representations of small datasets (e.g., patches from a single image, in MRI \cite{ravishankar2011mr}). 

To analyze the behavior of dictionary learning algorithms theoretically, it useful to posit that $\mb Y = \mb A_0\mb X_0$ for some true dictionary $\mb A_0 \in O(n)$ and sparse coefficient matrix $\mb X_0 \in \reals^{n \times p}$, and ask whether a given algorithm recovers the pair $(\mb A_0, \mb X_0)$.\footnote{This problem exhibits a sign permutation symmetry: $\mb A_0 \mb X_0 = (\mb A_0 \mb \Gamma ) (\mb \Gamma^* \mb X_0)$ for any signed permutation matrix $\mb \Gamma$. Hence, we only ask for recovery up to a signed permutation.} In this work, we further assume that the sparse matrix $\mb X_0$ is random, with entries i.i.d.\ Bernoulli-Gaussian\footnote{$[\mb X_0]_{ij} = \mb V_{ij} \mb \Omega_{ij}$, with $\mb V_{ij} \sim\mathcal{N}(0,1)$, $\mb \Omega_{ij} \sim \mathrm{Bern}(\theta)$ independent.}. For simplicity, we will let $\mb A_0 = \mb I$; our arguments extend directly to general $\mb A_0$ via the simple change of variables $\mb q \mapsto \mb A_0^* \mb q$. 

%Given $p$ observations $\mb y_{k}\in\mathbb{R}^{n}$ arranged into the
%columns of the matrix $\mathbf{Y}\in\mathbb{R}^{n\times p}$, the aim of dictionary learning is
%to recover matrices $\mb A$ and $\mb X$ such that $\mb Y=\mb A\mb X$. We consider here the complete orthogonal setting
%$\mb A\in O(n)$, with the elements of $\mb X$ i.i.d. Bernoulli-Gaussian  
%. There is a discrete symmetry in this problem since the rows of $\mb A$
%can be recovered only up to signs and permutations. 

\cite{spielman2012exact} showed that under mild conditions, the complete dictionary recovery problem can be reduced to the geometric problem of finding a sparse vector in a linear subspace \cite{qu2014finding}. Notice that because $\mb A_0$ is orthogonal, $\mathrm{row}(\mb Y) = \mathrm{row}(\mb X_0)$. Because $\mb X_0$ is a sparse random matrix, the rows of $\mb X_0$ are sparse vectors. Under mild conditions \cite{spielman2012exact}, they are the {\em sparsest} vectors in the row space of $\mb Y$, and hence can be recovered by solving the conceptual optimization problem 
\[
\min \; \left\Vert \mb q^{\ast}\mb Y\right\Vert_{0} \quad \mathrm{s.t.} \quad \mb q^{\ast} \mb Y\neq \mb 0. 
\]
This is not a well-structured optimization problem: the objective is discontinuous, and the constraint set is open. A natural remedy is to replace the $\ell^0$ norm with a smooth sparsity surrogate, and to break the scale ambiguity by constraining $\mb q$ to the sphere, giving
\begin{equation} \label{dl2eq}
\min \;  f_{\mr{DL}}(\mb q) \equiv \frac{1}{p} \underset{k=1}{\overset{p}{\sum}} h_\mu( \mb q^{\ast}\mb y_{k} ) \quad \mathrm{s.t.} \quad  \mb q \in \mathbb{S}^{n-1}.
\end{equation}
Here, we choose $h_\mu( t) = \mu \log(\cosh(t/\mu))$ as a smooth sparsity surrogate.
 This objective was analyzed in \cite{sun2015complete}, which showed that (i) although this optimization problem is nonconvex, when the data are sufficiently large, with high probability every local optimizer is near a signed column of the true dictionary $\mb A_0$, (ii) every other critical point has a direction of strict negative curvature, and (iii) as a consequence, a second-order Riemannian trust region method efficiently recovers a column of $\mb A_0$.\footnote{Combining with a deflation strategy, one can then efficiently recover the entire dictionary $\mb A_0$.} The Riemannian trust region method is of mostly theoretical interest: it solves complicated (albeit polynomial time) subproblems that involve the Hessian of $f_{\mr{DL}}$. 

In practice, simple iterative methods, including randomly initialized gradient descent are also observed to rapidly obtain high-quality solutions. In the sequel, we will give a geometric explanation for this phenomenon, and bound the rate of convergence of randomly initialized gradient descent to the neighborhood of a column of $\mb A_0$. Our analysis of $f_{\mr{DL}}$ is probabilistic in nature: it argues that with high probability in the sparse matrix $\mb X_0$, randomly initialized gradient descent rapidly produces a minimizer.

To isolate more clearly the key intuitions behind this analysis, we first analyze the simpler {\em separable objective}
\begin{equation} \label{sepeq}
\min \; \fs(\mb q) \equiv \overset{n}{\underset{i=1}{\sum}} h_{\mu}( \mb q_i ) \quad \mr{s.t.} \quad \mb q \in \bb S^{n-1}. 
\end{equation}
Figure \ref{sepobjfig} plots both $\fs$ and $f_{\mr{DL}}$ as functions over the sphere. Notice that many of the key geometric features in $f_{\mr{DL}}$ are present in $\fs$; indeed, $\fs$ can be seen as an ``ultrasparse'' version of $f_{\mr{DL}}$ in which the columns of the true sparse matrix $\mb X_0$ are taken to have only one nonzero entry. A virtue of this model function is that its critical points and their stable manifolds have simple closed form expressions (see Lemma \ref{seplemma}). 

%The similarity between the separable objective and the dictionary learning objective, which is apparent visually in low dimensions, can be understood when considering the case where the data is sparse. If $\theta = \frac{1}{n}$, in expectation only one element of every $y_k$ is non-zero. The resulting population objective will then resemble a coordinate-wise sum over a sparsity inducing function, which is precisely the form of the separable objective $\fs(\mb q)$. 
%
%A low dimensional visualization can be seen in Figure \ref{sepobjfig}.  

\begin{figure} 
	\floatbox[{\capbeside\thisfloatsetup{capbesideposition={right,top},capbesidewidth=3.5cm}}]{figure}[\FBwidth]
	{\caption{\textit{Left:} The separable objective for $n=3$. Note the similarity to the dictionary learning objective. \textit{Right:} The objective for complete orthogonal dictionary learning (discussed in section \ref{dlsec}) for $n=3$.}\label{sepobjfig}}
	{ \includegraphics[height=1.4in]{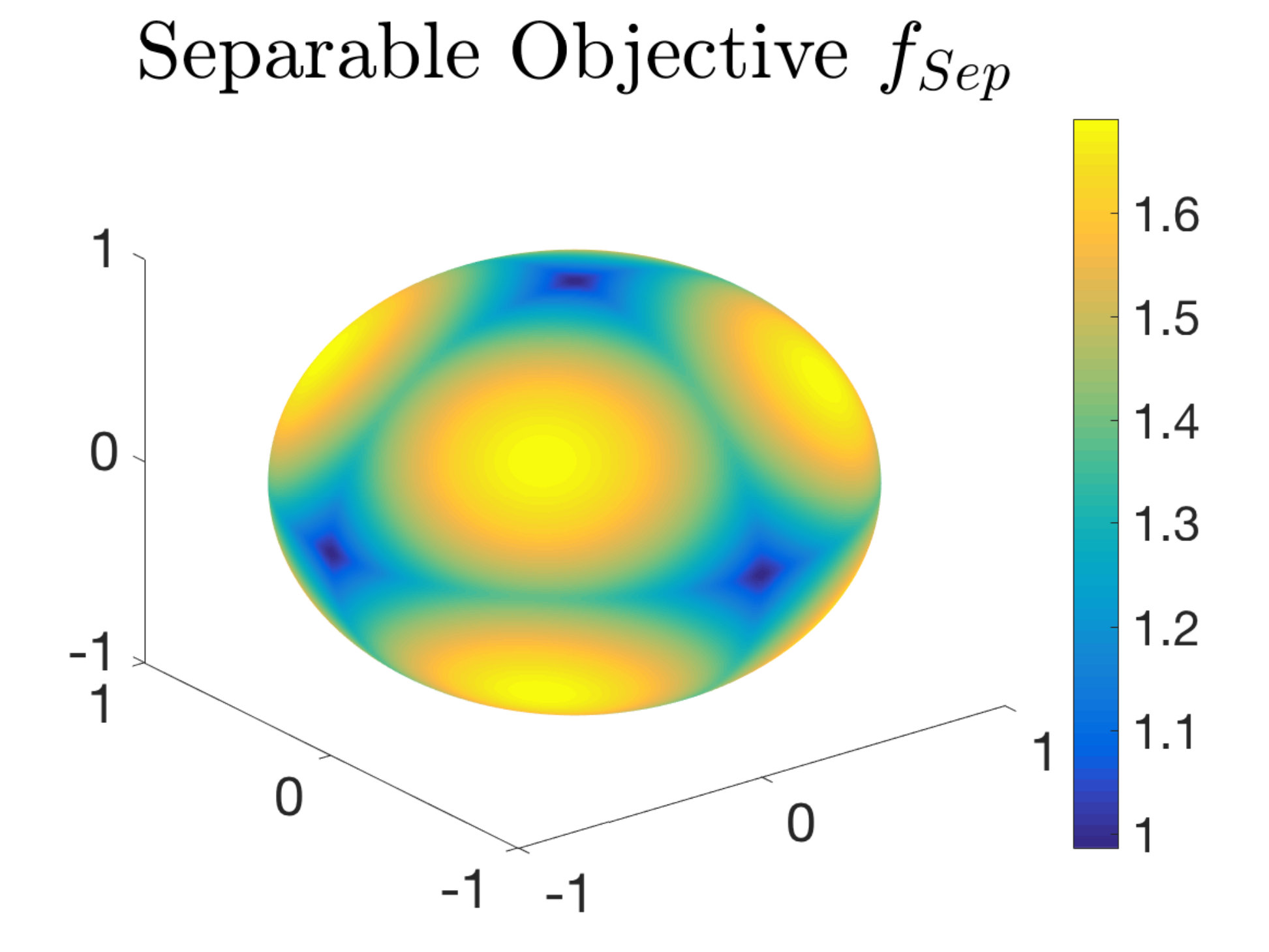}
		\includegraphics[height=1.4in]{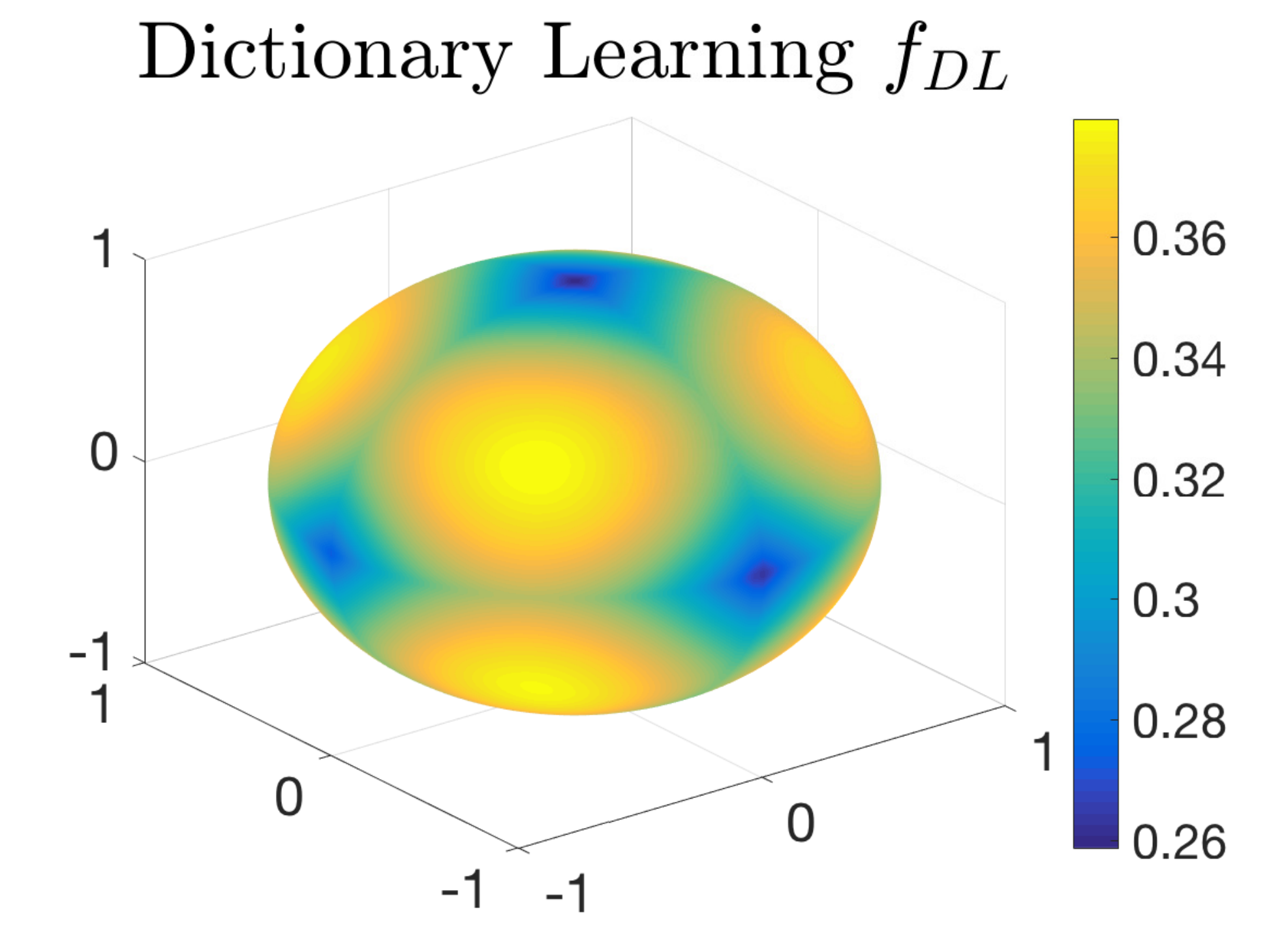}} 

\end{figure}

% !TEX root = dispersive-full.tex

\section{Outline of Important Geometric Features} \label{bgsec}
Our problems of interest have the form
\[
{\min}\; f(\mb q) \quad \mathrm{s.t.} \quad \mb q \in \bb S^{n-1}, 
\]
where $f : \reals^n \to \reals$ is a smooth function. We let $\nabla f(\mb q)$ and $\nabla ^2 f(\mb q)$ denote the Euclidean gradient and hessian (over $\reals^n$), and let $\grad{f}{\mb q}$ and $\hess{f}{\mb q}$ denote their Riemannian counterparts (over $\mathbb{S}^{n-1}$). We will obtain results for Riemannian gradient descent defined by the update 
\[
\mathbf{q}\rightarrow\exp_{\mathbf{q}}(-\eta \,\mathrm{grad}[f](\mathbf{q}))
\]
for some step size $\eta > 0$, where $\exp_{\mathbf{q}}:T_{\mathbf{q}}\mathbb{S}^{n-1}\rightarrow\mathbb{S}^{n-1}$ is the exponential map. The Riemannian gradient on the sphere is given by $\mathrm{grad}[f](\mathbf{q})=(\mathbf{I}-\mathbf{qq}^{\ast})\nabla f(\mathbf{q})$.

We let $A$ denote the set of critical points of $f$ over $\mathbb{S}^{n-1}$ -- these are the points $\bar{\mb q}$ s.t.\  $\grad{f}{\bar{\mb q}} = \mb 0$. We let $\breve{A}$ denote the set of local minimizers, and $\Anc$ its complement. Both $f_{\mr{Sep}}$ and $f_{\mr{DL}}$ are {\em Morse functions} on $\bb S^{n-1}$,\footnote{Strictly speaking, $f_{\mr{DL}}$ is Morse with high probability, due to results of \cite{sun2017complete}.} we can assign an index $\alpha$ to every $\bar{\mb q} \in A$, which is the number of negative eigenvalues of $\hess{f}{\bar{\mb q}}$.

Our goal is to understand when gradient descent efficiently converges to a local minimizer. In the small-step limit, gradient descent follows gradient flow lines $\mb \gamma : \bb R \to \mc M$, which are solution curves of the ordinary differential equation 
\[
\dot{\mb \gamma}(t)=-\grad{f}{\mb \gamma(t)}
\]
To each critical point $\mb \alpha \in A$ of index $\lambda$, there is an associated {\em stable manifold} of dimension $\mr{dim}(\mc M) - \lambda$, which is roughly speaking, the set of points that flow to $\alpha$ under gradient flow:
\[
W^{s}(\mb \alpha) \equiv \left\{\mb q\in\mathcal{M}\; \middle| \; \begin{array}{l} \underset{t\rightarrow \infty}{\lim}\mb \gamma(t)=\mb \alpha \\ \text{\scriptsize $\mb \gamma$ a gradient flow line s.t.\ $\mb \gamma(0) =\mb q$} \end{array} \right\}.
\]

%From the stable manifold theorem \cite{nicolaescu2011invitation} $ W^{s}(\alpha)$ is diffeomorphic
%to the open disk $D^{n-1-\lambda}$. The stable manifolds partition
%$\mathcal{M}$ into disjoint sets such that $\mathcal{M}=\underset{a}{\bigcup}W^{s}(\alpha)$.

%\jw{i feel the below generalities on hardness of nonconvex optimization are out of place here -- shouldn't we have addressed this in the introduction?} 
%\dg{moved some of this paragraph to intro and removed the rest}
%{\color{red}
%Nonconvex optimization is in general an NP-hard probem
%\cite{bertsekas1999nonlinear}, and the main issues that arise
%when using first-order methods on such problems are convergence
%to first-order critical points of poor objective value (saddle
%points or spurious local minimizers) or slow convergence due to
%small gradient values. In all the problems considered here there
%are no spurious local minimizers and thus the main challenge will be ruling out convergence to saddles or slow progress from small gradient regions. Having ruled this out, one can obtain an efficient rate for convergence to the vicinity of a global optimum. }

Our analysis uses the following convenient coordinate chart 
\begin{equation} \label{phieq}
\mb \varphi(\mb w)= \left(\mb w,\sqrt{1-\left\Vert \mb w\right\Vert ^{2}}\right) \equiv \mb q(\mb w)
\end{equation}
where $\mb w \in B_1(0)$. We also define two useful sets: 
\[
\mathcal{C} \equiv \{ \mb q\in\mathbb{S}^{n-1}|q_{n}\geq\left\Vert \mb w\right\Vert _{\infty}\}
\]
\begin{equation} \label{czeq}
\mathcal{C}_{\zeta} \equiv \left\{\mb q\in\mathbb{S}^{n-1} \; \middle | \; \frac{q_{n}}{\left\Vert \mb w\right\Vert _{\infty}}\geq 1+\zeta\right\}.
\end{equation}

Since the problems considered here are symmetric with respect to a signed permutation of the coordinates we can consider a certain $\mathcal{C}$ and the results will hold for the other symmetric sections as well. We will show that at every point in $\mathcal{C}$ aside from a neighborhood of a global minimizer for the separable objective (or a solution to the dictionary problem that may only be a local minimizer), there is either a large gradient component in the direction of the minimizer or negative curvature in a direction normal to $\partial \mathcal{C}$. For the case of the separable objective, one can show that the stable manifolds of the saddles lie on this boundary, and hence this curvature is normal to the stable manifolds of the saddles and allows rapid progress away from small gradient regions and towards a global minimizer \footnote{The direction of this negative curvature is important here, and it is this feature that distinguishes these problems from other problems in the strict-saddle class where this direction may be arbitrary}. These regions are depicted in Figure \ref{curvfig}.

\begin{figure}
	\includegraphics[width=2.5in]{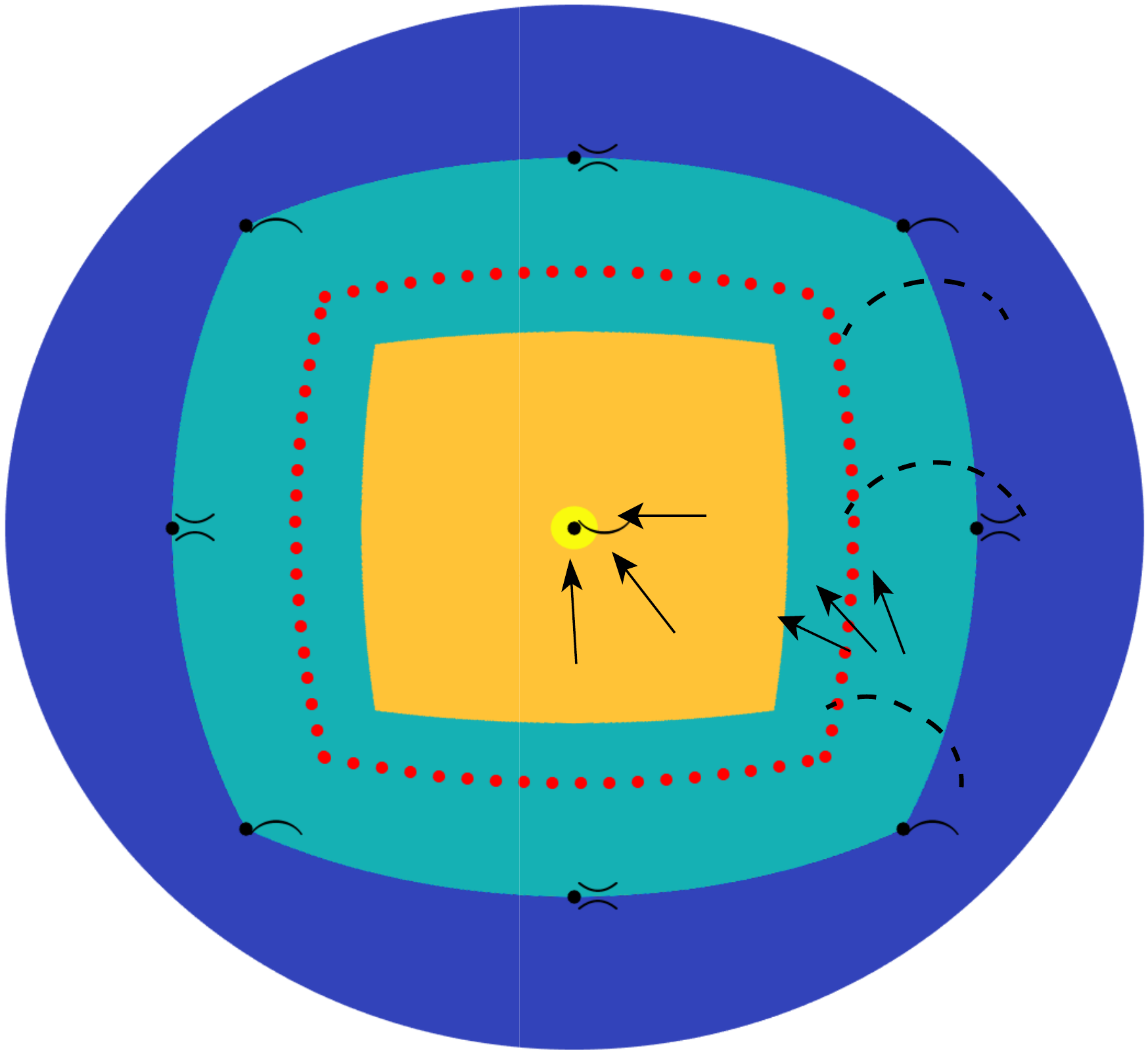}	
	\caption{{\bf Negative curvature and efficient gradient descent.} The union of the light blue, orange and yellow sets is the set $\mathcal{C}$. In the light blue region, there is negative curvature normal to $\partial \mathcal{C}$, while in the orange region the gradient norm is large, as illustrated by the arrows. There is a single global minimizer in the yellow region. For the separable objective, the stable manifolds of the saddles and maximizers all lie on $\partial \mathcal{C}$ (the black circles denote the critical points, which are either maximizers "$\smallfrown$", saddles "$\asymp$", or minimizers "$\smallsmile$"). The red dots denote $\partial \mathcal{C}_\zeta$ with $\zeta = 0.2$.} \label{curvfig}
\end{figure}

In the sequel, we will make the above ideas precise for the two specific nonconvex optimization problems discussed in Section \ref{objsec} and use this to obtain a convergence rate to a neighborhood of a global minimizer. Our analysis are specific to these problems. However, as we will describe in more detail later, they hinge on important geometric characteristics of these problems which make them amenable to efficient optimization, which may obtain in much broader classes of problems. 

%which we view as standins for general problems of {\em sparsification} and {\em matrix recovery}. 

%The exact form of $Q_\varepsilon$ will generally be quite unwieldy, but we will construct simple upper bounds for it. Figure \ref{setsfig} show these regions for the case of the dictionary learning objective discussed in Section \ref{dlsec}. The general structure of all the proofs in this work will be to prove that efficient optimization is possible if a randomly initialized method is not initialized in $Q_\varepsilon$, and then to bound the probability of this event by constructing bounds for $\mathrm{Vol}(Q_\varepsilon)$.

%One can analogously define unstable manifolds by considering the $t\rightarrow-\infty$ limit instead, in which case the resulting submanifold is diffeomorphicto $D^{\lambda}$. 
%We then define the $m-1$ skeleton of $\mathcal{M}$
%as the union of the stable manifolds of critical points that are not
%minima:
%
%\begin{equation} \label{Keq}
%K \equiv \underset{\alpha,\lambda_{\alpha}>0}{\bigcup}W^{s}(\alpha)
%\end{equation}

% !TEX root = dispersive-full.tex

\section{Separable Function Convergence Rate} \label{secsep}

In this section, we study the behavior of randomly initialized gradient descent on the separable function $f_{\mr{Sep}}$. We begin by characterizing the critical points:

\begin{lemma}[Critical points of $f_{\mr{Sep}}$] \label{seplemma}
	The critical points of the separable problem (\ref{sepeq}) are 
	\begin{equation} \label{Aeq}
	A=\left\{ \mathcal{P}_{\mathbb{S}^{n-1}}[\mb a]\left|\mb a\in\{-1,0,1\}^{\otimes n},\left\Vert \mb a\right\Vert >0\right.\right\}. 
	\end{equation}
	For every $\mb \alpha \in A$ and corresponding $\mb a(\mb \alpha)$, for $\mu < \frac{c}{\sqrt{n}\log n}$ the stable manifold of $\mb \alpha$ takes the form 
	\begin{equation} \label{weq}
	W^{s}(\mb \alpha)=\left\{ \mathcal{P}_{\mathbb{S}^{n-1}}\left[ \, \mb a(\mb \alpha)+\mb b \, \right] \;\middle | \;\begin{array}{c}
	\mathrm{supp}(\mb a(\mb \alpha)) 
	\cap\mathrm{supp}(\mb b) =\varnothing,\\
	\left\Vert \mb b\right\Vert _{\infty}<1
	\end{array}\right\} 
	\end{equation}
	where $c>0$ is a numerical constant.
\end{lemma}
\begin{proof}
	Please see Appendix \ref{Rpf}
\end{proof}

By inspecting the dimension of the stable manifolds, it is easy to verify that that
there are $2n$ global minimizers at the 1-sparse vectors on the
sphere $\pm\widehat{\mb e}_{i}$, $2^n$ maximizers at the least
sparse vectors and an exponential number of saddle points of
intermediate sparsity. This is because the dimension of $
W^s(\alpha)$ is simply the dimension of $b$ in \ref{weq}, and it
follows directly from the stable manifold theorem that only
minimizers will have a stable manifold of dimension $n-1$. The
objective thus possesses no spurious local minimizers.

%\begin{figure}
%\includegraphics[height=1.5in]{../figures/talk_figs/sparsity_funcs.pdf}
%\caption{The smoothed version of the $L^1$ loss used in the separable objective. \dg{fig can be cut to save space}}
%\label{sparsefig}
%\end{figure} 

When referring to critical points and stable manifolds from now on we refer only to those that are contained in $\mathcal{C}$ or on its boundary. It is evident from Lemma \ref{seplemma} that the critical points in $\Anc$ all lie on $\partial \mathcal{C}$ and that $\underset{\mb \alpha \in \Anc}{\bigcup}W^{s}(\mb \alpha)=\partial \mathcal{C}$ , and there is a minimizer at its center given by $\mb q(\mb 0)=\widehat{\mb e}_n$.

\subsection{The effect of negative curvature on the gradient}

We now turn to making precise the notion that negative curvature normal to stable manifolds of saddle points enables gradient descent to rapidly exit small gradient regions. We do this by defining vector fields $\mb u^{(i)}(\mb q),i\in [n-1]$ such that each field is normal to a continuous piece of $\partial \mathcal{C}_\zeta$ and points outwards relative to $\mathcal{C}_\zeta$ defined in \ref{czeq}. By showing that the Riemannian gradient projected in this direction is positive and proportional to $\zeta$, we are then able to show that gradient descent acts to increase $\zeta(\mathbf{q}(\mathbf{w}))=\frac{q_{n}}{\left\Vert \mathbf{w}\right\Vert _{\infty}}-1$ geometrically. This corresponds to the behavior illustrated in the light blue region in Figure \ref{curvfig}.

%\jw{we need to make this section teach / explain, somewhere we need to tie this lemma back to the picture in figure 3, talk about both curvature and gradient, which leads to the rate}

\begin{lemma}[Separable objective gradient projection] \label{neggradlem}
	%The gradient of the separable objective \ref{sepeq} obeys
	For any $\mb w\in \mathcal{C}_{\zeta},i \in [n-1]$, we define a vector $\mb u^{(i)} \in T_{\mb q(\mb w)}\mathbb{S}^{n-1}$ by
	\begin{equation} \lbl{ueq}
	u_{j}^{(i)}=\begin{cases} 0 & j \notin \{i,n\}, \\ \mr{sign}(w_{i}) & j=i, \\ -\frac{\left\vert w_{i}\right\vert }{q_{n}} & j=n. \end{cases}
	\end{equation}
	If $\mu \log\left(\frac{1}{\mu}\right) \leq w_{i}$ and $\mu<\frac{1}{16}$, then
	\[
	\mathbf{u}^{(i)\ast}\mathrm{grad}[\fs ](\mathbf{q}(\mathbf{w})) \, \geq \, c \left\Vert \mb w\right\Vert_{\infty}\zeta,
	\]
	where $c>0$ is a numerical constant.
\end{lemma}
\begin{proof}
	Please see Appendix \ref{neggradlemp}.
\end{proof}

%\begin{lemma}[$\mathcal{C}_{\zeta}$ is absorbing] \label{sepczlem}
%
%For any any $\mb q(\mb w) \in \mathcal{C}_{\zeta},\zeta > \eta$ and $\eta<\frac{1}{L}$ we have
%\[
%\mb q(\mb w-\eta\nabla g_s(\mb w)) \in \mathcal{C}_{\zeta}
%\]
%\end{lemma}

Since we will use this property of the gradient in $ \mathcal{C}_{\zeta}$ to derive a convergence rate, we will be interested in bounding the probability that gradient descent initialized randomly with respect to a uniform measure on the sphere is initialized in $\mathcal{C}_{\zeta}$. This will require bounding the volume of this set, which is done in the following lemma:

\begin{lemma}[Volume of $\mathcal{C}_{\zeta}$] \label{vollem}
	For $\mathcal{C}_{\zeta}$ defined as in \eqref{czeq} we have 
	\[
	\frac{\mathrm{Vol}(\mathcal{C}_{\zeta})}{\mathrm{Vol}(\mathbb{S}^{n-1})} \geq \frac{1}{2n} -\frac{\log(n)}{n} \zeta
	\]
\end{lemma}
\begin{proof}
	Please see Appendix \ref{vollemp}.
\end{proof}
%Using Lemma \ref{vollem}, if we denote the probability of initializing in $Q_\varepsilon$ by $\mathbb{P}_Q$, we are guaranteed that 
%\begin{equation} \label{pqeq}
%\mathbb{P}_Q=\frac{\mathrm{Vol}(Q_\varepsilon)}{\mathrm{Vol}(\mathbb{S}^{n-1})}\leq1-2n\frac{\mathrm{Vol}(\mathcal{C}_{\zeta})}{\mathrm{Vol}(\mathbb{S}^{n-1})}\leq 2 \log(n)\zeta
%\end{equation}
\subsection{Convergence rate}
Using the results above, one can obtain the following convergence rate:

\begin{theorem}[Gradient descent convergence rate for separable function] \label{newratelem}
	For any $0 < \zeta_0 < 1$, $r > \mu \log\left(\frac{1}{\mu} \right)%\max(\eta, \log(\frac{1}{\mu})\mu)
	$, Riemannian gradient descent with step size  $\eta < \min\left\{ \frac{c_1}{n},\frac{\mu}{2}\right\}$ on the separable objective \eqref{sepeq} with $\mu<\frac{c_2}{\sqrt{n} \log{n}}$, enters an $L^\infty$ ball of radius $r$ around a global minimizer in 
	\[
	T<\frac{C}{\eta}\left(\frac{\sqrt{n}}{r^{2}}+\log\left(\frac{1}{\zeta_0}\right)\right)
	%T<\frac{16}{\sqrt{n}\eta c^{2}r^{2}}+\frac{\left(1-\frac{\zeta^{2}}{(n+3)r^{2}}\right)\log(\frac{1}{\zeta})}{\log(1+\frac{c}{2}\sqrt{\frac{n}{n+3}}\eta)}
	\]
	iterations with probability 
	\[
	\mathbb{P} \geq 1-2 \log(n) \zeta_0,
	\]
	where $c_i,C>0$ are numerical constants.% and $s(\xi)=\frac{1}{\sqrt{(2+\xi)\xi+n}}$ 
	%\jw{stylistically, prefer to reserve $c$ for small numerical constants, and $C$ for large numerical constants. Prefer to state a result with vague constants, but a clear dependence on dimension and other leading parameters. E.g., I would upper bound your $c$ with $c' n^{-1/2}$, where $c'$ is numerical. I would then change the first term in your upper bound on $T$ from 
	%$$ \frac{2 \sqrt{n}}{\eta c^2 r^2} $$ to $$ \frac{C n^{3/2}}{\eta r^2}$$ where $C$ is a numerical constant. I believe this brings out the dependencies more clearly. Similar comments might obtain for the second term -- I haven't thought about it clearly.}
\end{theorem}

\begin{proof} Please see Appendix \ref{newratelemp}. \end{proof}

%\jw{what is the radius of strong convexity about a minimizer for this problem? want to understand what is the lower bound on reasonable choices of $r$.} 

%\begin{theorem}[Gradient descent convergence rate for separable function]
%\label{septhm}
%For any $r,\zeta > 0$, gradient descent on the separable objective \ref{sepeq} with $\mu < \frac{1}{\sqrt{n}}$ converges to an $L_\infty$ ball of radius $r$ of a global minimizer within 
%\[
%n_{\text{iter}}< \frac{8 n^{3/2}}{\mu \left[\frac{r}{\mu}\mathrm{sech}^{2}\left(\frac{r}{\mu}\right)-\mathrm{tanh}\left(\frac{r}{\mu}\right)\right]^2\zeta^2}
%\]
%iterations with probability 
%\[
%\mathbb{P} \gtrsim 1-4 \log(n) \zeta
%\]
%\end{theorem}

We have thus obtained a convergence rate for gradient descent that relies on the negative curvature around the stable manifolds of the saddles to rapidly move from these regions of the space towards the vicinity of a global minimizer. This is evinced by the logarithmic dependence of the rate on $\zeta$.
As was shown for orthogonal dictionary learning in \cite{sun2017complete}, we also expect a linear convergence rate due to strong convexity in the neighborhood of a minimizer, but do not take this into account in the current analysis.

\section{Dictionary Learning Convergence Rate} \label{dlsec}

The proofs in this section will be along the same lines as those of Section \ref{secsep}. While we will not describe the positions of the critical points explicitly, the similarity between this objective and the separable function motivates a similar argument. It will be shown that initialization in some $\mathcal{C}_\zeta$ will guarantee that Riemannian gradient descent makes uniform progress in function value until reaching the neighborhood of a global minimizer. We will first consider the population objective which corresponds to the infinite data limit %(the $p \rightarrow \infty$ limit). 
%\jw{explain...}  
\begin{equation} \label{dleq}
f_{\mathrm{DL}}^{\mathrm{pop}}(\mb q)\equiv\underset{\mathbf{X}_{0}}{\mathbb{E}}f_{\mathrm{DL}}(\mb q)={\mathbb{E}_{\mb x\sim_{\mathrm{i.i.d.}}\mathrm{BG}(\theta)}}\bigl[\,h_{\mu}(\mb q^{*}\mb x)\,\bigr].%,\mbq\in\mathbb{S}^{n-1}
\end{equation}
and then bounding the finite sample size fluctuations of the relevant quantities. We begin with a lemma analogous to Lemma \ref{neggradlem}:

%\begin{lemma}[Dictionary learning gradient projection bound] \label{dlgradplem}
%For any $\mb q(\mb w)\in C_\zeta$ with $Q_\zeta$ defined in \ref{czeq} and $u$ defined in \ref{ueq} we have 
%
%\[
%\overline{u}^{T}\nabla g_{DL}^{pop}(\mb w)
%\]
%\begin{equation} \label{epseq}
%\leq\sqrt{\frac{2}{\pi}}\theta(1-\theta)\left(-2\left\Vert \mb w\right\Vert _{\infty}^{3}\zeta+\frac{\mu^{2}}{\left\Vert \mb w\right\Vert _{\infty}^{2}}+\frac{3\mu^{4}}{2\left\Vert \mb w\right\Vert _{\infty}^{4}}\right)
%\end{equation}
%
%\end{lemma}
%
%
%Proof: \ref{dlgradppf}

%It follows that for $\left\Vert \mb w\right\Vert _{\infty}>r$
%and $\zeta>\frac{\mu^{2}}{2r^{5}}+\frac{3}{4}\frac{\mu^{4}}{r^{7}}$
%($\mu$ can be chosen sufficiently small to ensure this for any $\zeta > 0$) we have  
%
%\[
%\overline{u}^{T}\nabla g_{DL}^{pop}(\mb w)<0
%\]
%
%By an argument similar to that of Lemma \ref{sepczlem} it is clear that gradient descent initialized in $C_\zeta$ will converge to $B_{r}^{\infty}(\mb w=0)$.  

\begin{lemma}[Dictionary learning population gradient] \label{neggraddllem}
	For $\mb w\in \mathcal{C}_{\zeta}, r < |w_{i}|,\mu<\cdlmu r^{5/2}\sqrt{\zeta}$ the dictionary learning population objective \ref{dleq} obeys
	\[
	\mathbf{u}^{(i)\ast}\mathrm{grad}[f_{\mathrm{DL}}^{\mathrm{pop}}](\mathbf{q}(\mathbf{w})) \geq \cdl r^3\zeta
	\]
	where $\cdl$ depends only on $\theta$, $\cdlmu$ is a positive numerical constant and $\mb u^{(i)}$ is defined in \ref{ueq}. %$\cdl\equiv \frac{\theta(1-\theta)}{\sqrt{2\pi}}$
\end{lemma}
\begin{proof}
	Please see Appendix \ref{neggraddllemp}
\end{proof}
Using this result, we obtain the desired convergence rate for the population objective, presented in Lemma \ref{newratedllem} in Appendix \ref{dlapp}. After accounting for finite sample size fluctuations in the gradient, one obtains a rate of convergence to the neighborhood of a solution (which is some signed basis vector due to our choice $\mb A_0 = \mb I$)
%\begin{theorem}[Gradient descent convergence rate for dictionary learning]
%\label{dlthm}
%For any $r,\zeta > 0$, we choose $\mu$ such that $\zeta\geq\frac{\mu^{2}}{2r^{5}}+\frac{3\mu^{4}}{4r^{7}}$ and $t>0$ such that $t<\sqrt{\frac{2}{\pi}}\theta(1-\theta)\left\vert -2r^{3}\zeta+\frac{\mu^{2}}{r^{2}}+\frac{3\mu^{4}}{2r^{4}}\right\vert $. Gradient descent on the dictionary learning objective \ref{dl2eq} with converges to an $L_\infty$ ball of radius $r$ of a global minimizer within 
%\[
%n_{\text{iter}}<\frac{4n^{3/2}\left\Vert \mb X\right\Vert _{\infty}^{2}\left(\frac{\left\Vert \mb X\right\Vert _{\infty}}{\mu}+1\right)}{\varepsilon^{2}(\zeta)}
%\]
%iterations with probability 
%\[
%\mathbb{P} \gtrsim \left(1-4 \log(n) \zeta\right)\left(1-2\exp\left(-\frac{pt^{2}}{4+2\sqrt{2}t}\right)\right)
%\]
%where lower order terms in $n$ are dropped, and
%\begin{equation} \label{epseq}
%\varepsilon(\zeta)\equiv\sqrt{\frac{2}{\pi}}\theta(1-\theta)\left(-2r^{3}\zeta+\frac{\mu^{2}}{r^{2}}+\frac{3\mu^{4}}{2r^{4}}\right)-t
%\end{equation}

%\end{theorem}

%We prove convergence to $B_{\mu/4\sqrt{2}}(0)$ since in \cite{sun2017complete} it was shown that within this set the objective is strongly convex and a linear convergence rate is expected. The argument can be modified to prove convergence to a different neighborhood.  

\begin{theorem}[Gradient descent convergence rate for dictionary learning] \label{dlratethm}
	For any $1 > \zeta_0 > 0,s > \frac{\mu}{4\sqrt{2}}$, Riemannian gradient descent with step size $\eta < \frac{\cetadl \theta s}{n \log{np}}$
	%	\[
	%	\eta<\min\left(\begin{array}{c}
	%	\frac{1}{24\sqrt{(n+3)n\log(np)}},\frac{1}{1440\sqrt{5n(n-1)\log(np)}},\\
	%	\frac{c_{3}\theta s}{n \log(np)}
	%	%\frac{\theta s}{160\sqrt{\pi n\log(np)}},\frac{5\sqrt{5}c_{\star}\theta s}{64n\log(np)}
	%	\end{array}\right)
	%	%\eta<\min\left(\begin{array}{c}
	%	%\frac{1}{6\sqrt{n^{2}+3n}\left\Vert \mathbf{X}\right\Vert _{\infty}},\frac{1}{360\sqrt{5n(n-1)}\left\Vert \mathbf{X}\right\Vert _{\infty}},\\
	%	%\frac{\theta s}{40\sqrt{\pi n}\left\Vert \mathbf{X}\right\Vert _{\infty}},\frac{5\sqrt{5}c_{\star}\theta s}{4n\left\Vert \mathbf{X}\right\Vert _{\infty}^{2}}
	%	%\end{array}\right)
	%	\]
	on the dictionary learning objective \ref{dl2eq} with $\mu < \frac{\cmudl \sqrt{\zeta_0}}{n^{5/4}},\theta \in (0,\frac{1}{2})$%$\mu< \min \left(\frac{\sqrt{\zeta_0}}{(n+3)^{5/4}8\sqrt{6}}, \frac{1}{\sqrt{12}(8000(n-1))^{5/4}}\right),\theta<\frac{1}{2}$
	, enters a ball of radius $\cdistfrommin s$ from a target solution in 
	
	\[
	T<\frac{ \Cdl}{\eta\theta}\left(\frac{1}{s} + n\log\frac{1}{\zeta_{0}}\right)
	%	T<\begin{array}{c}
	%	\frac{\log\frac{1}{\zeta_{0}}}{\log\left(1+\frac{\sqrt{n}\cdl}{32(n+3)^{3/2}}\eta\right)}+\frac{\log(8)}{\log\left(1+\frac{\sqrt{n}\cdl}{4(8000(n-1))^{3/2}}\eta\right)}\\
	%	+\frac{C_{3}}{s \theta \eta}
	%	\end{array}
	\]
	iterations with probability 
	\[
	\mathbb{P} \geq 1- 2 \log(n) \zeta_0 - \mathbb{P}_{y} - \cb p^{-6}
	\]
	%z(\theta,\zeta_{0})=\frac{\cz \theta(1-\theta)\zeta_{0}}{n^{3/2}}$, 
	where $y=\frac{\cz \theta(1-\theta)\zeta_{0}}{n^{3/2}}$, $\mathbb{P}_{y}$ is given in Lemma \ref{uniflem} and $c_i,C_i$ are positive constants.
\end{theorem}
\begin{proof}
	Please see Appendix \ref{dlratethmp}
\end{proof}
The two terms in the rate correspond to an initial geometric increase in the distance from the set containing the small gradient regions around saddle points, followed by convergence to the vicinity of a minimizer in a region where the gradient norm is large. The latter is based on results on the geometry of this objective provided in \cite{sun2017complete}.
% \input{gpr.tex}
% !TEX root = dispersive-full.tex

\section{Discussion}

The above analysis suggests that second-order properties - namely negative curvature normal to the stable manifolds of saddle points - play an important role in the success of randomly initialized gradient descent in the solution of complete orthogonal dictionary learning. This was done by furnishing a convergence rate guarantee that holds when the random initialization is not in regions that feed into small gradient regions around saddle points, and bounding the probability of such an initialization. In Appendix \ref{appgpr} we provide an additional example of a nonconvex problem that for which an efficient rate can be obtained based on an analysis that relies on negative curvature normal to stable manifolds of saddles - generalized phase retrieval. An interesting direction of further work is to more precisely characterize the class of functions that share this feature. 

The effect of curvature can be seen in the dependence of the maximal number of iterations $T$ on the parameter $\zeta_0$. This parameter controlled the volume of regions where initialization would lead to slow progress and the failure probability of the bound $1-\mathbb{P}$ was linear in $\zeta_0$, while $T$ depended logarithmically on $\zeta_0$. This logarithmic dependence is due to a geometric increase in the distance from the stable manifolds of the saddles during gradient descent, which is a consequence of negative curvature. Note that the choice of $\zeta_0$ allows one to flexibly trade off between $T$ and $1-\mathbb{P}$. By decreasing $\zeta_0$, the bound holds with higher probability, at the price of an increase in $T$. This is because the volume of acceptable initializations now contains regions of smaller minimal gradient norm. In a sense, the result is an extrapolation of works such as \cite{lee2017first} that analyze the $\zeta_0=0$ case to finite $\zeta_0$. 

Our analysis uses precise knowledge of the location of the stable manifolds of saddle points. 
%A shortcoming of this work is that the analysis requires (at least approximate) knowledge of the location of the stable manifolds of saddle points. 
For less symmetric problems, including variants of sparse blind deconvolution \cite{zhang2017global} and overcomplete tensor decomposition, there is no closed form expression for the stable manifolds. However, it is still possible to coarsely localize them in regions containing negative curvature. Understanding the implications of this geometric structure for randomly initialized first-order methods is an important direction for future work. 
%Extending our analysis to less-structured scenarios is an important direction for future work.
%
%This information is in general difficult to obtain for less structured problems. An interesting direction of further work will be to attempt to extend the analysis to problems where the form of the stable manifolds is not known, yet in certain regions of the space that contain them there exists a direction of negative curvature. One such problem is sphere-constrained sparse blind deconvolution in certain settings \cite{zhang2017global}. 

%It is likely that these results can be extended to account for finite sample size fluctuations, and this is a natural direction of further work. It will also be interesting to consider if there are other strict saddle problems that have the dispersive property. 
%An interesting direction of future work is to attempt a similar analysis of less structured problems \dg{add examples?}. %A first stage could be problems with no local minima where the global minima are in some sense incoherent but not as symmetrical as in the problems considered here. 
%Other natural direct extensions to consider are the cases of complete dictionary learning, as well as higher rank matrix recovery or matrix completion problems \cite{recht2011simpler}. 

One may hope that studying simple model problems and identifying structures (here, negative curvature orthogonal to the stable manifold) that enable efficient optimization will inspire approaches to broader classes of problems. 
%A longer term goal is to identify structures that enable efficient methods to obtain minimizers. 
One problem of obvious interest is the training of deep neural networks for classification, which shares certain high-level features with the problems discussed in this paper. The objective is also highly nonconvex and is conjectured to contain a proliferation of saddle points \cite{dauphin2014identifying}, yet these appear to be avoided by first-order methods \cite{goodfellow2014qualitatively} for reasons that are still quite poorly understood beyond the two-layer case \cite{venturi2018neural}. %While it is clear t00000hat the dispersive property does not hold in general for such problems there is evidence that in certain simplified settings curvature may play a similar role and thus help explain the success of gradient descent.  

% !TEX root = dispersive-full.tex
\clearpage
\newpage 
\appendix

\section{Proofs - Separable Objective}

\begin{proof}[\begin{minipage}{3.25in}\bf{Proof of Lemma \ref{seplemma}: (Critical point structure of separable objective)}\end{minipage}] \label{Rpf}
	Denoting by $\tanh(\frac{\mathbf{q}}{\mu})$ a vector in $\mathbb{R}^{n}$
	elements $\tanh(\frac{\mathbf{q}}{\mu})_{i}=\tanh(\frac{q_{i}}{\mu})$
	we have
	
	\[
	\mathrm{grad}[f_{Sep}](\mathbf{q})_{i}=(\mathbf{I}-\mathbf{q}\mathbf{q}^{\ast})\tanh(\frac{\mathbf{q}}{\mu})
	\]
	
	. Thus critical points are ones where either $\tanh(\frac{\mathbf{q}}{\mu})=\mathbf{0}$
	(which cannot happen on $\mathbb{S}^{n-1}$) or $\tanh(\frac{\mathbf{q}}{\mu})$
	is in the nullspace of $(\mathbf{I}-\mathbf{q}\mathbf{q}^{\ast})$,
	which implies $\tanh(\frac{\mathbf{q}}{\mu})=c\mathbf{q}$ for some
	constant $b$. The equation $\tanh(\frac{x}{\mu})=bx$ has either
	a single solution at the origin or 3 solutions at $\{0,\pm r(b)\}$ for
	some $r(b)$. Since this equation must be solves simultaneously for every
	element of $\mathbf{q}$, we obtain $\forall i\in[n]:\text{ }q_{i}\in\{0,\pm r(b)\}$.
	To obtain solutions on the sphere, one then uses the freedom we have
	in choosing $b$ (and thus $r(b)$) such that $\left\Vert \mathbf{q}\right\Vert =1$.
	The resulting set of critical points is thus 
	
	\[
	A=\mathcal{P}_{\mathbb{S}^{n-1}}\Bigl[\,\{-1,0,1\}^{n}\setminus\set{\mb0}\,\Bigr].
	\]
	
	To prove the form of the stable manifolds, we first show that for
	$q_{i}$ such that $\left\vert q_{i}\right\vert =\left\Vert \mathbf{q}\right\Vert _{\infty}$
	and any $q_{j}$ such that $\left\vert q_{j}\right\vert +\Delta=\left\vert q_{i}\right\vert $
	and sufficiently small $\Delta>0$, we have 
	
	\begin{equation} \lbl{sepgradeq}
	-\mathrm{grad}[f_{Sep}](\mathbf{q})_{i}\mathrm{sign}(q_{i})>-\mathrm{grad}[f_{Sep}](\mathbf{q})_{i}\mathrm{sign}(q_{j})
	\end{equation}
	For ease of notation we now assume $q_{i},q_{j}>0$ and hence $\Delta=q_{i}-q_{j}$,
	otherwise the argument can be repeated exactly with absolute values
	instead. The above inequality can then be written as 
	
	\[
	\underbrace{(q_{i}-q_{j})\underset{k=1}{\overset{n}{\sum}}\tanh(\frac{q_{k}}{\mu})q_{k}-\tanh(\frac{q_{i}}{\mu})+\tanh(\frac{q_{j}}{\mu})}_{\equiv h}>0.
	\]

	If we now define $s^{2}=\underset{\begin{array}{c}
		k=1\\
		k\neq i,n
		\end{array}}{\overset{n-1}{\sum}}q_{k}^{2}$ and $q_{n}=\sqrt{1-s^{2}-(q_{j}+\Delta)^{2}}$we have
	
	\[
	h=\begin{array}{c}
	\Delta\left(\begin{array}{c}
	\tanh(\frac{q_{j}+\Delta}{\mu})\left(q_{j}+\Delta\right)+\\
	\tanh(\frac{\sqrt{1-s^{2}-(q_{j}+\Delta)^{2}}}{\mu})\sqrt{1-s^{2}-(q_{j}+\Delta)^{2}}
	\end{array}\right)\\
	+\Delta\underset{k\neq i,n}{\sum}\tanh(\frac{q_{k}}{\mu})q_{k}-\tanh(\frac{q_{j}+\Delta}{\mu})+\tanh(\frac{q_{j}}{\mu})
	\end{array}
	\]
	
	\[
	=\Delta\left(\begin{array}{c}
	\underbrace{\begin{array}{c}
		\underset{k\neq i,n}{\sum}\tanh(\frac{q_{k}}{\mu})q_{k}+\tanh(\frac{q_{j}}{\mu})q_{j}\\
		+\tanh(\frac{\sqrt{1-s^{2}-q_{j}^{2}}}{\mu})\sqrt{1-s^{2}-q_{j}^{2}}
		\end{array}}_{\equiv h_{1}}\\
	-\underbrace{\mathrm{sech}^{2}(\frac{q_{j}}{\mu})\frac{1}{\mu}}_{\equiv h_{2}}
	\end{array}\right)+O(\Delta^{2})
	\]
	
	where the $O(\Delta^{2})$ term is bounded. Defining a vector $\mathbf{r}\in\mathbb{R}^{n}$
	by 
	
	\[
	k\neq i,n:r_{k}=q_{k},r_{i}=\tanh(\frac{q_{j}}{\mu})q_{j},r_{n}=\sqrt{1-s^{2}-q_{j}^{2}}
	\]
	
	we have $\left\Vert \mathbf{r}\right\Vert ^{2}=1$. Since $\tanh(x)$
	is concave for $x>0$, and $\left\vert r_{i}\right\vert \leq1$, we
	find
	
	\[
	h_{1}=\underset{k=1}{\overset{n}{\sum}}\tanh(\frac{r_{k}}{\mu})r_{k}\geq\tanh(\frac{1}{\mu})\underset{k=1}{\overset{n}{\sum}}r_{k}^{2}=\tanh(\frac{1}{\mu}).
	\]
	
	From $\left\vert q_{i}\right\vert =\left\Vert \mathbf{q}\right\Vert _{\infty}$
	it follows that $q_{i}\geq\frac{1}{\sqrt{n}}$ and thus $q_{j}\geq\frac{1}{\sqrt{n}}-\Delta$.
	Using this inequality and properties of the hyperbolic secant we obtain
	
	\[
	h_{2}\leq4\exp(-2\frac{q_{j}}{\mu}-\log\mu)\leq\exp(\frac{2\Delta}{\mu}-\frac{2}{\mu\sqrt{n}}-\log\mu+\log4)
	\]
	
	and plugging in $\mu=\frac{c}{\sqrt{n}\log n}$ for some $c<1$
	
	\[
	\leq\exp(\frac{2\Delta}{\mu}-\frac{2\log n}{c}-\log c+\frac{1}{2}\log n+\log\log n+\log4).
	\]
	
	We can bound this quantity by a constant, say $h_{2}\leq\frac{1}{2}$,
	by requiring 
	
	\[
	A\equiv\frac{2\Delta}{\mu}-\log c+(\frac{1}{2}-\frac{2}{c})\log n+\log\log n\leq-\log8
	\]
	
	and for and $c<1$, using $-\log n+\log\log n<0$ we have 
	
	\[
	A<\frac{2\Delta}{\mu}-\log c-(\frac{2}{c}-1)\log n.
	\]
	
	Since $\Delta$ can be taken arbitrarily small, it is clear that $c$
	can be chosen in an $n$-independent manner such that $A\leq-\log8$.
	We then find
	
	\[
	h_{1}-h_{2}\geq\tanh(\frac{1}{\mu})-\frac{1}{2}\geq\tanh(\sqrt{n}\log n)-\frac{1}{2}>0
	\]
	
	since this inequality is strict, $\Delta$ can be chosen small enough
	such that $\left\vert O(\Delta^{2})\right\vert <\Delta (h_{1}-h_{2})$ and
	hence 
	
	\[
	h>0,
	\]
	
	proving \ref{sepgradeq}. 
	
	It follows that under negative gradient flow, a point
	with $|q_{j}|< ||\mb q ||_\infty$ cannot flow to a point $\mb q'$ such that $|q'_{j}|= ||\mb q' ||_\infty$. From the form of the critical points, for every such $j$, $\mb q$ must thus flow to a point such that $q'_{j}=0$ (the value of the $j$ coordinate cannot pass through 0 to a point where $|q'_{j}|= ||\mb q' ||_\infty$ since from smoothness of the objective this would require passing some $\mb q''$  with $q''_j=0$, at which point $\grad{\fs}{\mb q''}_j=0$).
	
	As for the maximal magnitude coordinates, if there is more than one coordinate satisfying
	$\left\vert q_{i_{1}}\right\vert =\left\vert q_{i_{2}}\right\vert =\left\Vert \mathbf{q}\right\Vert _{\infty}$,
	it is clear from symmetry that at any subsequent point $\mathbf{q}'$
	along the gradient flow line $\left\vert q'_{i_{1}}\right\vert =\left\vert q'_{i_{2}}\right\vert $.
	These coordinates cannot change sign since from the smoothness of the objective this would require that they pass through a point where they have magnitude smaller than $1/\sqrt{n}$, at which point some other coordinate must have a larger magnitude (in order not to violate the spherical constraint), contradicting the above result for non-maximal elements. It follows that the sign pattern of these elements is
	preserved during the flow. Thus there is a single critical point to
	which any $\mathbf{q}$ can flow, and this is given by setting all
	the coordinates with $\left\vert q_{j}\right\vert <\left\Vert \mathbf{q}\right\Vert _{\infty}$
	to 0 and multiplying the remaining coordinates by a positive constant
	to ensure the resulting vector is on $\mathbb{S}^{n}$. Denoting this
	critical point by $\mb \alpha$, there is a vector $\mb b$ such that $\mathbf{q}=\mathcal{P}_{\mathbb{S}^{n-1}}\left[\mb a(\mb \alpha)+\mb b\right]$
	and $\mathrm{supp}(\mb a(\mb \alpha))\cap\mathrm{supp}(\mb b)=\varnothing$,
	$\left\Vert \mb b\right\Vert _{\infty}<1$ with the form of $\mb a(\mb \alpha)$
	given by \ref{Aeq} . The collection of all such points defines the stable
	manifold of $\mb \alpha$. 	
\end{proof}

\begin{proof}[\bf{Proof of Lemma \ref{neggradlem}}: (Separable objective gradient projection)] \label{neggradlemp}
	
	i) We consider the $\mathrm{sign}( w_{i})=1$ case; the $\mathrm{sign}( w_{i})=-1$
	case follows directly. Recalling that $\mathbf{u}^{(i)\ast}\mathrm{grad}[\fs](\mathbf{q}(\mathbf{w}))=\tanh\left(\frac{w_{i}}{\mu}\right)-\tanh\left(\frac{q_{n}}{\mu}\right)\frac{w_{i}}{q_{n}}$, we first prove 
	
	\begin{equation} \lbl{gradqweq}
	\tanh\left(\frac{w_i}{\mu}\right)-\tanh\left(\frac{q_{n}}{\mu}\right)\frac{w_i}{q_{n}}\geq c(q_{n}-w_i)
	\end{equation}
	
	for some $c>0$ whose form will be determined later. The inequality clearly holds for $ w_i=q_{n}$. To verify that it holds for smaller values of $w_i$ as well, we now show that 
	
	\[
	\frac{\partial}{\partial w_i}\left[\tanh\left(\frac{w_i}{\mu}\right)-\tanh\left(\frac{q_{n}}{\mu}\right)\frac{w_i}{q_{n}}-c(q_{n}-w_i)\right]  < 0
	\]
	
	which will ensure that it holds for all $w_i$. We define $s^2=1-|| \mb w ||^2 + w_i^2$ and denote $q_{n}=\sqrt{s^{2}-w_i^{2}}$
	to extract the $w_i$ dependence, giving
	
	\[
	\frac{\partial}{\partial w_i}\left[\tanh\left(\frac{w_i}{\mu}\right)-\tanh\left(\frac{q_{n}}{\mu}\right)\frac{w_i}{q_{n}}-c(q_{n}-w_i)\right]
	\]
	
	\[
	=\begin{array}{c}
	\frac{1}{\mu}\mathrm{sech}^{2}\left(\frac{w_i}{\mu}\right)+\frac{1}{\mu}\mathrm{sech}^{2}\left(\frac{\sqrt{s^{2}-w_i^{2}}}{\mu}\right)\frac{w_i^{2}}{s^{2}-w_i^{2}}\\
	-\tanh\left(\frac{\sqrt{s^{2}-w_i^{2}}}{\mu}\right)\frac{s^{2}}{(s^{2}-w_i^{2})^{3/2}}+c(\frac{w_i}{\sqrt{s^{2}-w_i^{2}}}+1)
	\end{array}
	\]
	
	\[
	\leq\begin{array}{c}
	\frac{4}{\mu}\left(e^{-2\frac{w_i}{\mu}}+e^{-2\frac{\sqrt{s^{2}-w_i^{2}}}{\mu}}\right)\\
	-\tanh\left(\frac{\sqrt{s^{2}-w_i^{2}}}{\mu}\right)\frac{s^{2}}{(s^{2}-w_i^{2})^{3/2}}+2c
	\end{array}	
	\]
	
	Where in the last inequality we used properties of the $\mathrm{sech}$ function and $q_n \geq w_i$. We thus want to show 
	
	\[
	\frac{4}{\mu}\left(e^{-2\frac{w_i}{\mu}}+e^{-2\frac{q_{n}}{\mu}}\right)+2c\leq\tanh\left(\frac{q_{n}}{\mu}\right)\frac{q_{n}^{2}+w_i^{2}}{q_{n}^{3}}
	\]
	
	and using $\log(\frac{1}{\mu})\mu\leq w_i\leq q_{n}$ and $c = \frac{\frac{1-\mu^{2}}{1+\mu^{2}}-8\mu}{2}$ we have 
	
	\[
	\frac{4}{\mu}\left(e^{-2\frac{w_i}{\mu}}+e^{-2\frac{q_{n}}{\mu}}\right)+2c
	\]
	\[
	\leq\frac{8e^{-2\frac{w_i}{\mu}}}{\mu}+2c\leq8\mu+2c\leq\frac{1-\mu^{2}}{1+\mu^{2}}
	\]
	
	\[
	=\tanh\left(\log(\frac{1}{\mu})\right)\leq\tanh\left(\frac{q_{n}}{\mu}\right)\frac{1}{q_{n}} 
	\]
	\[
	< \tanh\left(\frac{q_{n}}{\mu}\right)\frac{q_{n}^{2}+w_i^{2}}{q_{n}^{3}}
	\]
	and it follows that \ref{gradqweq} holds. For $\mu < \frac{1}{16}$ we are guaranteed that $c>0$. 
	
	From examining the RHS of \ref{gradqweq} (and plugging in $q_{n}=\sqrt{s^{2}-w_i^{2}}$) we see that any lower bound on the gradient
	of an element $w_{j}$ applies also to any element $\left\vert w_{i}\right\vert \leq \left\vert w_{j}\right\vert $. Since for $|w_j| = || \mb w ||_\infty$ we have $q_{n}-w_j = w_j \zeta$, for every $\log(\frac{1}{\mu})\mu\leq w_i$ we obtain the bound
	
	\[
	\mathbf{u}^{(i)\ast}\mathrm{grad}[\fs](\mathbf{q}(\mathbf{w})) \geq c\left\Vert \mb w\right\Vert _{\infty}\zeta
	\]

	%. Denoting by $\overline{\mathbf{u}}$ the restriction of $\mathbf{u}$ to the first $n-1$ elements, it is easy to check that $\mb u^{(i)}(\mathbf{q})\in T_{\mathbf{q}}\mathbb{S}^{n-1}$ and 
	%\[
	%\mathbf{u}^{(i)\ast}\mathrm{grad}[f](\mathbf{q}(\mathbf{w}))=\overline{\mathbf{u}}^{(i)\ast}\mathbf{w}=\mathrm{sign}(w_{i})\nabla_{\mathbf{w}}g_{s}(\mathbf{w})_{i}
	%\]
	%
	%hence we have shown that 
	%\[
	%\mathbf{u}^{(i)\ast}\mathrm{grad}[f](\mathbf{q}(\mathbf{w})) > c\left\Vert \mb w\right\Vert _{\infty}\zeta
	%\]
	%ii) Using the notation of the previous section, and recalling that $0\leq w_i\leq\frac{s}{\sqrt{2}}$, solutions to 
	%\[
	%\mathrm{sign}(w_{i})\nabla_{\mb w}g_{s}(\mb w)_{i}
	%\]
	%\[
	%=\tanh\left(\frac{w_i}{\mu}\right)-\tanh\left(\frac{\sqrt{s^{2}-w_i^{2}}}{\mu}\right)\frac{w_i}{\sqrt{s^{2}-w_i^{2}}}=0
	%\]
	%are given by $ w_i=0$ or $ w_i=\frac{s}{\sqrt{2}}$. This can be seen by checking the $ w_i=0$ case and assuming $ w_i \neq 0$, writing 
	%\[
	%\frac{\tanh\left(\frac{w_i}{\mu}\right)}{w_i}=\frac{\tanh\left(\frac{\sqrt{s^{2}-w_i^{2}}}{\mu}\right)}{\sqrt{s^{2}-w_i^{2}}}
	%\]
	%and using the injectivity of $\frac{\tanh\left(x\right)}{x}$ for $x \geq 0$ which is easy to verify. The previous section guarantees $\mathrm{sign}(w_{i})\nabla_{\mb w}g_{s}(\mb w)_{i}>0$ for some intermediate values of $w_i$. Since we have shown that this quantity cannot change sign we are guaranteed  that the result holds for all intermediate values.
\end{proof}

%\begin{proof}[\begin{minipage}{3.25in}\bf{Proof of Theorem \ref{newratelem}}: (Gradient descent convergence rate for separable function)\end{minipage}] \label{newratelemp}
\begin{proof}[\bf{Proof of Theorem \ref{newratelem}}: (Gradient descent convergence rate for separable function)] \label{newratelemp}

	We obtain a convergence rate by first bounding the number of iterations of Riemannian gradient descent in $\mathcal{C}_{\zeta_0} \backslash \mathcal{C}_1$, and then considering $\mathcal{C}_1 \backslash B_r^\infty$.
	
	From Lemma \ref{czlemma} we obtain $\mathcal{C}_{\zeta_0} \backslash \mathcal{C}_1 \subseteq \mathcal{C}_{\zeta_0} \backslash B^\infty_{1/\sqrt{n+3}}$. Choosing $c_2$ so that $\mu < \frac{1}{2}$, we can apply Lemma \ref{neggradlem}, and for $\mb u$ defined in \ref{ueq}, we thus have 
	\[
	|w_i| > \mu \log(\frac{1}{\mu}) \Rightarrow \mathbf{u}^{(i)\ast}\mathrm{grad}[\fs](\mathbf{q}(\mathbf{w})) > c || \mb w ||_{\infty}  \zeta_0.
	\]
	Since from Lemma \ref{gradupperlemsep} the Riemannian gradient norm is bounded by $\sqrt{n}$, we can choose $c_1,c_2$ such that  $\mu\log(\frac{1}{\mu})<\frac{1}{2\sqrt{n+3}},\eta<\frac{1}{6\sqrt{n^{2}+3n}}$. This choice of $\eta$ then satisfies the conditions of Lemma \ref{zetalem} with $r=\mu\log(\frac{1}{\mu}), b=\frac{1}{\sqrt{n+3}},M=\sqrt{n}$, which gives that after a gradient step 
	\begin{equation} \lbl{zchangeq}
	\zeta'\geq\zeta\left(1+\frac{c}{2}\sqrt{\frac{n}{n+3}}\eta\right)\geq\zeta\left(1+\tilde{c}\eta\right)
	\end{equation}
	
	for some suitably chosen $\tilde{c}>0$. If we now define by $\mathbf{w}^{(t)}$ the $t$-th iterate of Riemannian gradient descent and $\zeta^{(t)}\equiv\frac{q_{n}^{(t)}}{\left\Vert \mathbf{w}^{(t)}\right\Vert _{\infty}}-1,\zeta^{(0)} \equiv \zeta_0$, for iterations such that $\mathbf{w}^{(t)} \in \mathcal{C}_{\zeta} \backslash \mathcal{C}_1$ we find 
	\[
	\zeta^{(t)}\geq\zeta^{(t-1)}\left(1+\tilde{c}\eta\right)\geq\zeta_0\left(1+\tilde{c}\eta\right)^{t}
	\]
	and the number of iterations required to exit $\mathcal{C}_{\zeta_0} \backslash \mathcal{C}_1$ is 
	\begin{equation}  \lbl{t_1_eq}
	t_{1}=\frac{\log(\frac{1}{\zeta_0})}{\log(1+\tilde{c}\eta)}.
	\end{equation}
	
	To bound the remaining iterations, we use Lemma \ref{neggradlem} to obtain that for every $\mathbf{w}\in \mathcal{C}_{\zeta_0}\backslash B_r^\infty$, 
	\[
	\left\Vert \mathrm{grad}[\fs](\mathbf{q}(\mathbf{w}))\right\Vert ^{2}\geq \frac{\left\Vert \mathbf{u}^{(i)\ast}\mathrm{grad}[\fs](\mathbf{q}(\mathbf{w}))\right\Vert^2}{||\mathbf{u}^{(i)}||^2} \geq \zeta_0^2 c^2r^2
	\]
	
	where we have used $||\mathbf{u}^{(i)}||^2 = 1+\frac{w_i^2}{q_n^2} \leq 2$. We thus have
	
	\[
	\underset{i=0}{\overset{T-1}{\sum}}\left\Vert \mathrm{grad}[\fs](\mathbf{q}(\mathbf{w})^{(i)})\right\Vert ^{2}
	\]
	\[
	=\underset{i=0}{\overset{t_{1}-1}{\sum}}\left\Vert \mathrm{grad}[\fs](\mathbf{q}(\mathbf{w})^{(i)})\right\Vert ^{2}+\underset{i=t_{1}}{\overset{T-1}{\sum}}\left\Vert \mathrm{grad}[\fs](\mathbf{q}(\mathbf{w})^{(i)})\right\Vert ^{2}
	\]
	\begin{equation} \lbl{teq2}
	> \frac{ \zeta_0^{2} c^{2}}{(n+3)} t_{1}+(T-t_{1})c^{2}r^2 .
	\end{equation}
	
	Choosing $\eta < \frac{1}{2L}$ where $L$ is the gradient Lipschitz constant of $f_s$, from Lemma \ref{riemgdlem} we obtain 
	\[
	\frac{2\left(\fs(\mb q^{(0)})-\fs^{\ast}\right)}{\eta}>\underset{i=0}{\overset{T-1}{\sum}}\left\Vert \mathrm{grad}[\fs](\mathbf{q}^{(i)})\right\Vert ^{2}.
	\]
	According to Lemma \ref{lipproof}, $L=1/\mu$ and thus the above holds if we demand $\eta < \frac{\mu}{2}$. Combining \ref{t_1_eq} and \ref{teq2} gives

	\[
	T<\frac{2\left(\fs(\mb q^{(0)})-\fs^{\ast}\right)}{\eta c^{2}r^{2}}+\frac{\left(1-\frac{\zeta_0^{2}}{(n+3)r^{2}}\right)\log(\frac{1}{\zeta_0})}{\log(1+\tilde{c}\eta)}.
	\]
	
	To obtain the final rate, we use in $g(\mb w^{0})-g^{\ast} \leq \sqrt{n}$ and $\tilde{c}\eta<1\Rightarrow\frac{1}{\log(1+\tilde{c}\eta)}<\frac{\tilde{C}}{\tilde{c}\eta}$ for some $\tilde{C}>0$. Thus one can choose $C>0$ such that
	\begin{equation} \lbl{teq}
	T<\frac{C}{\eta}\left(\frac{\sqrt{n}}{r^{2}}+\log(\frac{1}{\zeta_0})\right).
	\end{equation}
	
	From Lemma \ref{seplemma} the ball $B_r^\infty$ contains a global minimizer of the objective, located at the origin. 
	
	The probability of initializing in $\underset{\breve{A}}{\bigcup} \mathcal{C}_{\zeta_0}$ is simply given from Lemma \ref{vollem} and by summing over the $2n$ possible choices of $\mathcal{C}_{\zeta_0}$, one for each global minimizer (corresponding to a single signed basis vector). 
	
\end{proof}

\begin{lemma}[Riemannian gradient descent iterate bound] \label{riemgdlem}
	For a Riemannian gradient descent algorithm on the sphere with
	step size $t_k < \frac{1}{2L}$, where $L$ is a lipschitz
	constant for $\nabla f ({\V q})$, one has
	\begin{align*}
	f(\V q_1) -f(\V q^\star) &\geq
	f(\V q_1) - f(\V q_T) \\
	&\geq 
	\frac{t_k}{2}\norm{\grad{f}{\V q_k}}^2.
	\end{align*}
\end{lemma}
\begin{proof}
	Just as in the euclidean setting, we can obtain a lower bound on
	progress in function values of iterates of the Riemannian
	gradient descent algorithm from a lower bound on
	the Riemannian gradient. Consider $f : S^{n-1} \to \bb R$,
	which has $L$-lipschitz gradient. Let $\V q_k$ denote the
	current iterate of Riemannian gradient descent, and let $t_k >
	0$ denote the step size. Then we can form the Taylor
	approximation to $f \circ \Exp_{\V q_k}(\V v)$ at $\V 0_{\V
		q_k}$:
	\begin{equation*}
	\hat{f} : B_1(\V 0_{\V q_k}) \cap T_{\V q_k}S^{n-1} \to \bb R : \V v \mapsto f(\V q_k)
	+ \ip{\V v, \nabla f(\V q_k)}.
	\end{equation*}
	%Here $c > 0$ is a small constant that determines the locality of
	%the exponential map.%
	From Taylor's theorem, we have for any $\V
	v \in B_1(\V 0_{\V q_k}) \cap T_{\V q_k}S^{n-1}$
	\begin{equation*}
	\abs{ \hat{f}(\V v) - f \circ \Exp_{\V q_k}(\V v)} \leq
	\frac{1}{2} \norm{\Hess{f}{\V q_k}}\norm{\V v - \V 0_{\V
			q_k}}^2,%
	%\sup_{\V u \in
	%B_c(0) \cap T_{\V x_k}S^{n-1}} \V u^* \left(\nabla^2 f(\V x_k) -
	%\ip{ \nabla f(\V x_k), \V x_k}\V I\right)\V u.
	\end{equation*}
	where the matrix norm is the operator norm on $\bb R^{n \times
		n}$. Using the gradient-lipschitz property of $f$, we readily
	compute
	\begin{align*}
	\norm{\Hess{f}{\V q_k}} &\leq \norm{\nabla^2 f(\V q_k)} +
	\abs{\ip{ \nabla f(\V q_k), \V q_k}} \\
	&\leq 2L,
	\end{align*}
	since $\nabla f(\V 0) = \V 0$ and $\V q_k \in S^{n-1}$. We thus have
	\begin{equation*}
	f \circ \Exp_{\V q_k}(\V v) \leq f(\V q_k) + \ip{\V v, \nabla
		f(\V q_k)} + L\norm{\V v}^2.
	\end{equation*}
	If we put $\V v = -t_k \Grad{f}{\V q_k}$ and write $\V q_{k+1} =
	\Exp_{\V q_k}(-t_k \grad{f}{\V q_k})$, the previous
	expression becomes
	\begin{align*}
	f(\V q_{k+1}) &\leq f(\V q_k) - t_k \norm{\grad{f}{\V q_k}}^2 +
	t_k^2L\norm{\grad{f}{\V q_k}}^2 \\
	&\leq f(\V q_k) - \frac{t_k}{2}\norm{\grad{f}{\V q_k}}^2
	\end{align*}
	if $t_k < \tfrac{1}{2L}$. Thus progress in objective value is
	guaranteed by lower-bounding the Riemannian gradient.
	
	As in the euclidean setting, summing the previous expression
	over iterations $k$ now yields
	\begin{align*}
	\sum_{k=1}^{T-1} f(\V q_{k}) - f(\V q_{k+1}) &= f(\V q_1) -
	f(\V q_{T})\\
	&\geq \frac{t_k}{2}\sum_{k=1}^{T-1}\norm{\grad{f}{\V q_k}}^2;
	\end{align*}
	in addition, it holds $f(\V q_1) - f(\V q_T) \leq f(\V q_1) -
	f(\V q^\star)$. Plugging in a constant step size gives the desired result.
\end{proof}

\begin{lemma}[Lipschitz constant of $\nabla f$] \label{lipproof}
	For any $\V x_1, \V x_2 \in \bb R^n$, it holds
	\begin{equation*}
	\norm{\nabla f(\V x_1) - \nabla f(\V x_2)} \leq
	\frac{1}{\mu} \norm{\V x_1 - \V x_2}.
	\end{equation*}
	\begin{proof}
		It will be enough to study a single coordinate function of
		$\nabla f$. Using a derivative given in section \ref{derivs}, we
		have for $x \in \bb R$
		\begin{equation*}
		\frac{d}{d x} \tanh(x / \mu) =
		\frac{1}{\mu}\sech^2\left(\frac{x}{\mu}\right). 
		\end{equation*}
		A bound on the magnitude of the derivative of this smooth
		function implies a lipschitz constant for $x \mapsto \tanh(x /
		\mu)$. To find the bound, we differentiate again and find the
		critical points of the function. We have, using the chain rule,
		\begin{align*}
		\frac{d}{dx} \left( \frac{1}{\mu}
		\sech^2\left(\frac{x}{\mu}\right)
		\right) &= \frac{-4}{\mu}\sech\left(\frac{x}{\mu}\right)\cdot
		\frac{1}{(e^{x/\mu} + e^{-x/\mu})^2} \\
		&\hphantom{=}\cdot
		\left(\frac{1}{\mu}e^{x/\mu} - \frac{1}{\mu}e^{-x/\mu}\right)
		\\
		&= -\frac{1}{\mu^2} \frac{e^{x/\mu} - e^{-x/\mu}}{(e^{x/\mu} +
			e^{-x/\mu})^3}.
		\end{align*}
		The denominator of this final expression vanishes nowhere.
		Hence, the only critical point satisfies $x/\mu = -x/\mu$, which
		implies $x = 0$. Therefore it holds
		\begin{equation*}
		\frac{d}{dx} \tanh(x/\mu) \leq \frac{1}{\mu}\sech^2(0) =
		\frac{1}{\mu},
		\end{equation*}
		which shows that $\tanh(x/\mu)$ is $(1/\mu)$-lipschitz.
		
		Now let $\V x_1$ and $\V x_2$ be any two points of $\bb R^n$.
		Then one has
		\begin{align*}
		\norm{\nabla f(\V x_1) - \nabla f(\V x_2)} &= \left( \sum_i
		\left( \tanh(x_{1i}/\mu) - \tanh(x_{2i}/\mu) \right)^2
		\right)^{1/2} \\
		&= \left( \sum_i
		\left| \tanh(x_{1i}/\mu) - \tanh(x_{2i}/\mu) \right|^2
		\right)^{1/2} \\
		&\leq \left( \sum_i \frac{1}{\mu}\left| \frac{x_{1i}}{\mu} -
		\frac{x_{2i}}{\mu}\right|^2 \right)^{1/2}\\
		&= \frac{1}{\mu} \norm{\V x_1 - \V x_2},
		\end{align*}
		completing the proof.
	\end{proof}
	
	\begin{lemma}[Separable objective gradient bound] \label{gradupperlemsep}
		The separable objective gradient obeys
		\[
		\left\Vert \nabla_{\mb w}g(\mb w)\right\Vert \leq \sqrt{2n}
		\]
		\[
		\left\Vert \mathrm{grad}[f](\mathbf{q})\right\Vert \leq \sqrt{n}
		\]
	\end{lemma}
	\begin{proof}
		Recalling that the Euclidean gradient is given by $\nabla \fs(\mathbf{q})_{i}=\tanh\left(\frac{q_{i}}{\mu}\right)$ we use Jensen's inequality, convexity of the $L^2$ norm and the triangle inequality to obtain
		\[
		\left\Vert \nabla g_{s}(\mathbf{w})\right\Vert ^{2}\leq\left\Vert \nabla \fs(\mathbf{q})\right\Vert ^{2}+\left\vert \tanh\left(\frac{q_{n}}{\mu}\right)\right\vert ^{2}\frac{\left\Vert \mathbf{w}\right\Vert ^{2}}{q_{n}^{2}}\leq2n
		\]
		while 
		\[
		\left\Vert \mathrm{grad}[\fs](\mathbf{q})\right\Vert =\left\Vert (\mathbf{I}-qq^{\ast})\nabla \fs(\mathbf{q})\right\Vert \leq\left\Vert \nabla \fs(\mathbf{q})\right\Vert =\sqrt{n}
		\]
	\end{proof}
	
\end{lemma}
\section{Proofs - Dictionary Learning} \lbl{dlapp}

%Proof: \ref{neggraddllemp}
{\tiny }
%\begin{proof}[\begin{minipage}{3.25in}\bf{Proof of Lemma \ref{neggraddllem}}:(Dictionary learning population gradient)\end{minipage}] \label{neggraddllemp}
\begin{proof}[\bf{Proof of Lemma \ref{neggraddllem}}:(Dictionary learning population gradient)] \label{neggraddllemp}
	%We have 
	%
	%\[
	%\overline{u}^{T}\nabla g(\mb w)=\mathbb{E}\left[\begin{array}{c}
	%\tanh\left(\frac{\mb q^{\ast}(\mb w) \mb x}{\mu}\right)\\
	%*\left(x_{n}\frac{\left\Vert \mb w\right\Vert _{\infty}}{q_{n}}-\frac{1}{\left\vert K\right\vert }\underset{i\in K}{\sum}x_{i}sign(w_{i})\right)
	%\end{array}\right]
	%\]
	%
	%. Since the distribution of the $x_{i}$ is symmetric, we can absorb
	%the $\mathrm{sign}(\mb w_{i\in K})$ factors into the $x_{i\in K}$. All the terms
	%in the sum are then identical so we can replace them with a single
	%term, denoting its index by $\infty$ and ensuring $\mb w_{\infty}=\left\Vert \mb w\right\Vert _{\infty}$.
	For simplicity we consider the case $\mathrm{sign}(\mb w_i) = 1$. The converse follows by a similar argument. We have 
	\[
	%\mathrm{sign}(w_{i})\nabla_{\mb w}g_{DL}^{pop}(\mb w)_{i} = 
	\mathbf{u}^{(i)\ast}\mathrm{grad}[f_{DL}^{pop}](\mathbf{q}(\mathbf{w})) = 
	\]
	\begin{equation} \lbl{ugeq}
	\mathbb{E}_{\mb x}\left[\tanh\left(\frac{\mb q^{\ast}(\mb w) \mb x}{\mu}\right)\left(-x_{n}\frac{w_{i}}{q_{n}}+x_{i}\right)\right]
	\end{equation}
	
	Following the notation of \cite{sun2017complete}, we write $x_{j}=b_{j}v_{j}$ where
	$b_{j}\sim\mathrm{Bern}(\theta),v_{j}\sim\mathcal{N}(0,1)$ and denote
	the vectors of these variables by $\mathcal{J},v$ respectively. Defining
	$Y^{(n)}=\underset{j\neq n}{\sum} q(\mb w)_{j}x_{j},X^{(n)}=q_{n}v_{n}$, $Y$
	is Gaussian conditioned on a certain setting of $\mathcal{J}$. Using
	Lemma 40 in \cite{sun2017complete} the first term in \ref{ugeq} is 
	
	\[
	-\frac{w_{i}\theta}{q_{n}^{2}}\mathbb{E}_{ \mb v,\mathcal{J}|b_{n}=1}\left[\tanh\left(\frac{Y^{(n)}+X^{(n)}}{\mu}\right)X^{(n)}\right]
	\]
	\[
	=-\frac{w_{i}}{\mu}\theta\mathbb{E}_{ \mb v,\mathcal{J}|b_{n}=1}\left[\mathrm{sech}^{2}\left(\frac{Y^{(n)}+X^{(n)}}{\mu}\right)\right]
	\]
	
	%\[
	%=-\frac{w_{i}}{\mu}\mathbb{E}_{ \mb v,\mathcal{J}}\left[b_{n}\mathrm{sech}^{2}\left(\frac{\mb q^{\ast}(\mb w) \mb x}{\mu}\right)\right]
	%\]
	
	and similarly the second term in \ref{ugeq} is, with $X^{(i)}=w_{i}v_{i},Y^{(i)}=\underset{j\neq i}{\sum} q(\mb w)_{j}x_{j}$
	
	\[
	\frac{\theta}{w_{i}}\mathbb{E}_{ \mb v,\mathcal{J}|b_{i}=1}\left[\tanh\left(\frac{Y^{(i)}+X^{(i)}}{\mu}\right)X^{(i)}\right]
	\]
	\[
	=\frac{w_{i}\theta}{\mu}\mathbb{E}_{ \mb v,\mathcal{J}|b_i=1}\left[\mathrm{sech}^{2}\left(\frac{\mb q^{\ast}(\mb w) \mb x}{\mu}\right)\right]
	\]
	
	if we now define $X=\underset{j\neq n,i}{\sum} q^{\ast}(\mb w)_{j}x_{j}$
	we have 
	
	\[
	\mathbf{u}^{(i)\ast}\mathrm{grad}[f_{DL}^{pop}](\mathbf{q}(\mathbf{w})) = 
	\]
	\[
	=\frac{w_{i}\theta}{\mu}\left(\begin{array}{c}
	\mathbb{E}_{ \mb v,\mathcal{J}|b_{i}=1}\left[\mathrm{sech}^{2}\left(\frac{\mb q^{\ast}(\mb w) \mb x}{\mu}\right)\right]\\
	-\mathbb{E}_{ \mb v,\mathcal{J}|b_{n}=1}\left[\mathrm{sech}^{2}\left(\frac{\mb q^{\ast}(\mb w) \mb x}{\mu}\right)\right]
	\end{array}\right)
	\]
	
	\[
	=\frac{w_{i}\theta}{\mu}\mathbb{E}_{ \mb v,\mathcal{J}}\left[\begin{array}{c}
	\mathrm{sech}^{2}\left(\frac{X+b_{n}q_{n}v_{n}+w_{i}v_{i}}{\mu}\right)\\
	-\mathrm{sech}^{2}\left(\frac{X+q_{n}v_{n}+w_{i}b_{i}v_{i}}{\mu}\right)
	\end{array}\right]
	\]

	\begin{equation} \lbl{ugradeq}
	=\frac{w_{i}\theta(1-\theta)}{\mu}\mathbb{E}_{ \mb v,\mathcal{J}\backslash\{n,i\}}\left[\begin{array}{c}
	\mathrm{sech}^{2}\left(\frac{X+w_{i}v_{i}}{\mu}\right)\\
	-\mathrm{sech}^{2}\left(\frac{X+q_{n}v_{n}}{\mu}\right)
	\end{array}\right]
	\end{equation}

	\subsection{Bounds for $\mathbb{E}\left[\mathrm{sech}^{2}(Y)\right]$}
	
	We already have a lower bound in Lemma 20 of \cite{sun2017complete} that we can use for the
	second term, so we need an upper bound for the first term. Following
	from p. 865, we define $Y\sim\mathcal{N}(0,\sigma_{Y}^{2})$ , $Z=\exp\left(\frac{-2Y}{\mu}\right)$,
	and defining $\beta=1-\frac{1}{\sqrt{T}}$ for some $T>1$ we have
	
	\[
	\mathrm{sech}^{2}(Y/\mu)=\frac{4Z}{(1+Z)^{2}}\leq\frac{4Z}{(1+\beta Z)^{2}}=\underset{k=0}{\overset{\infty}{\sum}}b_{k}Z^{k+1}
	\]
	
	Where $b_{k}=(-\beta)^{k}(k+1)$. Using B.3 from Lemma 40 in \cite{sun2017complete} we have 
	
	\[
	\mathbb{E}\left[\underset{k=0}{\overset{\infty}{\sum}}b_{k}Z^{k+1}\mathbbm{1}_{Y>0}\right]=\underset{k=0}{\overset{\infty}{\sum}}b_{k}\mathbb{E}\left[e^{-2(k+1)Y/\mu}\mathbbm{1}_{Y>0}\right]
	\]
	\[
	=\underset{k=0}{\overset{\infty}{\sum}}b_{k}\exp\left(\frac{1}{2}\left(\frac{2(k+1)}{\mu}\right)^{2}\sigma_{Y}^{2}\right)\Phi^{c}\left(\frac{2(k+1)}{\mu}\sigma_{Y}\right)
	\]
	
	Where $\Phi^{c}(x)$ is the complementary Gaussian CDF (The exchange
	of summation and expectation is justified since $Y>0$ implies $Z\in[0,1]$,
	see proof of Lemma 18 in \cite{sun2017complete} for details). Using the following bounds
	$\frac{1}{\sqrt{2\pi}}\left(\frac{1}{x}-\frac{1}{x^{3}}\right)e^{-x^{2}/2}\leq\Phi^{c}(x)\leq\frac{1}{\sqrt{2\pi}}\left(\frac{1}{x}-\frac{1}{x^{3}}+\frac{3}{x^{5}}\right)e^{-x^{2}/2}$
	by applying the upper (lower) bound to the even (odd) terms in the
	sum, and then adding a non-negative quantity, we obtain
	
	\[
	\leq\frac{1}{\sqrt{2\pi}}\underset{k=0}{\overset{\infty}{\sum}}(-\beta)^{k}(k+1)\left(\frac{1}{\frac{2(k+1)}{\mu}\sigma_{Y}}-\frac{1}{\left(\frac{2(k+1)}{\mu}\sigma_{Y}\right)^{3}}\right)
	\]
	\[
	+\frac{1}{\sqrt{2\pi}}\underset{k=0}{\overset{\infty}{\sum}}\beta^{k}(k+1)\left(\frac{3}{\left(\frac{2(k+1)}{\mu}\sigma_{Y}\right)^{5}}\right)
	\]
	
	and using $\underset{k=0}{\overset{\infty}{\sum}}(-\beta)^{k}=\frac{1}{1+\beta},\underset{k=0}{\overset{\infty}{\sum}}\frac{b_{k}}{(k+1)^{3}}\geq0,\underset{k=0}{\overset{\infty}{\sum}}\frac{\left\vert b_{k}\right\vert }{(k+1)^{5}}\leq2$
	(from Lemma 17 in \cite{sun2017complete}) and taking $T\rightarrow\infty$ so that $\beta\rightarrow1$
	we have 
	
	\[
	\underset{k=0}{\overset{\infty}{\sum}}b_{k}\mathbb{E}\left[Z^{k+1}\mathbbm{1}_{Y>0}\right]\leq\frac{1}{2\sqrt{2\pi}}\frac{1}{\frac{2}{\mu}\sigma_{Y}}+\frac{1}{\sqrt{2\pi}}\frac{6}{\left(\frac{2}{\mu}\sigma_{Y}\right)^{5}}
	\]
	
	giving the upper bound 
	
	\[
	\mathbb{E}\left[\mathrm{sech}^{2}(Y/\mu)\right]=\mathbb{E}\left[1-\tanh^{2}(Y/\mu)\right]\leq8\underset{k=0}{\overset{\infty}{\sum}}b_{k}\mathbb{E}\left[Z^{k+1}\mathbbm{1}_{Y>0}\right]
	\]
	\[
	\leq\sqrt{\frac{2}{\pi}}\frac{\mu}{\sigma_{Y}}+\frac{3\mu^{5}}{2\sqrt{2\pi}\sigma_{Y}^{5}}
	\]
	
	while the lower bound (Lemma 20 in \cite{sun2017complete}) is 
	\[
	\sqrt{\frac{2}{\pi}}\frac{\mu}{\sigma_{Y}}-\frac{2\mu^{3}}{\sqrt{2\pi}\sigma_{Y}^{3}}-\frac{3\mu^{5}}{2\sqrt{2\pi}\sigma_{Y}^{5}}\leq\mathbb{E}\left[\mathrm{sech}^{2}(Y)\right]
	\]

	\subsection{Gradient bounds}\label{dlgradlem}
	After conditioning on $\mathcal{J} \backslash \{n,i \}$ the variables $X+q_{n}v_{n},X+q_{i}v_{i}$ are Gaussian. We can thus plug the bounds into \ref{ugradeq} to obtain 
	
	\[
	\mathbf{u}^{(i)\ast}\mathrm{grad}[f_{DL}^{pop}](\mathbf{q}(\mathbf{w}))  \geq\sqrt{\frac{2}{\pi}}w_{i}\theta(1-\theta)
	%\mathrm{sign}(w_{i})\nabla_{\mb w}g_{s}(\mb w)_{i} 
	\]
	\[
	*\mathbb{E}_{\mathcal{J}\backslash\{n,i\}}\left[\begin{array}{c}
	\frac{1}{\sqrt{\sigma_{X}^{2}+w_{i}^{2}}} - \frac{\mu^{2}}{\left(\sigma_{X}^{2}+w_{i}^{2}\right)^{3/2}} - \frac{3\mu^{4}}{4\left(\sigma_{X}^{2}+w_{i}^{2}\right)^{5/2}}\\
	-\frac{1}{\sqrt{\sigma_{X}^{2}+q_{n}^{2}}}-\frac{3\mu^{4}}{4\left(\sigma_{X}^{2}+q_{n}^{2}\right)^{5/2}}
	\end{array}\right]
	\]
	
	\[
	\geq\sqrt{\frac{2}{\pi}}w_{i}\theta(1-\theta)\left(\begin{array}{c}
	\mathbb{E}_{\mathcal{J}\backslash\{n,i\}}\left[\frac{\sqrt{\sigma_{X}^{2}+q_{n}^{2}}-\sqrt{\sigma_{X}^{2}+w_{i}^{2}}}{\sqrt{\sigma_{X}^{2}+q_{n}^{2}}\sqrt{\sigma_{X}^{2}+w_{i}^{2}}}\right]\\
	-\frac{\mu^{2}}{w_{i}^{3}}-\frac{3\mu^{4}}{2w_{i}^{5}}
	\end{array}\right)
	\]
	
	the term in the expectation is positive since $q_{n}>||w||_{\infty}\left(1+\zeta\right)>w_{i}$
	giving
	
	\[
	\geq\sqrt{\frac{2}{\pi}}w_{i}\theta(1-\theta)\left(\begin{array}{c}
	\mathbb{E}_{\mathcal{J}\backslash\{n,i\}}\left[\begin{array}{c}
	\sqrt{\sigma_{X}^{2}+q_{n}^{2}}\\
	-\sqrt{\sigma_{X}^{2}+w_{i}^{2}}
	\end{array}\right]\\
	-\frac{\mu^{2}}{w_{i}^{3}}-\frac{3\mu^{4}}{2w_{i}^{5}}
	\end{array}\right)
	\]
	
	. To extract the $\zeta$ dependence we plug in $q_{n}>w_{i}\left(1+\zeta\right)$ and develop to first order in $\zeta$ (since the resulting function of $\zeta$ is convex) giving
	
	\[
	\geq\sqrt{\frac{2}{\pi}}w_{i}\theta(1-\theta)\left(\begin{array}{c}
	\mathbb{E}_{\mathcal{J}\backslash\{n,i\}}\left[\frac{w_{i}^{2}\zeta}{\sqrt{\sigma_{X}^{2}+w_{i}^{2}}}\right]\\
	-\frac{\mu^{2}}{w_{i}^{3}}-\frac{3\mu^{4}}{2w_{i}^{5}}
	\end{array}\right)
	\]
	
	\[
	\geq\sqrt{\frac{2}{\pi}}\theta(1-\theta)\left(w_{i}^{3}\zeta-\frac{\mu^{2}}{w_{i}^{2}}-\frac{3\mu^{4}}{2w_{i}^{4}}\right)
	\]
	
	Given some $\zeta$ and $r$ such that $w_{i}>r$, if we now choose $\mu$ such that $\mu<\sqrt{\frac{\sqrt{1+\frac{3}{4}r^{3}\zeta}-1}{3}}r$ we have the desired result. This can be achieved by requiring $\mu<\cdlmu r^{5/2}\sqrt{\zeta}$ for a suitably chosen $\cdlmu > 0$.
	
	%ii) 
	%We consider the case $\mathrm{sign}(w_{i})=1$, a similar proof
	%follows for the converse case. Starting from \ref{ugradeq}, we look for solutions to
	%\[
	%\mathrm{sign}(w_{i})\nabla_{\mb w}g_{DL}^{pop}(\mb w)_{i}
	%\]
	%\[
	%=\frac{w_{i}\theta(1-\theta)}{\mu}\mathbb{E}_{ \mb v,\mathcal{J}\backslash\{n,i\}}\left[\begin{array}{c}
	%\mathrm{sech}^{2}\left(\frac{X+w_{i}v_{i}}{\mu}\right)\\
	%-\mathrm{sech}^{2}\left(\frac{X+q_{n}v_{n}}{\mu}\right)
	%\end{array}\right]=0
	%\]
	%
	%. One solution is $w_{i}=0$. We show that the only other solution
	%is $w_{i}=q_{n}$ by showing that $\mathbb{E}\left[\mathrm{sech}^{2}\left(\frac{X+w_{i}v_{i}}{\mu}\right)\right]$
	%is injective. This is done by showing that the gradient of this function
	%with respect to $w_{i}$ cannot change sign. Indeed
	%
	%\[
	%\mathrm{sign}\left[\frac{\partial}{\partial w_{i}}\mathbb{E}\left[\mathrm{sech}^{2}\left(\frac{X+w_{i}v_{i}}{\mu}\right)\right]\right]
	%\]
	%
	%\[
	%=\mathrm{sign}\left[\mathbb{E}\left[-2\mathrm{sech}^{2}\left(\frac{X+w_{i}v_{i}}{\mu}\right)\mathrm{tanh}\left(\frac{X+w_{i}v_{i}}{\mu}\right)\frac{v_{i}}{\mu}\right]\right]
	%\]
	%
	%\[
	%=-\mathrm{sign}\left[\mathbb{E}\left[\mathrm{tanh}\left(\frac{X+w_{i}v_{i}}{\mu}\right)v_{i}\right]\right]
	%\]
	%
	%using B.8 in \cite{sun2017complete}
	%
	%\[
	%=-\mathrm{sign}\left[\sigma_{X}^{2}\mathbb{E}\left[\mathrm{sech}^{2}\left(\frac{X+w_{i}v_{i}}{\mu}\right)\right]\right]=-1
	%\]
	%
	%From part i) $\mathrm{sign}(w_{i})\nabla_{\mb w}g_{DL}^{pop}(\mb w)_{i}$ is positive for intermediate values of
	%$w_{i}$ and cannot change sign (since $0 \leq w_i \leq q_n$), which completes the proof.
	
\end{proof}

\begin{lemma}[Point-wise concentration of projected gradient] \label{conclem}
	For $\mathbf{u}^{(i)}$ defined in \ref{ueq}, the gradient of the objective \ref{dl2eq} obeys
	\[
	\mathbb{P}\left[\left\vert \mathbf{u}^{(i)\ast}\mathrm{grad}[f_{DL}](\mathbf{q})-\mathbb{E}\left[\mathbf{u}^{(i)\ast}\mathrm{grad}[f_{DL}](\mathbf{q})\right]\right\vert \geq t\right]
	\]
	\[
	\leq2\exp\left(-\frac{pt^{2}}{4+2\sqrt{2}t}\right)
	\]
\end{lemma}
%Proof: \ref{conclemp}

%\begin{proof}[\begin{minipage}{3.25in}\bf{Proof of Lemma \ref{conclem}: (Point-wise concentration of projected gradient)}\end{minipage}] \label{conclemp}
\begin{proof}[\bf{Proof of Lemma \ref{conclem}: (Point-wise concentration of projected gradient)}] \label{conclemp}
	If we denote by $\mb x^{i}$ a column of the data matrix with entries
	$x_{j}^{i}\sim BG(\theta)$, we have 
	\[
	\mathbf{u}^{(i)\ast}\mathrm{grad}[f_{DL}](\mathbf{q}(\mathbf{w})) 
	\]
	\[
	=\frac{1}{p}\underset{k=1}{\overset{p}{\sum}}\tanh\left(\frac{\mb q^{\ast}(\mb w) \mb x^{k}}{\mu}\right)\left(x^k_i - x_{n}^{k}\frac{ w_{i}}{q_{n}}\right)\equiv\frac{1}{p}\underset{k=1}{\overset{p}{\sum}}Z_{k}
	\]
	. Since $\tanh(x)$ is bounded by 1, 
	\[
	\left\vert Z_{k}\right\vert \leq\left\vert \left(x^k_i - x_{n}^{k}\frac{ w_{i}}{q_{n}}\right)\right\vert \equiv \left\vert u^{T}x^{k}\right\vert 
	\]
	
	. Invoking Lemma 21 from  \cite{sun2017complete} and $\left\Vert u\right\Vert^2 =1+\frac{ w_{i}^{2}}{q_{n}^{2}}\leq2$
	we obtain
	\[
	\mathbb{E}\left[\left\vert Z_{k}\right\vert ^{m}\right]\leq\mathbb{E}_{Z\sim\mathcal{N}(0,2)}\left[\left\vert Z\right\vert ^{m}\right]\leq\sqrt{2}^{m}(m-1)!!
	\]
	\[
	\leq2\sqrt{2}^{m-2}\frac{m!}{2}
	\]
	and using Lemma 36 in \cite{sun2017complete} with $R=\sqrt{2},\sigma=\sqrt{2}$ we have 
	
	\[
	\mathbb{P}\left[\left\vert \nabla g_{DL}(\mb w)_i-\mathbb{E}\left[\nabla g_{DL}(\mb w)_i\right]\right\vert \geq t\right]
	\]
	\[
	\leq2\exp\left(-\frac{pt^{2}}{4+2\sqrt{2}t}\right)
	\]
	
\end{proof}

\begin{lemma} [Projection Lipschitz Constant] \label{liplemma}
	The Lipschitz constant for $\mathbf{u}^{(i)\ast}\mathrm{grad}[f_{DL}](\mathbf{q}(\mathbf{w}))$ is 
	\[
	L=2\sqrt{n}\left\Vert \mb X\right\Vert _{\infty}\left(\frac{\left\Vert \mb X\right\Vert _{\infty}}{\mu}+1\right)
	\]
\end{lemma}

\begin{proof}[\bf{Proof of Lemma \ref{liplemma}:} (Projection Lipschitz Constant)] \label{lipproof}
	We have 
	\[
	|\mathbf{u}^{(j)\ast}\mathrm{grad}[f_{DL}](\mathbf{q}(\mathbf{w}))-\mathbf{u}^{(j)\ast}\mathrm{grad}[f_{DL}](\mathbf{q}(\mathbf{w'}))|
	\]
	\[
	=\left\vert \frac{1}{p}\underset{i=1}{\overset{p}{\sum}}\left[\begin{array}{c}
	\mathrm{tanh}(\frac{\mb q^{\ast}(\mb w)\mb x^{i}}{\mu})\left(x^{i}_j-\frac{x_{n}^{i}}{q_{n}(\mb w)}w_j\right)\\
	-\mathrm{tanh}(\frac{\mb q^{\ast}(\mb w')\mb x^{i}}{\mu})\left(x^{i}_j-\frac{x_{n}^{i}}{q_{n}(\mb w')}w'_j\right)
	\end{array}\right]\right\vert 
	\]
	
	\[
	\equiv\left\vert \frac{1}{p}\underset{i=1}{\overset{p}{\sum}}\left[\mathrm{tanh}(\frac{\mb q^{\ast}(\mb w)\mb x^{i}}{\mu})s(\mb w)-\mathrm{tanh}(\frac{\mb q^{\ast}(\mb w')\mb x^{i}}{\mu})s(\mb w')\right]\right\vert 
	\]
	
	where we have defined $s(\mb w)=x_j^i-\frac{x_{n}}{q_{n}(\mb w)}w_j$.
	Using $\mb q(\mb w),\mb q(\mb w')\in C \Rightarrow q_{n}(\mb w),q_{n}(\mb w')\geq\frac{1}{2\sqrt{n}}$
	we have 
	
	\[
	\left\vert s(\mb w)-s(\mb w')\right\vert =\left\vert x_{n}^i\right\vert \left\vert \frac{w_j}{q_{n}(\mb w)}-\frac{w'_j}{q_{n}(\mb w')}\right\Vert 
	\]
	\[
	\leq\left\vert x_{n}\right\vert 2\sqrt{n}\left\Vert \mb w- \mb w'\right\Vert 
	\]
	
	Lemma 25 in \cite{sun2017complete} gives 
	
	\[
	\left\vert \mathrm{tanh}(\frac{\mb q^{\ast}(\mb w)\mb x}{\mu})-\mathrm{tanh}(\frac{\mb q^{\ast}(\mb w')\mb x}{\mu})\right\vert \leq\frac{2\sqrt{n}}{\mu}\left\Vert x\right\Vert \left\Vert \mb w-\mb w'\right\Vert 
	\]
	
	We also use the fact that $\mathrm{tanh}$ is bounded by 1 and $s(\mb w)$ is bounded
	by $\left\Vert \mb X\right\Vert _{\infty}$. We can then use Lemma 23 in \cite{sun2017complete}
	to obtain 
	
	\[
	|\mathbf{u}^{(j)\ast}\mathrm{grad}[f_{DL}](\mathbf{q}(\mathbf{w}))-\mathbf{u}^{(j)\ast}\mathrm{grad}[f_{DL}](\mathbf{q}(\mathbf{w'}))|
	\]
	\[
	\leq\frac{2\sqrt{n}}{p}\underset{i=1}{\overset{p}{\sum}}(\frac{1}{\mu}\left\Vert x^{i}\right\Vert_\infty ^{2} +\left\Vert x^{i}\right\Vert_\infty ) \left\Vert \mb w- \mb w'\right\Vert 
	\]
	
	\[
	\leq2\sqrt{n}\left\Vert \mb X\right\Vert _{\infty}\left(\frac{\left\Vert \mb X\right\Vert _{\infty}}{\mu}+1\right)\left\Vert \mb w- \mb w'\right\Vert 
	\]
	
	we thus have $L=2\sqrt{n}\left\Vert \mb X\right\Vert _{\infty}\left(\frac{\left\Vert \mb X\right\Vert _{\infty}}{\mu}+1\right)$. 
\end{proof}

\begin{lemma}[Uniformized gradient fluctuations] \label{uniflem} %
	For all $\mathbf{w}\in \mathcal{C}_{\zeta},i\in[n],$
	with probability $\mathbb{P}>\mathbb{P}_{y}$
	
	we have 
	\[
	\left\vert \begin{array}{c}
	\mathbf{u}^{(i)\ast}\mathrm{grad}[f_{DL}](\mathbf{q}(\mathbf{w}))\\
	-\mathbb{E}\left[\mathbf{u}^{(i)\ast}\mathrm{grad}[f_{DL}](\mathbf{q}(\mathbf{w}))\right]
	\end{array}\right\vert \leq y(\theta,\zeta)
	\]
	where 
	\[
	\mathbb{P}_{y}\equiv2\exp\left(\begin{array}{c}
	-\frac{1}{4}\frac{py(\theta,\zeta)^{2}}{4+\sqrt{2}y(\theta,\zeta)}+\log(n)\\
	+n\log\left(\frac{48\sqrt{n}\left(\frac{4\log(np)}{\mu}+\sqrt{\log(np)}\right)}{y(\theta,\zeta)}\right)
	\end{array}\right)
	\]
\end{lemma}
Proof: \ref{uniflemp}

\begin{proof}[\bf{Proof of Lemma \ref{uniflem}:(Uniformized gradient fluctuations)}] \label{uniflemp}
	%
	%We need our $\varepsilon$-net to provide 
	%
	%\[
	%\left\vert \nabla g-E[\nabla g]\right\vert \leq\frac{1}{2}c_{DL}\min\left(\zeta,\left(\frac{1-2\frac{\theta(1-\theta)}{\sqrt{n-1}}}{1+2\frac{\theta(1-\theta)}{\sqrt{n-1}}}\frac{1}{40\sqrt{5(n-1)}}\right)^{3}\right)=\frac{1}{2}\frac{\theta(1-\theta)}{\sqrt{2\pi}}\min\left(\frac{\zeta}{(n+3)^{3/2}},\left(\frac{1-2\frac{\theta(1-\theta)}{\sqrt{n-1}}}{1+2\frac{\theta(1-\theta)}{\sqrt{n-1}}}\frac{1}{40\sqrt{5(n-1)}}\right)^{3}\right)\equiv y(\theta,\zeta)
	%\]
	
	For $\mathbf{X}\in\mathbb{R}^{n\times p}$ with i.i.d. $BG(\theta)$
	entries, we define the event $\mathcal{E}_{\infty}\equiv\{1\leq\left\Vert \mathbf{X}\right\Vert _{\infty}\leq4\sqrt{\log(np)}\}$.
	We have 
	\[
	\mathbb{P}[\mathcal{E}_{\infty}^{c}]\leq\theta(np)^{-7}+e^{-0.3\theta np}
	\]
	For any $\varepsilon\in(0,1)$ we can construct an $\varepsilon$-net
	$N$ for $\mathcal{C}_{\zeta}\backslash B_{1/20\sqrt{5(n-1)}}^{2}(0)$ with
	at most $(3/\varepsilon)^{n}$ points. Using Lemma \ref{liplemma},
	on $\mathcal{E}_{\infty}$, $\mathrm{grad} [f_{DL}](\mathbf{q})_{i}$ is $L$-Lipschitz
	with 
	\[
	L=8\sqrt{n}\left(\frac{4\log(np)}{\mu}+\sqrt{\log(np)}\right)
	\]
	. If we choose $\varepsilon=\frac{y(\theta,\zeta)}{2L}$ we have 
	\[
	\left\vert N\right\vert \leq(\frac{6L}{y(\theta,\zeta)})^{n}
	\]
	. We then denote by $\mathcal{E}_{g}$ the event 
	\[
	\underset{\mathbf{w}\in N,i\in[n]}{\max}\left\vert \begin{array}{c}
	\mathbf{u}^{(i)\ast}\mathrm{grad}[f_{DL}](\mathbf{q}(\mathbf{w}))\\
	-\mathbb{E}\left[\mathbf{u}^{(i)\ast}\mathrm{grad}[f_{DL}](\mathbf{q}(\mathbf{w}))\right]
	\end{array}\right\vert \leq\frac{y(\theta,\zeta)}{2}
	\]
	and obtain that on $\mathcal{E}_{g}\cap\mathcal{E}_{\infty}$ 
	\[
	\underset{\mathbf{w}\in \mathcal{C}_{\zeta},i\in[n]}{\sup}\left\vert \nabla g_{DL}(\mathbf{w})_{i}-\mathbb{E}\left[\nabla g_{DL}(\mathbf{w})_{i}\right]\right\vert \leq y(\theta,\zeta)
	\]
	
	. Setting $t=\frac{b(\theta)}{2}$ in the result of Lemma \ref{conclem}
	gives that for all $\mathbf{w}\in \mathcal{C}_{\zeta},i\in[n]$, 
	
	\[
	\mathbb{P}\left[\left\vert \begin{array}{c}
	\mathbf{u}^{(i)\ast}\mathrm{grad}[f_{DL}](\mathbf{q}(\mathbf{w}))\\
	-\mathbb{E}\left[\mathbf{u}^{(i)\ast}\mathrm{grad}[f_{DL}](\mathbf{q}(\mathbf{w}))\right]
	\end{array}\right\vert  \geq\frac{y(\theta,\zeta)}{2}\right]
	\]
	\[
	\leq2\exp\left(-\frac{1}{4}\frac{py(\theta,\zeta)^{2}}{4+2\sqrt{2}y(\theta,\zeta)}\right)
	\]
	
	and thus 
	
	\[
	\mathbb{P}\left[\mathcal{E}_{g}^{c}\right]\leq2\exp\left(\begin{array}{c}
	-\frac{1}{4}\frac{py(\theta,\zeta)^{2}}{4+\sqrt{2}y(\theta,\zeta)^{2}}\\
	+n\log\left(\frac{6L}{b(\theta)}\right)+\log(n)
	\end{array}\right)
	\]
	
	%\[
	%=2\exp\left(-\frac{1}{4}\frac{py(\theta,\zeta)^{2}}{4+\sqrt{2}y(\theta,\zeta)}+n\log\left(\frac{48\sqrt{n}\left(\frac{4\log(np)}{\mu}+\sqrt{\log(np)}\right)}{b(\theta)}\right)+\log(n)\right)\equiv1-\mathbb{P}_{g}
	%\]
\end{proof}

\begin{lemma}[Gradient descent convergence rate for dictionary learning - population] \label{newratedllem}
	For any $1 > \zeta_0 > 0$ and $s > \frac{\mu}{4\sqrt{2}}$, Riemannian gradient descent with step size $\eta<\frac{\cetadlpop s}{n}$ 
	%\[\eta<\min\left(
	%\frac{1}{6\sqrt{\theta n(n+3)}},\frac{1}{360\sqrt{5\theta n(n-1)}},\frac{\tilde{c}_{3}s}{n}
	%%,\frac{\sqrt{\theta}s}{40\sqrt{\pi n}},\sqrt{\frac{5}{2\pi}}\frac{s}{16n},
	%\right)\]
	on the dictionary learning population objective \ref{dleq} with $\mu < \frac{\cmudlpop \sqrt{\zeta_0}}{n^{5/4}}, \theta \in (0,\frac{1}{2}) $% $\mu< \min \left(\frac{\sqrt{\zeta_0}}{(n+3)^{5/4}8\sqrt{6}}, \frac{1}{\sqrt{12}(8000(n-1))^{5/4}},\frac{9}{50}\right),\theta<\frac{1}{2}$
	, enters a ball of radius $\cdistfrommin s$ from a target solution in 
	\[
	T<\frac{ \Cdlpop}{\eta\theta}\left(\frac{1}{s} +n\log\frac{1}{\zeta_{0}}\right)
	%T<\begin{array}{c}
	%\frac{\log\frac{1}{\zeta_{0}}}{\log\left(1+\frac{\sqrt{n}\cdl}{16(n+3)^{3/2}}\eta\right)}+\frac{\log(8)}{\log\left(1+\frac{\sqrt{n}\cdl}{2(8000(n-1))^{3/2}}\eta\right)}\\
	%+\frac{\tilde{C}_{3}}{s \theta \eta}
	%
	%\end{array}
	\]
	iterations with probability 
	\[
	\mathbb{P} \geq 1- 2 \log(n) \zeta_0%\mathbb{P}_{Q}
	\]
	where the $c_i,C_i$ are positive constants. % and $\mathbb{P}_{Q}$ is defined in \ref{pqeq}. 
	%\footnote{The leading order $n$ dependence required for the smoothness constant is $\mu\sim n^{-5/4}$ (in line with \cite{sun2017complete}). As a result, one obtains that for large $n$, $T\sim\log(n)n^{13/4}$.%, where the leading order contribution is from the second term in the sum. }
	
\end{lemma}
\begin{proof}[\bf{Proof of Lemma \ref{newratedllem}}: (Gradient descent convergence rate for dictionary learning - population)] \label{newratedllemp}
	
	The rate will be obtained by splitting $\mathcal{C}_{\zeta_0}$ into three regions. We consider convergence to $B^2_{s}(0)$ since this set contains a global minimizer. Note that the balls in the proof are defined with respect to $\mb w$. %Since $\mu$ is generally small, for practical purposes convergence to this neighborhood may be sufficient. If not, Theorem 2 in \cite{sun2017complete} guarantees that the objective is strongly convex in this region, ensuring linear convergence within $B_{\mu/4\sqrt{2}}(0)$. 
	
	%Choosing $\eta < 1/L$ where $L$ is the gradient Lipschitz constant, we bound the change in the objective by 
	%\[
	%\frac{2\left(g_{s}(\mb w^{(0)})-g_{s}^{\ast}\right)}{\eta}\geq\begin{array}{c}
	%\underset{\mathcal{C}_{\zeta}\backslash C_{1}}{\underbrace{b_{1}^{2}t_{1}}}+\underset{B_{1/20\sqrt{5}}(0)\backslash B_{\mu/4\sqrt{2}}(0)}{\underbrace{b_{2}^{2}t_{2}}}\\
	%+\underset{C_{1}\backslash B_{1/20\sqrt{5}}(0)}{\underbrace{b_{0}^{2}(T-t_{1}-t_{2})}}
	%\end{array}
	%\]
	%
	%where the $b_i$ are bounds on the gradient norm in the respective regions. 
	%\subsection{$\mathcal{C}_{\zeta} \backslash B_1^\infty(0)$}
	%In this region, 
	%\[
	%\frac{1}{\sqrt{2+\zeta}} \geq \left\Vert \mb w\right\Vert _{\infty} > 1
	%\] 
	%where the upper bound is obtained form Lemma \ref{czlemma}. Additionally, using Lemma \ref{neggraddllem}, at every iteration of gradient descent we have 
	%\[
	%\left\Vert \mb w^{(t+1)}\right\Vert _{\infty}\leq\left\Vert \mb w^{(t)}\right\Vert _{\infty}-\eta c_{DL}\left\Vert \mb w^{(t)}\right\Vert _{\infty}^{3}
	%\]
	%
	%\[
	%\leq\left\Vert \mb w^{(t)}\right\Vert _{\infty}\left(1-\frac{\eta c_{DL}}{2000(n-1)}\right)
	%\]
	%
	%the maximal number of iterations in this region given by 
	%\[
	%\left(1-\frac{\eta c_{DL}}{2000(n-1)}\right)^{t_{2}}\leq\frac{\sqrt{3}}{20\sqrt{5(n-1)}}
	%\]
	%\[
	%t_{2}\leq\frac{\log\left(\frac{\sqrt{3}}{20\sqrt{5(n-1)}}\right)}{\log\left(1-\frac{\eta c_{DL}}{2000(n-1)}\right)}
	%\]
	\subsection{$\mathcal{C}_{\zeta_0} \backslash B^2_{1/20\sqrt{5}}(0)$}
	
	The analysis in this region is completely analogous to that in the  first part of the proof of Lemma \ref{newratelem}. For every point in this set we have  
	\[
	\left\Vert \mb w\right\Vert _{\infty} > \frac{1}{20\sqrt{5(n-1)}}
	\]
	. From Lemma \ref{czlemma} we know that $\sqrt{\frac{n-1}{(2+\zeta^{(t)})\zeta^{(t)}+n}}<\frac{1}{20\sqrt{5}} \Rightarrow \mb w^{(t)} \in B^2_{1/20\sqrt{5}}(0)$ hence in this set $\zeta<8$. If we choose $r=\frac{1}{40\sqrt{5(n-1)}}$, since for every point in this region $r^3\zeta < 1$, we have $\frac{r^{5/2}\sqrt{\zeta}}{2\sqrt{3}}<\sqrt{\frac{\sqrt{1+\frac{3}{4}r^{3}\zeta}-1}{3}}r=z(r,\zeta)$ and we thus demand $\mu<\frac{\sqrt{\zeta_{0}}}{\left(40\sqrt{5(n-1)}\right)^{5/2}2\sqrt{3}}\leq\frac{r^{5/2}\sqrt{\zeta}}{2\sqrt{3}}$ and obtain from Lemma \ref{neggraddllem} that for $|w_i| > r$
	\[
	\mathbf{u}^{(i)\ast}\mathrm{grad}[f_{DL}^{pop}](\mathbf{q}(\mathbf{w}))\geq\frac{c_{DL}}{(8000(n-1))^{3/2}}
	\]
	. We now require $\eta<\frac{1}{360\sqrt{5\theta n(n-1)}}=\frac{b-r}{3M}$ we can apply Lemma \ref{zetalem} with $ b = \frac{1}{20\sqrt{5(n-1)}},r =\frac{1}{40\sqrt{5(n-1)}}, M = \sqrt{\theta n}$ (since the maximal norm of the Riemannian gradient is $\sqrt{\theta n}$ from Lemma \ref{gradupperlem}), obtaining that at every iteration in this region 
	\[\zeta'\geq\zeta\left(1+\frac{\sqrt{n}c_{DL}}{2(8000(n-1))^{3/2}}\eta\right)\]
	and the maximal number of iterations required to obtain $\zeta > 8$ and exit this region is given by 
	\begin{equation} \lbl{t1eq}
	t_{1}=\frac{\log(8/\zeta_0)}{\log\left(1+\frac{\sqrt{n}c_{DL}}{2(8000(n-1))^{3/2}}\eta\right)}
	\end{equation}

	\subsection{$B^2_{1/20\sqrt{5}}(0) \backslash B^2_{s}(0)$}
	
	According to Proposition 7 in \cite{sun2017complete}, which we can apply since $s \geq \frac{\mu}{4\sqrt{2}},\mu < \frac{9}{50}$, in this region we have 
	
	\[
	\frac{\mb w^{\ast}\nabla_{\mb w}g_{DL}^{pop}(\mb w)}{\left\Vert \mb w \right\Vert }\geq c \theta
	\]
	
	A simple calculation shows that $\nabla_{\mb w}g_{DL}^{pop}(\mb w)=\left(\frac{\partial \mb \varphi}{\partial\mathbf{w}}\right)^{\ast}\mathrm{grad}[f_{DL}^{pop}](\mathbf{q}(\mathbf{w}))$
	where $\mb \varphi$ is the map defined in \ref{phieq}, and thus
	
	\[
	\frac{\mathbf{w}^{\ast}\left(\frac{\partial \mb \varphi}{\partial\mathbf{w}}\right)^{\ast}\mathrm{grad}[f_{DL}^{pop}](\mathbf{q}(\mathbf{w}))}{\left\Vert \mathbf{w}\right\Vert }=\frac{\left(\begin{array}{c}
		\mathbf{w}^{\ast}\\
		-\frac{\left\Vert \mathbf{w}\right\Vert ^{2}}{q_{n}}
		\end{array}\right)\mathrm{grad}[f_{DL}^{pop}](\mathbf{q}(\mathbf{w}))}{\left\Vert \mathbf{w}\right\Vert }
	\]
	\begin{equation} \label{gradprojeq}
	>\theta c
	\end{equation}
	
	. Defining $h(\mathbf{q})=\frac{\left\Vert \mathbf{w}\right\Vert ^{2}}{2}$,
	and denoting by $\mathbf{q}'$ an update of Riemannian gradient descent
	with step size $\eta$, we have (using a Lagrange remainder term)
	
	\[
	h(\mathbf{q}')=h(\mathbf{q})+\frac{\partial h(\mathbf{q}')}{\partial\eta}\eta+\underbrace{\underset{0}{\overset{\eta}{\int}}dt\frac{\partial^{2}h(\mathbf{q}')}{\partial\eta^{2}}_{\eta=t}(\eta-t)}_{\equiv R}
	\]
	
	\[
	=\frac{\left\Vert \mathbf{w}\right\Vert ^{2}}{2}-\left\langle \mathrm{grad}[f_{DL}^{pop}](\mathbf{q}),\frac{\partial h(\mathbf{q})}{\partial\mathbf{q}}\right\rangle +R
	\]
	
	where in the last line we used $\mathbf{q}'=\cos(g\eta)\mathbf{q}-\sin(g\eta)\frac{\mathrm{grad}[f_{DL}^{pop}](\mathbf{q})}{g}$
	where $g\equiv\left\Vert \mathrm{grad}[f_{DL}^{pop}](\mathbf{q})\right\Vert $.
	Since $\left\langle \mathrm{grad}[f_{DL}^{pop}](\mathbf{q}),\frac{\partial h(\mathbf{q})}{\partial\mathbf{q}}\right\rangle =\left\langle \mathrm{grad}[f_{DL}^{pop}](\mathbf{q}),\left(\mathbf{I}-\mathbf{q}\mathbf{q}^{\ast}\right)\frac{\partial h(\mathbf{q})}{\partial\mathbf{q}}\right\rangle $
	and
	
	\[
	\left(\mathbf{I}-\mathbf{q}\mathbf{q}^{\ast}\right)\frac{\partial h(\mathbf{q})}{\partial\mathbf{q}}=\left(\mathbf{I}-\mathbf{q}\mathbf{q}^{\ast}\right)\left(\begin{array}{c}
	\mathbf{w}\\
	-q_{n}
	\end{array}\right)
	\]
	\[
	=\left(\begin{array}{c}
	\mathbf{w}\\
	-q_{n}
	\end{array}\right)-(\left\Vert \mathbf{w}\right\Vert ^{2}-q_{n}^{2})\mathbf{q}
	=2(1-\left\Vert \mathbf{w}\right\Vert ^{2})\left(\begin{array}{c}
	\mathbf{w}\\
	-\frac{\left\Vert \mathbf{w}\right\Vert ^{2}}{q_{n}}
	\end{array}\right)
	\]
	
	we obtain (using \ref{gradprojeq})
	
	\[
	\frac{\left\Vert \mathbf{w}'\right\Vert ^{2}}{2}=\frac{\left\Vert \mathbf{w}\right\Vert ^{2}}{2}+2(1-\left\Vert \mathbf{w}\right\Vert ^{2})\eta\left\langle \mathrm{grad}[f_{DL}^{pop}](\mathbf{q}),\left(\begin{array}{c}
	\mathbf{w}\\
	-\frac{\left\Vert \mathbf{w}\right\Vert ^{2}}{q_{n}}
	\end{array}\right)\right\rangle +R
	\]
	\[
	<\frac{\left\Vert \mathbf{w}\right\Vert ^{2}}{2}-2(1-\left\Vert \mathbf{w}\right\Vert ^{2})\left\Vert \mathbf{w}\right\Vert \theta c\eta+R
	\]
	
	It remains to bound $R$. Denoting $r=\left(\begin{array}{c}
	\mathbf{w}\\
	-q_{n}
	\end{array}\right)^{\ast}\mathrm{grad}[f](\mathbf{q})$ we have 
	
	\[
	\frac{\partial^{2}h(\mathbf{q}')}{\partial\eta^{2}}_{\eta=t}=\left(\frac{\partial\mathbf{q}'}{\partial\eta}\right)^{\ast}\frac{\partial^{2}h(\mathbf{q})}{\partial\mathbf{q}\partial\mathbf{q}}\frac{\partial\mathbf{q}'}{\partial\eta}{}_{\eta=t}+\frac{\partial h(\mathbf{q})}{\partial\mathbf{q}}^{\ast}\frac{\partial^{2}\mathbf{q}'}{\partial\eta^{2}}{}_{\eta=t}
	\]
	
	%\[
	%\begin{array}{c}
	%\left(g\sin(gt)\mathbf{q}+\cos(gt)\mathrm{grad}[f](\mathbf{q})\right)^{\ast}\left(\begin{array}{cc}
	%\mathbf{I}_{n-1\times n-1} & \mathbf{0}\\
	%\mathbf{0} & -1
	%\end{array}\right)\left(g\sin(gt)\mathbf{q}+\cos(gt)\mathrm{grad}[f](\mathbf{q})\right)\\
	%+\left(\begin{array}{c}
	%\mathbf{w}\\
	%-q_{n}
	%\end{array}\right)^{\ast}\left(-g^{2}\cos(gt)\mathbf{q}+g\sin(gt)\mathrm{grad}[f](\mathbf{q})\right)
	%\end{array}
	%\]
	
	\[
	=\begin{array}{c}
	\cos^{2}(gt)\left(\overline{\mathrm{grad}}[f_{DL}^{pop}](\mathbf{q})^{2}-\mathrm{grad}[f_{DL}^{pop}](\mathbf{q})_{n}^{2}\right)\\
	+g^{2}\left(\sin^{2}(gt)-\cos(gt)\right)\left(\left\Vert \mathbf{w}\right\Vert ^{2}-q_{n}^{2}\right)\\
	+g\sin(gt)r(1+2\cos(gt))
	\end{array}
	\]
	
	hence for some $C>0$, if $\left\Vert \mathrm{grad}[f_{DL}^{pop}](\mathbf{q})\right\Vert <M$
	we have 
	
	\[
	R<CM^{2}\eta^{2}
	\]
	
	and thus choosing $\eta<\frac{(1-\left\Vert \mathbf{w}\right\Vert ^{2})\left\Vert \mathbf{w}\right\Vert \theta c}{CM^{2}}$ we find 
	
	\[
	\left\Vert \mathbf{w}'\right\Vert ^{2}<\left\Vert \mathbf{w}\right\Vert ^{2}-2(1-\left\Vert \mathbf{w}\right\Vert ^{2})\left\Vert \mathbf{w}\right\Vert c \theta \eta
	\]
	
	and in our region of interest $\left\Vert \mathbf{w}'\right\Vert ^{2}<\left\Vert \mathbf{w}\right\Vert ^{2}-\tilde{c} s \theta \eta$
	for some $\tilde{c}>0$ and thus summing over iterations, we obtain
	for some $\tilde{C}_{2}>0$
	
	\begin{equation} \lbl{t2eq}
	t_{2} = \frac{\tilde{C}_{2}}{s \theta \eta}.
	\end{equation}
	
	From Lemma \ref{gradupperlem}, $M=\sqrt{\theta n}$ and thus with a suitably chosen
	$\cetadlpop>0$, $\eta<\frac{\cetadlpop s}{n}$ satisfies the above requirement on $\eta$ as well as the previous requirements, since $\theta < 1$.

	\subsection{Final rate and distance to minimizer}
	Combining these results gives, we find that when initializing in $\mathcal{C}_{\zeta_0}$, the maximal number of iterations required for Riemannian gradient descent to enter $B_{s}^{2}(0)$ is 
	\[
	T\leq t_1 + t_2 < \frac{ \Cdlpop}{\eta\theta}\left(n\log\frac{1}{\zeta_{0}}+\frac{1}{s}\right)
	\]
	
	for some suitably chosen $\Cdlpop$, where $t_1,t_2$ are given in \ref{t1eq},\ref{t2eq}. 
	The probability of such an initialization is given by the probability of initializing in one of the $2n$ possible choices of $\mathcal{C}_{\zeta}$, which is bounded in Lemma \ref{vollem}. 
	
	Once $\mathbf{w}\in B_{s}^{2}(0)$, the distance in $\mathbb{R}^{n-1}$ between $\mathbf{w}$ and a solution to the problem (which is a signed basis vector, given by the point $\mb w= \mb 0$ or an analog on a different symmetric section of the sphere) is no larger than $s$, which in turn implies that the Riemannian distance between $\mb \varphi(\mb w)$ and a solution is no larger than $\cdistfrommin s$ for some $\cdistfrommin > 0$. We note that the conditions on $\mu$ can be satisfied by requiring $\mu < \frac{\cmudlpop \sqrt{\zeta_0}}{n^{5/4}} $. 
	
\end{proof}

\begin{lemma}[Dictionary learning gradient upper bound] \label{gradupperlem}
	The dictionary learning population gradient obeys
	\[
	\left\Vert \nabla_{\mb w}g_{DL}^{pop}(\mb w)\right\Vert \leq \sqrt{2 \theta n}
	\]
	\[
	\left\Vert \mathrm{grad}[f^{pop}_{DL}](\mathbf{q})\right\Vert \leq \sqrt{\theta n}
	\]
	while in the finite sample case 
	\[\left\Vert \nabla_{\mb w}g_{DL}(\mb w)\right\Vert ^{2}\leq\sqrt{2n}\left\Vert \mathbf{X}\right\Vert _{\infty}
	\]
	\[\left\Vert \mathrm{grad}[f_{DL}](\mathbf{q})\right\Vert \leq \sqrt{n}\left\Vert \mathbf{X}\right\Vert _{\infty}\]
	where $\mathbf{X}$ is the data matrix with i.i.d. $BG(\theta)$ entries.
\end{lemma}
\begin{proof}
	Denoting $\mb x \equiv (\overline{\mathbf{x}},x_n)$ we have 
	\[
	\left\Vert \nabla_{\mb w}g_{DL}^{pop}(\mb w)\right\Vert ^{2}=\left\Vert \mathbb{E}\left[\tanh\left(\frac{\mathbf{q}^{\ast}\mathbf{x}}{\mu}\right)\left(\overline{\mathbf{x}}-x_{n}\frac{\mathbf{w}}{q_{n}}\right)\right]\right\Vert ^{2}
	\]
	and using Jensen's inequality, convexity of the $L^2$ norm and the triangle inequality to obtain
	\[
	\leq\mathbb{E}\left[\left\Vert \tanh\left(\frac{\mathbf{q}^{\ast}\mathbf{x}}{\mu}\right)\overline{\mathbf{x}}\right\Vert ^{2}+\left\Vert \tanh\left(\frac{\mathbf{q}^{\ast}\mathbf{x}}{\mu}\right)\left(x_{n}\frac{\mathbf{w}}{q_{n}}\right)\right\Vert ^{2}\right]
	\]
	\[
	\leq\mathbb{E}\left[\left\Vert \overline{\mathbf{x}}\right\Vert ^{2}+\left\Vert x_{n}\frac{\mathbf{w}}{q_{n}}\right\Vert ^{2}\right]\leq 2\theta n
	\]
	while 
	\[
	\left\Vert \mathrm{grad}[f^{pop}_{DL}](\mathbf{q})\right\Vert \leq\left\Vert \nabla f^{pop}_{DL}(\mathbf{q})\right\Vert 
	\]
	\[
	=\left\Vert \mathbb{E}\left[\tanh\left(\frac{\mathbf{q}^{\ast}\mathbf{x}}{\mu}\right)\mathbf{x}\right]\right\Vert \leq \sqrt{\theta n}
	\]
	Similarly, in the finite sample size case one obtains 
	\[\left\Vert \nabla_{\mb w}g_{DL}(\mb w)\right\Vert ^{2}\leq\frac{1}{p}\underset{i=1}{\overset{p}{\sum}}\left\Vert \overline{\mathbf{x}}^{i}\right\Vert ^{2}+\left\Vert x_{n}^{i}\frac{\mathbf{w}}{q_{n}}\right\Vert ^{2}\leq2 n \left\Vert \mathbf{X}\right\Vert _{\infty}^{2}\]
	\[\left\Vert \mathrm{grad}[f_{DL}](\mathbf{q})\right\Vert \leq\frac{1}{p}\underset{i=1}{\overset{p}{\sum}}\left\Vert \tanh\left(\frac{\mathbf{q}^{\ast}\mathbf{x}^{i}}{\mu}\right)\mathbf{x}^{i}\right\Vert 
	\]
	\[
	\leq \sqrt{n} \left\Vert \mathbf{X}\right\Vert _{\infty}\]
	
\end{proof}

%\begin{proof}[\begin{minipage}{3.25in}\bf{Proof of Theorem \ref{dlratethm}}: (Gradient descent convergence rate for dictionary learning)\end{minipage}] \label{dlratethmp}
\begin{proof}[\bf{Proof of Theorem \ref{dlratethm}}: (Gradient descent convergence rate for dictionary learning)] \label{dlratethmp}	
	The proof will follow exactly that of Lemma \ref{newratedllem}, with the finite sample size fluctuations decreasing the guaranteed change in $\zeta$ or $|| \mb w ||$ at every iteration (for the initial and final stages respectively) which will adversely affect the bounds.
	
	\subsection{$\mathcal{C}_{\zeta_0} \backslash B^2_{1/20\sqrt{5}}(0)$}
	To control the fluctuations in the gradient projection, we choose  
	\[
	y(\theta,\zeta_{0})=\frac{\zeta_{0}c_{DL}}{2(8000(n-1))^{3/2}}
	\]
	which can be satisfied by choosing $y(\theta,\zeta_{0})=\frac{\cz \theta(1-\theta)\zeta_{0}}{n^{3/2}}$ for an appropriate $\cz > 0$ . According to Lemma \ref{uniflem}, with probability greater than $\mathbb{P}_{y}$ we then have 
	\[\left\vert \begin{array}{c}
	\mathbf{u}^{(i)\ast}\mathrm{grad}[f_{DL}](\mathbf{q}(\mathbf{w}))\\
	-\mathbb{E}\left[\mathbf{u}^{(i)\ast}\mathrm{grad}[f_{DL}](\mathbf{q}(\mathbf{w}))\right]
	\end{array}\right\vert \leq y(\theta,\zeta)\]
	
	With the same condition on $\mu$ as in Lemma \ref{newratedllem}, combined with the uniformized bound on finite sample fluctuations, we have that at every point in this set
	\[\mathbf{u}^{(i)\ast}\mathrm{grad}[f_{DL}^{pop}](\mathbf{q}(\mathbf{w}))\geq\frac{c_{DL}}{2(8000(n-1))^{3/2}}\]
	. According to Lemma \ref{gradupperlem} the Riemannian gradient norm is bounded by $M=\sqrt{n}\left\Vert \mathbf{X}\right\Vert _{\infty}$. Choosing $r,b$ as in Lemma \ref{newratedllem}, we require $\eta<\frac{1}{360 \left\Vert \mathbf{X}\right\Vert _{\infty}\sqrt{5 n(n-1)}}=\frac{b-r}{3M}$ and obtain from Lemma \ref{zetalem}
	\[\zeta'\geq\zeta\left(1+\frac{\sqrt{n}c_{DL}}{4(8000(n-1))^{3/2}}\eta\right)\]
	\begin{equation} \lbl{t12eq}
	t_{1}=\frac{\log(8/\zeta_0)}{\log\left(1+\frac{\sqrt{n}c_{DL}}{4(8000(n-1))^{3/2}}\eta\right)}
	\end{equation}

	\subsection{$B^2_{1/20\sqrt{5}}(0) \backslash B^2_{s}(0)$}
	From Theorem 2 in \cite{sun2017complete} there are numerical constants $c_b,c_\star$ such that in this region 
	\[
	\frac{\mb w^{\ast}\nabla_{\mb w}g_{DL}(\mb w)}{\left\Vert \mb w\right\Vert } = \frac{\mathbf{w}^{\ast}\left(\frac{\partial \mb \varphi}{\partial\mathbf{w}}\right)^{\ast}\mathrm{grad}[f](\mathbf{q}(\mathbf{w}))}{\left\Vert \mathbf{w}\right\Vert } \geq c_\star \theta
	\]
	with probability $\mathbb{P}>1-c_{b}p^{-6} $. Following the same analysis as in Lemma \ref{newratedllem}, since from Lemma \ref{gradupperlem} the norm of the gradient gradient is bounded by $\sqrt{n} || \mathbf{X}|| _{\infty}$ we require $\eta<\frac{(1-\left\Vert \mathbf{w}\right\Vert ^{2})\left\Vert \mathbf{w}\right\Vert \theta c_\star}{C n || \mathbf{X}|| _{\infty}^{2}}$ which is satisfied by requiring $\eta<\frac{\tilde{c}\theta s}{n ||\mathbf{X}||_{\infty}^{2}}$ for some chosen $\tilde{c}>0$. We then obtain 
	
	\begin{equation} \lbl{t32eq}
	t_{3} = \frac{C_{2}}{s \theta \eta}
	\end{equation}
	
	for a suitably chosen $C_2 > 0$. 
	
	%------------ old
	%Since $\eta<\frac{s}{\left\Vert \mathbf{X}\right\Vert _{\infty}\sqrt{2n}}\frac{\theta}{20\sqrt{2\pi}}$ following a similar analysis to the proof of Lemma \ref{newratedllem} we have for the gradient step on the domain of the chart
	%\[
	%\left\Vert \mb w_c^{(t+1)}\right\Vert ^{2}\leq\left\Vert \mb w^{(t)}\right\Vert ^{2}(1-20\sqrt{5} c_\star \theta \eta )
	%\]
	%and requiring $\eta<\frac{5\sqrt{5}c_{\star}\theta s}{4n\left\Vert \mathbf{X}\right\Vert _{\infty}^{2}}$ gives 
	%\[
	%2\frac{M^{2}\eta^{2}}{s}<5\sqrt{5}c_{\star}\theta\eta
	%\]
	%and since $5\sqrt{5}c_{\star}\theta \eta < 1$ we have 
	%\[2\frac{M^{2}\eta^{2}}{s}+\frac{M^{4}\eta^{4}}{s^2}<10\sqrt{5}c_{\star}\theta\eta
	%\]
	%which, according to Lemma \ref{riemnormlem} gives for the Riemannian gradient update
	%\[
	%\left\Vert \mathbf{w}_{r}^{(t+1)}\right\Vert ^{2}\leq\left\Vert \mathbf{w}^{(t)}\right\Vert ^{2}(1-10\sqrt{5}c_{\star}\theta\eta)
	%\]
	%hence
	%\begin{equation} \lbl{t32eq}
	%t_{3} = \frac{2\log\left( 20 \sqrt{5} s\right)}{\log(1-10\sqrt{5}c_{\star}\theta\eta)}
	%\end{equation}
	\subsection{Final rate and distance to minimizer}
	The final bound on the rate is obtained by summing over the terms for the three regions as in the population case, and convergence is again to a distance of less than $\cdistfrommin s$ from a local minimizer. The probability of achieving this rate is obtained by taking a union bound over the probability of initialization in $\mathcal{C}_{\zeta_0}$ (given in Lemma \ref{vollem}) and the probabilities of the bounds on the gradient fluctuations holding (from Lemma \ref{uniflem} and \cite{sun2017complete}). Note that the fluctuation bound events imply by construction the event $\mathcal{E}_{\infty} = \{1\leq\left\Vert \mathbf{X}\right\Vert _{\infty}\leq4\sqrt{\log(np)}\}$ hence we can replace $\left\Vert \mathbf{X}\right\Vert _{\infty}$ in the conditions on $\eta$ above by $4\sqrt{\log(np)}$.
	The conditions on $\eta,\mu$ can be satisfied by requiring $\eta < \frac{\cetadl \theta s}{n \log{np}}, \mu < \frac{\cmudl \sqrt{\zeta_0}}{n^{5/4}}$ for suitably chosen $\cetadl,\cmudl > 0$. The bound on the number of iterations can be simplified to the form in the theorem statement as in the population case.
\end{proof}

\section{Generalized Phase Retrieval} \label{appgpr}

We show below that negative curvature normal to stable manifolds of saddle points in strict saddle functions is a feature that is found not only in dictionary learning, and can be used to obtain efficient convergence rates for other nonconvex problems as well, by presenting an analysis of generalized phase retrieval that is along similar lines to the dictionary learning analysis. We stress that this contribution is not novel since a more thorough analysis was carried out by \cite{chen2018gradient}. The resulting rates are also suboptimal, and pertain only to the population objective.

Generalized phase retrieval is the problem of recovering a vector
$\mathbf{x}\in\mathbb{C}^{n}$ given a set of magnitudes of projections
$y_{k}=\left\vert \mathbf{x}^{\ast}\mathbf{a}_{k}\right\vert $ onto
a known set of vectors $\mathbf{a}_{k}\in\mathbb{C}^{n}$. It arises in numerous domains including microscopy \cite{miao2002high}, acoustics \cite{balan2006signal}, and quantum mechanics \cite{corbett2006pauli} (see \cite{shechtman2015phase} for a review). Clearly
$\mathbf{x}$ can only be recovered up to a global phase. We consider
the setting where the elements of every $\mathbf{a}_{k}$ are i.i.d.
complex Gaussian, (meaning $(a_{k})_{j}=u+iv$ for
$u,v\sim\mathcal{N}(0,1/\sqrt{2})$).
We analyze the least squares formulation of the problem \cite{candes2015phase}
%\jw{some citation here, e.g., the candes li soltanolkotabi paper. somewhere we should also cite https://arxiv.org/abs/1706.08167 }, 
given by 
\[
\underset{\mathbf{z}\in\mathbb{C}^{n}}{\min}f(\mathbf{z})=\frac{1}{2p}\underset{k=1}{\overset{p}{\sum}}\left(y_{k}^{2}-\left\vert \mathbf{z}^{\ast}\mathbf{a}_{k}\right\vert ^{2}\right)^{2}.
\]
Taking the expectation (large $p$ limit) of the above objective and organizing its derivatives using Wirtinger calculus \cite{kreutz2009complex}, we obtain
\begin{equation} \label{gprobjeq}
\mathbb{E}[f]=\left\Vert \mathbf{x}\right\Vert ^{4}+\left\Vert \mathbf{z}\right\Vert ^{4}-\left\Vert \mathbf{x}\right\Vert ^{2}\left\Vert \mathbf{z}\right\Vert ^{2}-\left\vert \mathbf{x}^{\ast}\mathbf{z}\right\vert ^{2}
\end{equation}
\[
\nabla\mathbb{E}[f]=\left[\begin{array}{c}
\nabla_{\mathbf{z}}\mathbb{E}[f]\\
\nabla_{\overline{\mathbf{z}}}\mathbb{E}[f]
\end{array}\right]
\]
\[
=\left[\begin{array}{c}
\left((2\left\Vert \mathbf{z}\right\Vert ^{2}-\left\Vert \mathbf{x}\right\Vert ^{2})\mathbf{I}-\mathbf{xx}^{\ast}\right)\mathbf{z}\\
\left((2\left\Vert \mathbf{z}\right\Vert ^{2}-\left\Vert \mathbf{x}\right\Vert ^{2})\mathbf{I}-\mathbf{\overline{x}x}^{T}\right)\mathbf{\overline{z}}
\end{array}\right].
\]
For the remainder of this section, we analyze this objective, leaving the consideration of finite sample size effects to future work. 

\subsection{The geometry of the objective}
In \cite{sun2016geometric} it was shown that aside from the manifold of minima 
\[
\breve{A} \;\equiv\; \mathbf{x}e^{i\theta},
\]
the only critical points of $\mathbb{E}[f]$ are a maximum at $\mathbf{z}=\mb 0$
and a manifold of saddle points given by 
\[
\Anc \; \backslash \; \{ \mb 0\} \;\equiv\; \left\{ \mathbf{z} \; \middle| \; \mathbf{z}\in W,\left\Vert \mathbf{z}\right\Vert =\frac{\left\Vert \mathbf{x}\right\Vert }{\sqrt{2}} \right\}
\]
where $W\equiv\{\mathbf{z}|\mathbf{z}^{\ast}\mathbf{x}=0\}$.
We decompose $\mathbf{z}$ as 
\begin{equation} \label{zdefeq}
\mathbf{z}=\mathbf{w}+\zeta e^{i\phi}\frac{\mathbf{x}}{\left\Vert \mathbf{x}\right\Vert },
\end{equation}
where $\zeta>0,\mathbf{w}\in W$. This gives $\left\Vert \mathbf{z}\right\Vert ^{2}=\left\Vert \mathbf{w}\right\Vert ^{2}+\zeta^{2}$.
The choice of $\mathbf{w},\zeta,\phi$ is unique up to factors of
$2\pi$ in $\phi$, as can be seen by taking an inner product with
$\mathbf{x}$. Since the gradient decomposes as follows:
\[
\nabla_{\mathbf{z}}\mathbb{E}[f]=\left(2\left\Vert \mathbf{z}\right\Vert ^{2}I-\left\Vert \mathbf{x}\right\Vert ^{2}I-\mathbf{x}\mathbf{x}^{\ast}\right)(\mathbf{w}+\zeta e^{i\phi}\frac{\mathbf{x}}{\left\Vert \mathbf{x}\right\Vert })
\]
\begin{equation} \label{gradgpreq}
=\left(2\left\Vert \mathbf{z}\right\Vert ^{2}-\left\Vert \mathbf{x}\right\Vert ^{2}\right)\mathbf{w}+2\zeta e^{i\phi}\left(\left\Vert \mathbf{z}\right\Vert ^{2}-\left\Vert \mathbf{x}\right\Vert ^{2}\right)\frac{\mathbf{x}}{\left\Vert \mathbf{x}\right\Vert } 
\end{equation}
the directions $e^{i\phi}\frac{\mathbf{x}}{\left\Vert \mathbf{x}\right\Vert },\frac{\mathbf{w}}{\left\Vert \mathbf{w}\right\Vert }$
are unaffected by gradient descent and thus the problem reduces to
a two-dimensional one in the space $(\zeta,\left\Vert \mathbf{w}\right\Vert )$. Note also that the objective for this two-dimensional problem is a Morse function, despite the fact that in the original space there was a manifold of saddle points. 
It is also clear from this decomposition of the gradient that the
stable manifolds of the saddles are precisely the set $W$. 

It is evident from \ref{gradgpreq} that the dispersive property does not hold globally in this case. For $\mb z \notin B_{|| \mb x ||}$ we see that gradient descent will cause $\zeta$ to decrease, implying positive curvature normal to the stable manifolds of the saddles. This is a consequence of the global geometry of the objective. Despite this, in the region of the space that is more "interesting", namely $B_{|| \mb x ||}$, we do observe the dispersive property, and can use it to obtain a convergence rate for gradient descent. 

We define a set that contains the regions that feeds into small gradient regions
around saddle points within $B_{|| \mb x ||}$ by 
\[
\overline{Q}_{\zeta_{0}} \equiv \{\mathbf{z}(\zeta,\left\Vert \mathbf{w}\right\Vert )|\zeta\leq\zeta_{0}\}.
\]
We will show that, as in the case of orthogonal dictionary learning,
we can both bound the probability of initializing in (a subset of)
the complement of $\overline{Q}_{\zeta_{0}}$ and obtain a rate for convergence
of gradient descent in the case of such an initialization. \footnote{$\overline{Q}_{\zeta_{0}}$ is equivalent to the complement of the set $C_\zeta$ used in the analysis of the separable objective and dictionary learning. }

We now define four regions of the space which will be used in the analysis
of gradient descent:
\begin{eqnarray*}
	S_{1}&\equiv& \left\{\mathbf{z} \;\middle | \; \left\Vert \mathbf{z}\right\Vert ^{2}\leq\tfrac{1}{2} \left\Vert \mathbf{x}\right\Vert ^{2} \right \} \\
	S_{2}&\equiv& \left\{\mathbf{z}\; \middle | \; \tfrac{1}{2} \left\Vert \mathbf{x}\right\Vert ^{2}<\left\Vert \mathbf{z}\right\Vert ^{2}\leq(1-c)\left\Vert \mathbf{x}\right\Vert ^{2}\right\} \\
	S_{3} &\equiv& \left\{\mathbf{z}\; \middle | \; (1-c)\left\Vert \mathbf{x}\right\Vert ^{2}<\left\Vert \mathbf{z}\right\Vert ^{2}\leq\left\Vert \mathbf{x}\right\Vert ^{2} \right \} \\
	S_{4}&\equiv& \left\{\mathbf{z}\; \middle | \; \left\Vert \mathbf{x}\right\Vert ^{2}<\left\Vert \mathbf{z}\right\Vert ^{2}\leq(1+c)\left\Vert \mathbf{x}\right\Vert ^{2} \right \}
\end{eqnarray*}
defined for some $c<\frac{1}{4}$. These are shown in Figure \ref{gprsetsfig}. 

\begin{figure} \
	\includegraphics[width=3in]{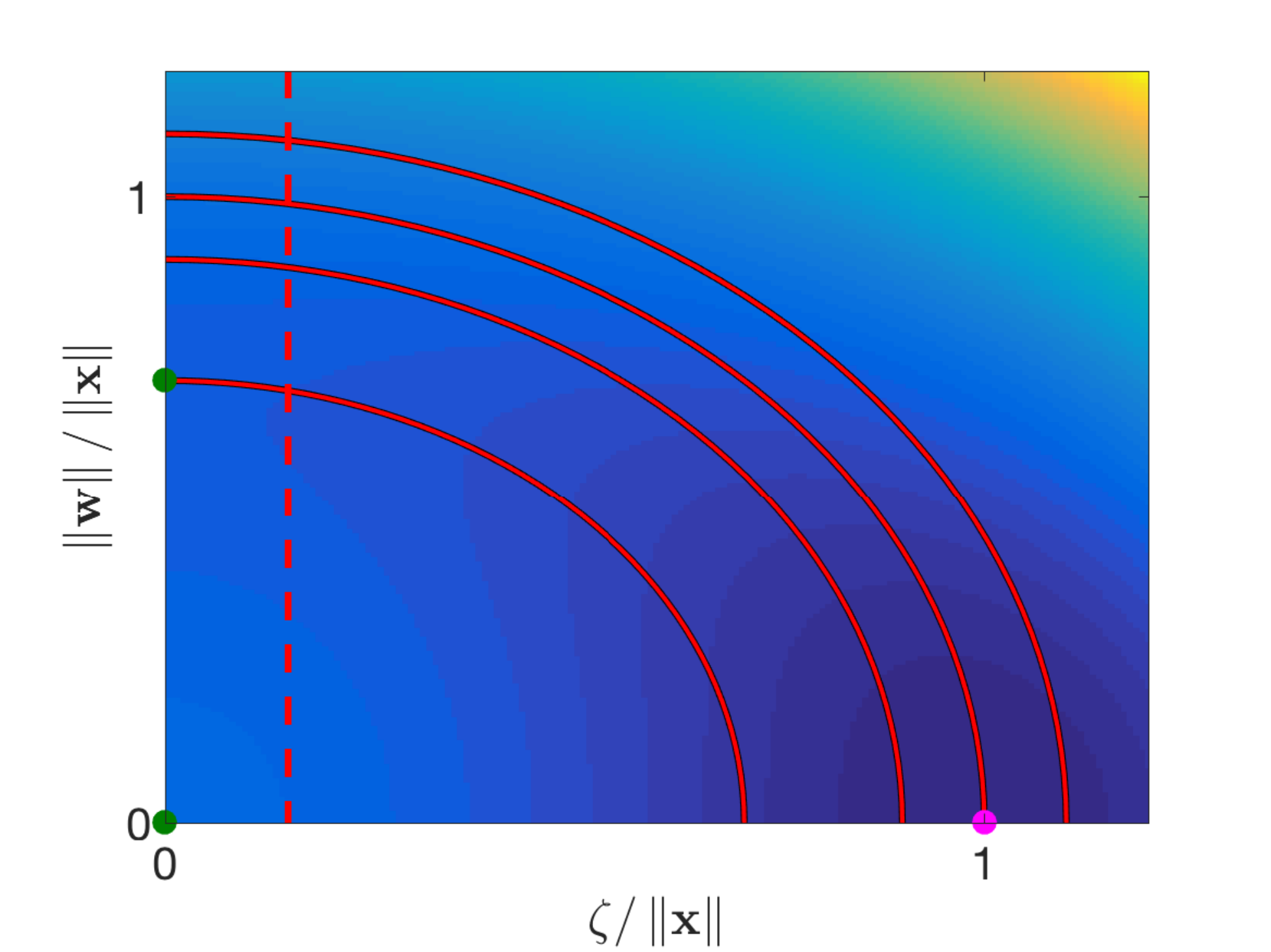}
	\caption{The projection of the objective of generalized phase retrieval on the $(\frac{\zeta}{\left\Vert \mathbf{x}\right\Vert },\frac{\left\Vert \mathbf{w}\right\Vert }{\left\Vert \mathbf{x}\right\Vert })$ plane. The full red curves are the boundaries between the sets $S_1,S_2,S_3,S_4$ used in the analysis. The dashed red line is the boundary of the set $\overline{Q}_{\zeta_{0}}$ that contains small gradient regions around critical points that are not minima. The maximizer and saddle point are shown in dark green, while the minimizer is in pink.} 
	\label{gprsetsfig}
\end{figure}

We now define 
\begin{equation} \label{ztageq}
\mathbf{z'}\equiv\mathbf{z}-\eta\nabla_{\mathbf{z}}\mathbb{E}[f]\equiv\mathbf{w}'+\zeta'e^{i\phi}\frac{\mathbf{x}}{\left\Vert \mathbf{x}\right\Vert }
\end{equation}
and using \ref{gradgpreq} obtain
\begin{subequations} \label{zweqs}
	\begin{align}
	\zeta' =& \left(1-2\eta(\left\Vert \mathbf{z}\right\Vert ^{2}-\left\Vert \mathbf{x}\right\Vert ^{2})\right)\zeta \\
	\left\Vert \mathbf{w}'\right\Vert =& \left(1-\eta\left(2\left\Vert \mathbf{z}\right\Vert ^{2}-\left\Vert \mathbf{x}\right\Vert ^{2}\right)\right)\left\Vert \mathbf{w}\right\Vert.
	\end{align}
\end{subequations}
These are used to find the change in $\zeta,\left\Vert \mathbf{w}\right\Vert $
at every iteration in each region:
\begin{subequations} \label{sineqs}
	\begin{align}
	\text{\bf On $S_1$:} \qquad &\zeta' \geq (1+\eta\left\Vert \mathbf{x}\right\Vert ^{2})\zeta\\
	& \left\Vert \mathbf{w}'\right\Vert \geq \left\Vert \mathbf{w}\right\Vert \\
	\text{\bf On $S_2$:} \qquad & \zeta' \geq (1+2c\eta\left\Vert \mathbf{x}\right\Vert ^{2})\zeta\\
	&\left\Vert \mathbf{w}'\right\Vert \leq \left\Vert \mathbf{w}\right\Vert \\
	\text{\bf On $S_3$:} \qquad & \left(1-\eta\left\Vert \mathbf{x}\right\Vert ^{2}\right)\left\Vert \mathbf{w}\right\Vert \leq \left\Vert \mathbf{w}'\right\Vert \nonumber \\
	& \qquad  \leq \left(1-(1-2c)\eta\left\Vert \mathbf{x}\right\Vert ^{2}\right)\left\Vert \mathbf{w}\right\Vert \\
	& \zeta \leq\zeta'\leq(1+2c\eta\left\Vert \mathbf{x}\right\Vert ^{2})\zeta \\
	\text{\bf On $S_4$:} \qquad &
	\left(1-(1+2c)\eta\left\Vert \mathbf{x}\right\Vert ^{2}\right)\left\Vert \mathbf{w}\right\Vert \leq\left\Vert \mathbf{w}'\right\Vert \nonumber \\
	& \qquad \leq\left(1-\eta\left\Vert \mathbf{x}\right\Vert ^{2}\right)\left\Vert \mathbf{w}\right\Vert 
	\\
	& (1-2c\eta\left\Vert \mathbf{x}\right\Vert ^{2})\zeta\leq\zeta'\leq\zeta
	\end{align}
\end{subequations}

\subsection{Behavior of gradient descent in $\cup_{i=1}^4 S_{i}$}
We now show that gradient descent initialized in $S_1 \backslash \overline{Q}_{\zeta_{0}}$ cannot exit $\cup_{i=1}^4 S_{i}$ or enter $\overline{Q}_{\zeta_{0}}$. Lemma \ref{ziterlem2} guarantees that gradient descent initialized in $\cup_{i=1}^4 S_{i}$ remains in this set. 
From equation \ref{sineqs} we see that a gradient descent step can only decrease $\zeta$ if $\mathbf{z}\in S_{4}$. Under the mild assumption $\zeta_{0}^{2}<\frac{7}{16}\left\Vert \mathbf{x}\right\Vert ^{2}$ we are guaranteed from Lemma \ref{wziterlem} that at every iteration $\zeta\geq\zeta_{0}$. Thus the region with $\zeta<\zeta{}_{0}$ can only be entered if gradient descent is initialized in it. It follows that initialization in $S_1 \backslash \overline{Q}_{\zeta_{0}}$ rules out entering $\overline{Q}_{\zeta_{0}}$ at any future iteration of gradient descent. Since this guarantees that regions that feed into small gradient regions are avoided, an efficient convergence rate can again be obtained. 

\subsection{Convergence rate}
\begin{theorem}[Gradient descent convergence rate for generalized phase retrieval] \label{gprthm}
	Gradient descent on \ref{gprobjeq} with step size $\eta<\frac{\sqrt{c}}{4\left\Vert \mathbf{x}\right\Vert ^{2}},c<\frac{1}{4}$, initialized uniformly in $S_1$ converges to a point $\mb z$ such that $\mathrm{dist}(\mathbf{z},\breve{A})<\sqrt{5c}\left\Vert \mathbf{x}\right\Vert $ in  
	\[
	\begin{array}{c}
	T<\frac{\log\left(\frac{\left\Vert \mathbf{x}\right\Vert }{\zeta\sqrt{2}}\right)}{\log(1+\eta\left\Vert \mathbf{x}\right\Vert ^{2})}+\frac{\log(2)}{2\log(1+2c\eta\left\Vert \mathbf{x}\right\Vert ^{2})}\\
	+\frac{\log(2c)\log(\frac{4}{\sqrt{7}})}{\log\left(1-\left(1-2c\right)\eta\left\Vert \mathbf{x}\right\Vert ^{2}\right)\log(1+2c\eta\left\Vert \mathbf{x}\right\Vert ^{2})}
	\end{array}
	\]
	iterations with probability 
	\[
	\mathbb{P} \ge 1 - \sqrt{\frac{8}{\pi}} \, \mathrm{erf}\left(\frac{\sqrt{2n}}{\left\Vert \mathbf{x}\right\Vert }\zeta \right),
	\]
\end{theorem}

\begin{proof} Please see Appendix \ref{gprthmp}.
\end{proof} 

We find that in order to prevent the failure probability from approaching
1 in a high dimensional setting, if we assume that $\left\Vert \mathbf{x}\right\Vert $
does not depend on $n$ we require that $\zeta$ scale like $\frac{1}{\sqrt{n}}$. This is simply the consequence of the well-known concentration of
volume of a hypersphere around the equator. Even
with this dependence the convergence rate itself depends only logarithmically on dimension, and this again is a consequence of the logarithmic dependence of $\zeta$ due to the curvature properties of the objective. 

\begin{lemma} \label{wziterlem}
	For any iterate $\mathbf{z}$ of gradient descent on \ref{gprobjeq}, assuming $\eta<\frac{\sqrt{c}}{4\left\Vert \mathbf{x}\right\Vert ^{2}},c<\frac{1}{4}$ and defining $\zeta'$ as in \ref{ztageq}, we have 
	i) 
	\[
	\mathbf{z}\in \underset{i=1}{\overset{4}{\bigcup}}S_{i}\Rightarrow\left\Vert \mathbf{w}\right\Vert ^{2}\leq\frac{\left\Vert \mathbf{x}\right\Vert ^{2}}{2}
	\]
	ii) 
	\[
	\mathbf{z}\in S_{4}\Rightarrow\zeta'^{2}\geq\frac{7}{16}\left\Vert \mathbf{x}\right\Vert ^{2}
	\]
\end{lemma}

%Proof: \ref{wziterlemp}

\begin{proof}[\bf{Proof of Lemma \ref{wziterlem}}] \label{wziterlemp} 
	
	i) From \ref{sineqs} we see that in $\underset{i=2}{\overset{4}{\bigcup}}S_{i}$ the quantity $\left\Vert \mathbf{w}\right\Vert ^{2}$
	cannot increase, hence this can only happen in $S_{1}$. We show that
	for some $\mathbf{z}\in S_{1}$, a point with $\left\Vert \mathbf{w}\right\Vert =(1-\varepsilon)\frac{\left\Vert \mathbf{x}\right\Vert }{\sqrt{2}},\varepsilon<1$
	cannot reach a point with $\left\Vert \mathbf{w}\right\Vert '=\frac{\left\Vert \mathbf{x}\right\Vert }{\sqrt{2}}$
	by a gradient descent step. This would mean 
	
	\[
	\left(1-\eta\left(2\left\Vert \mathbf{w}\right\Vert ^{2}+2\zeta^{2}-\left\Vert \mathbf{x}\right\Vert ^{2}\right)\right)\left\Vert \mathbf{w}\right\Vert 
	\]
	\[
	=\left(1-\eta\left((1-\varepsilon)^{2}\left\Vert \mathbf{x}\right\Vert ^{2}+2\zeta^{2}-\left\Vert \mathbf{x}\right\Vert ^{2}\right)\right)(1-\varepsilon)\frac{\left\Vert \mathbf{x}\right\Vert }{\sqrt{2}}
	\]
	\[
	=\frac{\left\Vert \mathbf{x}\right\Vert }{\sqrt{2}}
	\]
	
	and since $\zeta^{2}\geq0$ this implies
	
	\[
	\left(1+\varepsilon\eta\left\Vert \mathbf{x}\right\Vert ^{2}(2-\varepsilon)\right)(1-\varepsilon)\geq1
	\]
	
	by considering the product of these two factors, this in turn implies
	
	\[
	\frac{1}{2b}(2-\varepsilon)\geq\eta\left\Vert \mathbf{x}\right\Vert ^{2}(2-\varepsilon)\geq1
	\]
	
	where we have used $\eta<\frac{\sqrt{c}}{b\left\Vert \mathbf{x}\right\Vert ^{2}},c<\frac{1}{4}$.
	Thus if we choose $b=4$ this inequality cannot be satisfied. 
	
	Additionally, if we initialize in $S_{1}\cap \overline{Q}_{\zeta_{0}}$ then we cannot initialize
	at a point where $\left\Vert \mathbf{w}\right\Vert '=\frac{\left\Vert \mathbf{x}\right\Vert }{\sqrt{2}}$ and hence the inequality is strict. %This also means there is no choice of $\varepsilon$ for which $\left\Vert \mathbf{w}\right\Vert ' > \frac{\left\Vert \mathbf{x}\right\Vert }{\sqrt{2}}$.
	
	ii) Since only a step from $S_{4}$ can decrease $\zeta$, we have that
	for the initial point $\left\Vert \mathbf{z}\right\Vert ^{2}>\left\Vert \mathbf{x}\right\Vert^2 $.
	Combined with $\left\Vert \mathbf{w}\right\Vert ^{2}\leq\frac{\left\Vert \mathbf{x}\right\Vert ^{2}}{2}$
	this gives 
	
	\[
	\zeta^{2}\geq\frac{\left\Vert \mathbf{x}\right\Vert ^{2}}{2}
	\]
	
	and using the lower bound $(1-2\eta\left\Vert \mathbf{x}\right\Vert ^{2}c)\zeta\leq\zeta'$
	we obtain 
	
	\[
	\zeta'^{2}\geq\frac{\left\Vert \mathbf{x}\right\Vert ^{2}}{2}(1-2\eta\left\Vert \mathbf{x}\right\Vert ^{2}c)^{2}\geq\frac{\left\Vert \mathbf{x}\right\Vert ^{2}}{2}(1-4\eta\left\Vert \mathbf{x}\right\Vert ^{2}c)
	\]
	\[
	\geq(1-\frac{1}{2b})\frac{\left\Vert \mathbf{x}\right\Vert ^{2}}{2}
	\]
	
	where in the last inequality we used $c<\frac{1}{4},\eta<\frac{\sqrt{c}}{b\left\Vert \mathbf{x}\right\Vert ^{2}}$.
	Choosing $b=4$ gives 
	
	\[
	\zeta'^{2}\geq\frac{7}{16}\left\Vert \mathbf{x}\right\Vert ^{2}
	\]
	
	If we require $\zeta_{0}^{2}<\frac{7}{16}\left\Vert \mathbf{x}\right\Vert ^{2}$
	this also ensures that the next iterate cannot lie in the small gradient
	regions around the stable manifolds of the saddles. 
\end{proof}

\begin{lemma} \label{ziterlem2}
	Defining $\mathbf{z'}$ as in \ref{ztageq}, under the conditions of Lemma \ref{wziterlem} and we have  
	
	i) 
	\[
	\mathbf{z}\in\underset{i=2}{\overset{4}{\bigcup}}S_{i}\Rightarrow\mathbf{z}'\in\underset{i=2}{\overset{4}{\bigcup}}S_{i}
	\]
	ii) 
	\[
	\mathbf{z}\in S_{1}\Rightarrow\mathbf{z}'\in S_{1}\cup S_{2}
	\]
\end{lemma}
%Proof: \ref{ziterlem2p}

\begin{proof}[\bf{Proof of Lemma \ref{ziterlem2}}] \label{ziterlem2p}
	We use the fact that for the next iterate we have 
	
	\begin{equation} \lbl{zeesqeq}
	\begin{array}{c}
	\left\Vert \mathbf{z}'\right\Vert ^{2}=\left(1-\eta(2\left\Vert \mathbf{z}\right\Vert ^{2}-\left\Vert \mathbf{x}\right\Vert ^{2})\right)^{2}\left\Vert \mathbf{w}\right\Vert ^{2}\\
	+\left(1-2\eta(\left\Vert \mathbf{z}\right\Vert ^{2}-\left\Vert \mathbf{x}\right\Vert ^{2})\right)^{2}\zeta^{2}
	\end{array}
	\end{equation}
	
	We will also repeatedly use $\eta<\frac{\sqrt{c}}{b\left\Vert \mathbf{x}\right\Vert ^{2}},c<\frac{1}{4}$
	and $\mathbf{z}\in \underset{i=1}{\overset{4}{\bigcup}}S_{i}\Rightarrow\left\Vert \mathbf{w}\right\Vert ^{2}\leq\frac{\left\Vert \mathbf{x}\right\Vert ^{2}}{2}$
	which is a shown in Lemma \ref{wziterlem}.
	
	\subsection{$\mathbf{z}\in S_{3}\Rightarrow\mathbf{z}'\in \underset{i=2}{\overset{4}{\bigcup}}S_{i}$}
	
	We want to show $\frac{\left\Vert \mathbf{x}\right\Vert ^{2}}{2}\underset{(1)}{<}\left\Vert \mathbf{z}'\right\Vert ^{2}\underset{(2)}{\leq}(1+c)\left\Vert \mathbf{x}\right\Vert ^{2}$.
	
	1) We have $\mathbf{z}\in S_{3}\Rightarrow\left\Vert \mathbf{z}\right\Vert ^{2}=(1-\varepsilon)\left\Vert \mathbf{x}\right\Vert ^{2}$
	for some $\varepsilon\leq c$ and using \ref{zeesqeq} we must show
	\[
	\frac{\left\Vert \mathbf{x}\right\Vert ^{2}}{2}\leq\begin{array}{c}
	\left(1-\eta\left\Vert \mathbf{x}\right\Vert ^{2}(1-2\varepsilon)\right)^{2}\left\Vert \mathbf{w}\right\Vert ^{2}\\
	+\left(1+2\eta\left\Vert \mathbf{x}\right\Vert ^{2}\varepsilon\right)^{2}\zeta^{2}
	\end{array}
	\]
	or equivalently
	\[
	A\equiv\varepsilon-\frac{\left\Vert \mathbf{x}\right\Vert ^{2}}{2}
	\]
	\[
	\leq\eta\left\Vert \mathbf{x}\right\Vert ^{2}\left[\begin{array}{c}
	\left(-2(1-2\varepsilon)+(1-2\varepsilon)^{2}\eta\left\Vert \mathbf{x}\right\Vert ^{2}\right)\left\Vert \mathbf{w}\right\Vert ^{2}\\
	+4\left(\varepsilon+\varepsilon^{2}\eta\left\Vert \mathbf{x}\right\Vert ^{2}\right)\zeta^{2}
	\end{array}\right]\equiv B
	\]
	
	and using $\eta<\frac{\sqrt{c}}{b\left\Vert \mathbf{x}\right\Vert ^{2}},c<\frac{1}{4}$
	
	\[
	\frac{-\left\Vert \mathbf{x}\right\Vert ^{2}}{b}<\frac{-2\left\Vert \mathbf{x}\right\Vert \sqrt{c}}{b}<-2\eta\left\Vert \mathbf{x}\right\Vert ^{4}\leq B
	\]
	
	while on the other hand 
	\[
	A\leq c-\frac{\left\Vert \mathbf{x}\right\Vert ^{2}}{2}<-\frac{\left\Vert \mathbf{x}\right\Vert ^{2}}{4}
	\]
	thus picking $b=4$ guarantees the desired result. 
	
	2) By a similar argument, $\left\Vert \mathbf{z}'\right\Vert ^{2}\leq(1+c)\left\Vert \mathbf{x}\right\Vert ^{2}$
	is equivalent to
	\[
	A\equiv\eta\left\Vert \mathbf{x}\right\Vert ^{2}\left[\begin{array}{c}
	\left(-2(1-2\varepsilon)+\eta\left\Vert \mathbf{x}\right\Vert ^{2}(1-2\varepsilon)^{2}\right)\left\Vert \mathbf{w}\right\Vert ^{2}\\
	+4\left(\varepsilon+\eta\left\Vert \mathbf{x}\right\Vert ^{2}\varepsilon^{2}\right)\zeta^{2}
	\end{array}\right]
	\]
	\[
	\leq\left\Vert \mathbf{x}\right\Vert ^{2}(c+\varepsilon)\equiv B
	\]
	
	. Since $\left\Vert \mathbf{w}\right\Vert ^{2}\leq\frac{\left\Vert \mathbf{x}\right\Vert ^{2}}{2}$
	and $\left\Vert \mathbf{z}\right\Vert ^{2}\leq\left\Vert \mathbf{x}\right\Vert ^{2}\Rightarrow\zeta^{2}\leq\frac{\left\Vert \mathbf{x}\right\Vert ^{2}}{2}$
	we obtain
	\[
	A\leq\eta\left[\eta\left\Vert \mathbf{x}\right\Vert ^{4}+4\left(\left\Vert \mathbf{x}\right\Vert ^{2}\varepsilon+\eta\left\Vert \mathbf{x}\right\Vert ^{4}\varepsilon^{2}\right)\right]\frac{\left\Vert \mathbf{x}\right\Vert ^{2}}{2}
	\]
	
	\[
	<\frac{1}{2b}\left[\frac{1}{b}+2\left(1+\frac{1}{8b}\right)\right]c\left\Vert \mathbf{x}\right\Vert ^{2}
	\]
	
	. If we choose $b=4$ we thus have $A<B$ which implies 
	
	\[
	\left\Vert \mathbf{z}'\right\Vert ^{2}<(1+c)\left\Vert \mathbf{x}\right\Vert ^{2}
	\]

	\subsection{$\mathbf{z}\in S_{4}\Rightarrow\mathbf{z}'\in \underset{i=2}{\overset{4}{\bigcup}}S_{i}$}
	
	We have $\mathbf{z}\in S_{4}\Rightarrow\left\Vert \mathbf{z}\right\Vert ^{2}=\left\Vert \mathbf{w}\right\Vert ^{2}+\zeta^{2}=(1+\varepsilon)\left\Vert \mathbf{x}\right\Vert ^{2}$
	for some $\varepsilon\leq c$ . 
	
	1) $\frac{\left\Vert \mathbf{x}\right\Vert ^{2}}{2}<\left\Vert \mathbf{z}'\right\Vert ^{2}$
	is equivalent to 
	
	\[
	A\equiv-(\varepsilon+\frac{1}{2})\left\Vert \mathbf{x}\right\Vert ^{2}
	\]
	\[
	\leq\eta\left\Vert \mathbf{x}\right\Vert ^{2}\left[\begin{array}{c}
	\left(-4(1+2\varepsilon)+\eta\left\Vert \mathbf{x}\right\Vert ^{2}(1+2\varepsilon)^{2}\right)\left\Vert \mathbf{w}\right\Vert ^{2}\\
	+4\left(-\varepsilon+\eta\left\Vert \mathbf{x}\right\Vert ^{2}\varepsilon^{2}\right)\zeta^{2}
	\end{array}\right]\equiv B
	\]
	
	. We have 
	\[
	B\geq-4\eta\left\Vert \mathbf{x}\right\Vert ^{2}\left((1+2\varepsilon)\left\Vert \mathbf{w}\right\Vert ^{2}+\varepsilon\zeta^{2}\right)\geq-\frac{15}{8b}\left\Vert \mathbf{x}\right\Vert ^{2}
	\]
	where the last inequality used $\left\Vert \mathbf{w}\right\Vert ^{2}\leq\frac{\left\Vert \mathbf{x}\right\Vert ^{2}}{2}$
	and $\left\Vert \mathbf{z}\right\Vert ^{2}\leq\left\Vert \mathbf{x}\right\Vert ^{2}(1+c)\Rightarrow\zeta^{2}\leq\left\Vert \mathbf{x}\right\Vert ^{2}(\frac{1}{2}+c)$.
	The choice $b=4$ gaurantees $A\leq B$ which ensures the desired
	result. 
	
	2) This is trivial since $\left\Vert \mathbf{z}\right\Vert ^{2}\leq(1+c)\left\Vert \mathbf{x}\right\Vert ^{2}$
	and in $S_{4}$ both $\zeta$ and $\left\Vert \mathbf{w}\right\Vert $decay
	at every iteration (ref eq). 
	
	\subsection{$\mathbf{z}\in S_{2}\Rightarrow\mathbf{z}'\in \underset{i=2}{\overset{4}{\bigcup}}S_{i}$}
	
	1) We use $\mathbf{z}\in S_{2}\Rightarrow\left\Vert \mathbf{z}\right\Vert ^{2}=\left\Vert \mathbf{w}\right\Vert ^{2}+\zeta^{2}=(\frac{1}{2}+\varepsilon)\left\Vert \mathbf{x}\right\Vert ^{2}$
	for some $\varepsilon\leq\frac{1}{2}-c$ . Using a similar argument
	as in the previous section, we are required to show 
	
	\[
	-\varepsilon\left\Vert \mathbf{x}\right\Vert ^{2}<\eta\left\Vert \mathbf{x}\right\Vert ^{2}\left[\begin{array}{c}
	4\left(-\varepsilon+\varepsilon^{2}\eta\left\Vert \mathbf{x}\right\Vert ^{2}\right)\left\Vert \mathbf{w}\right\Vert ^{2}\\
	+\left(2(1-2\varepsilon)+(1-2\varepsilon)^{2}\eta\left\Vert \mathbf{x}\right\Vert ^{2}\right)\zeta^{2}
	\end{array}\right]
	\]
	\[
	\equiv B
	\]
	
	where $B\geq-\varepsilon\frac{\left\Vert \mathbf{x}\right\Vert ^{2}}{b}$
	implies that $b=4$ gives the desired result.
	
	2) The condition is equivalent to 
	
	\[
	A\equiv\eta\left\Vert \mathbf{x}\right\Vert ^{2}\left[\begin{array}{c}
	4\left(-\varepsilon+\varepsilon^{2}\eta\left\Vert \mathbf{x}\right\Vert ^{2}\right)\left\Vert \mathbf{w}\right\Vert ^{2}\\
	+\left(2(1-2\varepsilon)+(1-2\varepsilon)^{2}\eta\left\Vert \mathbf{x}\right\Vert ^{2}\right)\zeta^{2}
	\end{array}\right]+\varepsilon\left\Vert \mathbf{x}\right\Vert ^{2}
	\]
	\[
	\leq(\frac{1}{2}+c)\left\Vert \mathbf{x}\right\Vert ^{2}\equiv B
	\]
	
	One can show by looking for critical points of $A(\varepsilon)$ in
	the range $0\leq\varepsilon\leq\frac{1}{2}$ that $A$ is maximized
	at $\varepsilon=0$, since there is only one critical point at $\varepsilon^{\ast}=\frac{4-\frac{b}{\sqrt{c}}+2\frac{\sqrt{c}}{b}}{8\frac{\sqrt{c}}{b}}$
	and $A(\varepsilon^{\ast})<0$, while 
	
	\[
	A(\frac{1}{2})\leq\left[\left(-2\frac{\sqrt{c}}{b}+\frac{c}{b^{2}}\right)\left\Vert \mathbf{w}\right\Vert ^{2}\right]+\frac{1}{2}\left\Vert \mathbf{x}\right\Vert ^{2}
	\]
	
	\[
	A(0)\leq\frac{1}{2b}\left(2+\frac{1}{2b}\right)\frac{\left\Vert \mathbf{x}\right\Vert ^{2}}{2}
	\]
	
	and in both cases $b=4$ ensures $A\leq B$. 

	\subsection{$\mathbf{z}\in S_{1}\Rightarrow\mathbf{z}'\in S_{1}\cup S_{2}$}
	
	We must show $\left\Vert \mathbf{z}'\right\Vert \leq(1-c)\left\Vert \mathbf{x}\right\Vert ^{2}$
	using $\left\Vert \mathbf{z}\right\Vert ^{2}=(1-\varepsilon)\frac{\left\Vert \mathbf{x}\right\Vert ^{2}}{2}$
	for $0\leq\varepsilon\leq1$. 
	
	\[
	\left\Vert \mathbf{z}'\right\Vert ^{2}=\left(1+\varepsilon\eta\left\Vert \mathbf{x}\right\Vert ^{2}\right)^{2}\left\Vert \mathbf{w}\right\Vert ^{2}+\left(1+2(\varepsilon+1)\eta\frac{\left\Vert \mathbf{x}\right\Vert ^{2}}{2}\right)^{2}\zeta^{2}
	\]
	
	\[
	A\equiv\eta\left\Vert \mathbf{x}\right\Vert ^{2}\left[\begin{array}{c}
	\left(2\varepsilon+\varepsilon^{2}\eta\left\Vert \mathbf{x}\right\Vert ^{2}\right)\left\Vert \mathbf{w}\right\Vert ^{2}\\
	+\left(2(\varepsilon+1)+(\varepsilon+1)^{2}\eta\frac{\left\Vert \mathbf{x}\right\Vert ^{2}}{4}\right)\zeta^{2}
	\end{array}\right]-\varepsilon\left\Vert \mathbf{x}\right\Vert ^{2}
	\]
	\[
	\leq(\frac{1}{2}-c)\left\Vert \mathbf{x}\right\Vert ^{2}\equiv B
	\]
	
	and since $A\leq\frac{1}{2b}\left[2+\frac{1}{b}\right]\frac{\left\Vert \mathbf{x}\right\Vert ^{2}}{2}$
	and $B\geq\frac{\left\Vert \mathbf{x}\right\Vert ^{2}}{4}$ once again
	$b=4$ suffices to obtain the desired result. 
\end{proof}
%(
%\[
%A=\eta\left\Vert \mathbf{x}\right\Vert ^{2}\left[\left(2\varepsilon+\varepsilon^{2}\eta\left\Vert \mathbf{x}\right\Vert ^{2}\right)\left\Vert \mathbf{w}\right\Vert ^{2}+\left(2(\varepsilon+1)+(\varepsilon+1)^{2}\eta\frac{\left\Vert \mathbf{x}\right\Vert ^{2}}{4}\right)\zeta^{2}\right]-\varepsilon\left\Vert \mathbf{x}\right\Vert ^{2}\leq\frac{1}{2b}\left[\left(2+\frac{1}{2b}\right)+\left(4+\frac{1}{2b}\right)\right]\frac{\left\Vert \mathbf{x}\right\Vert ^{2}}{2}-\frac{4b\varepsilon}{2}\left\Vert \mathbf{x}\right\Vert ^{2}
%\]
%
%\[
%\leq\frac{1}{2b}\left[\varepsilon^{2}\frac{1}{2b}+4\varepsilon(1-b)+2+(\varepsilon+1)^{2}\frac{1}{8b}\right]\frac{\left\Vert \mathbf{x}\right\Vert ^{2}}{2}\leq\frac{1}{2b}\left[2+\frac{1}{b}\right]\frac{\left\Vert \mathbf{x}\right\Vert ^{2}}{2}
%\]
%)

\begin{lemma} \label{gprconvlem}
	For $\mb z$ parametrized as in \ref{zdefeq}, 
	\[
	\left\Vert \mathbf{w}\right\Vert ^{2}<c\left\Vert \mathbf{x}\right\Vert ^{2}\vee\zeta^{2}>(1-c)\left\Vert \mathbf{x}\right\Vert ^{2}
	\]
	\[
	\Rightarrow\mathrm{dist}(\mathbf{z},\breve{A})<\sqrt{5c}\left\Vert \mathbf{x}\right\Vert 
	\]
\end{lemma}
\begin{proof}[\bf{Proof of Lemma \ref{gprconvlem}}] \label{gprconvlemp}
	Once $\left\Vert \mathbf{w}\right\Vert ^{2}<c\left\Vert \mathbf{x}\right\Vert ^{2}$
	for some $\mathbf{z}\in S_{3}\cup S_{4}$ we have 
	
	\[
	\left\Vert \mathbf{z}\right\Vert ^{2}=\zeta^{2}+\left\Vert \mathbf{w}\right\Vert ^{2}\geq(1-c)\left\Vert \mathbf{x}\right\Vert ^{2}
	\]
	
	\begin{equation} \lbl{zsqeq}
	\zeta^{2}\geq(1-c)\left\Vert \mathbf{x}\right\Vert ^{2}-\left\Vert \mathbf{w}\right\Vert ^{2}>(1-2c)\left\Vert \mathbf{x}\right\Vert ^{2}
	\end{equation}
	
	For some $\mathbf{z}=\mathbf{w}+\zeta e^{i\phi}\frac{\mathbf{x}}{\left\Vert \mathbf{x}\right\Vert }$
	we have 
	
	\[
	\begin{array}{c}
	\mathrm{dist}^{2}(\mathbf{z},\breve{A})=\underset{\theta}{\min}\left\Vert e^{i\theta}\mathbf{x}-\mathbf{w}-\zeta e^{i\phi}\frac{\mathbf{x}}{\left\Vert \mathbf{x}\right\Vert }\right\Vert ^{2}\\
	=\left\Vert \mathbf{w}\right\Vert ^{2}+\underset{\theta}{\min}\left\Vert e^{i\theta}\mathbf{x}-\zeta e^{i\phi}\frac{\mathbf{x}}{\left\Vert \mathbf{x}\right\Vert }\right\Vert ^{2}
	\end{array}
	\]
	
	\[
	=\left\Vert \mathbf{w}\right\Vert ^{2}+(1-\frac{\zeta}{\left\Vert \mathbf{x}\right\Vert })^{2}\left\Vert \mathbf{x}\right\Vert ^{2}=\left\Vert \mathbf{z}\right\Vert ^{2}+\left\Vert \mathbf{x}\right\Vert ^{2}-2\zeta\left\Vert \mathbf{x}\right\Vert 
	\]
	
	if we assume $\left\Vert \mathbf{z}\right\Vert ^{2}\leq(1+c)\left\Vert \mathbf{x}\right\Vert ^{2}$ 
	
	\begin{equation} \lbl{disteq}
	\mathrm{dist}^{2}(\mathbf{z},\breve{A})\leq(c+2)\left\Vert \mathbf{x}\right\Vert ^{2}-2\zeta\left\Vert \mathbf{x}\right\Vert 
	\end{equation}
	
	plugging in the value of $\zeta$ from \ref{zsqeq} and using fact that $-\sqrt{1-x}\leq-1+x$
	for $x<1$ we have 
	
	\[
	\mathrm{dist}^{2}(\mathbf{z},\breve{A})<(c+2)\left\Vert \mathbf{x}\right\Vert ^{2}-2\sqrt{1-2c}\left\Vert \mathbf{x}\right\Vert ^{2}\leq5c\left\Vert \mathbf{x}\right\Vert^2
	\]
	
	Alternatively, if $\zeta^{2}>(1-c)\left\Vert \mathbf{x}\right\Vert ^{2}$
	we have from \ref{disteq}
	
	\[
	\mathrm{dist}^{2}(\mathbf{z},\breve{A})\leq(c+2)\left\Vert \mathbf{x}\right\Vert ^{2}-2\zeta\left\Vert \mathbf{x}\right\Vert 
	\]
	\[
	<(c+2)\left\Vert \mathbf{x}\right\Vert ^{2}-2\sqrt{1-c}\left\Vert \mathbf{x}\right\Vert ^{2}\leq3c\left\Vert \mathbf{x}\right\Vert^2
	\]
	
	which gives the desired result. In particular, if we choose $c=\frac{1}{35}$ we converge to $\mathrm{dist}^{2}(\mathbf{z},\breve{A})<\frac{\left\Vert \mathbf{x}\right\Vert ^{2}}{7}$,
	a region which is strongly convex according to \cite{sun2017complete}.
\end{proof}
\begin{proof}[\begin{minipage}{3.25in}\bf{Proof of Theorem \ref{gprthm}: (Gradient descent convergence rate for generalized phase retrieval)}\end{minipage}] \label{gprthmp}
	
	We now bound the number of iterations that gradient descent, after random initialization in $S_1$, requires to reach a point where one of the convergence criteria detailed in Lemma \ref{gprconvlem} is fulfilled. From Lemma \ref{ziterlem2}, we know that after initialization in $S_1$ we need to consider only the set $\underset{i=1}{\overset{4}{\bigcup}}S_{i}$. The number of iterations in each set will be determined by the bounds on the change in $\zeta,|| \mb w ||$ detailed in \ref{sineqs}.
	\subsubsection{Iterations in $S_{1}$}
	
	Assuming we initialize with some $\zeta = \zeta_0$. Then the maximal number of iterations in
	this region is 
	
	\[
	\zeta_0(1+\eta\left\Vert \mathbf{x}\right\Vert ^{2})^{t_{1}}=\frac{\left\Vert \mathbf{x}\right\Vert }{\sqrt{2}}
	\]
	
	\[
	t_{1}=\frac{\log\left(\frac{\left\Vert \mathbf{x}\right\Vert }{\zeta_0\sqrt{2}}\right)}{\log(1+\eta\left\Vert \mathbf{x}\right\Vert ^{2})}
	\]
	
	since after this many iterations $\left\Vert \mathbf{z}\right\Vert ^{2}\geq\zeta^{2}\geq\frac{\left\Vert \mathbf{x}\right\Vert ^{2}}{2}$.
	
	\subsubsection{Iterations in $\underset{i=2}{\overset{4}{\bigcup}}S_{i}$}
	
	The convergence criteria are $\left\Vert \mathbf{w}\right\Vert ^{2}<c\left\Vert \mathbf{x}\right\Vert ^{2}$
	or $\zeta^{2}>(1-c)\left\Vert \mathbf{x}\right\Vert ^{2}$. 
	
	After exiting $S_{1}$ and assuming the next iteration is in $S_{2}$,
	the maximal number of iterations required to reach $S_{3}\cup S_{4}$
	is obtained using
	
	\[
	\zeta'\geq(1+2\eta\left\Vert \mathbf{x}\right\Vert ^{2}c)\zeta
	\]
	
	and is given by 
	
	\[
	\frac{\left\Vert \mathbf{x}\right\Vert }{\sqrt{2}}(1+2\eta\left\Vert \mathbf{x}\right\Vert ^{2}c)^{t_{2}}=(1-c)\left\Vert \mathbf{x}\right\Vert ^{2}
	\]
	\[
	t_{2}=\frac{\log\left(\sqrt{2(1-c)}\right)}{\log(1+2\eta\left\Vert \mathbf{x}\right\Vert ^{2}c)}\leq\frac{\log(2)}{2\log(1+2\eta\left\Vert \mathbf{x}\right\Vert ^{2}c)}
	\]
	
	since after this many iterations $\left\Vert \mathbf{z}\right\Vert ^{2}\geq\zeta^{2}\geq(1-c)\left\Vert \mathbf{x}\right\Vert ^{2}$.
	
	For every iteration in $S_{3}\cup S_{4}$ we are guaranteed 
	\[
	\left\Vert \mathbf{w}'\right\Vert \leq\left(1-(1-2c)\eta\left\Vert \mathbf{x}\right\Vert ^{2}\right)\left\Vert \mathbf{w}\right\Vert 
	\]
	thus using Lemmas \ref{wziterlem}.i and \ref{gprconvlem}
	the number of iterations in $S_{3}\cup S_{4}$ required for convergence
	is given by
	\[
	\frac{\left\Vert \mathbf{x}\right\Vert ^{2}}{2}\left(1-(1-2c)\eta\left\Vert \mathbf{x}\right\Vert ^{2}\right)^{t_{3+4}}=c\left\Vert \mathbf{x}\right\Vert ^{2}
	\]
	\[
	t_{3+4}=\frac{\log(2c)}{\log\left(1-(1-2c)\eta\left\Vert \mathbf{x}\right\Vert ^{2}\right)}
	\]
	The only concern is that after an iteration in $S_{3}\cup S_{4}$
	the next iteration might be in $S_{2}$. To account for this situation,
	we find the maximal number of iterations required to reach $S_{3}\cup S_{4}$
	again. This is obtained from the bound on $\zeta$ in Lemma \ref{wziterlem}. 
	
	Using this result, and the fact that for every iteration in $S_{2}$
	we are guaranteed $\zeta'\geq(1+2\eta\left\Vert \mathbf{x}\right\Vert ^{2}c)\zeta$
	the number of iterations required to reach $S_{3}\cup S_{4}$ again
	is given by
	\[
	\frac{\sqrt{7}}{4}\left\Vert \mathbf{x}\right\Vert (1+2\eta\left\Vert \mathbf{x}\right\Vert ^{2}c)^{t_{r}}=\sqrt{1-c}\left\Vert \mathbf{x}\right\Vert 
	\]
	\[
	t_{r}=\frac{\log\left(\frac{4\sqrt{1-c}}{\sqrt{7}}\right)}{\log(1+2\eta\left\Vert \mathbf{x}\right\Vert ^{2}c)}\leq\frac{\log(\frac{4}{\sqrt{7}})}{\log(1+2\eta\left\Vert \mathbf{x}\right\Vert ^{2}c)}
	\]
	\subsection{Final rate}
	
	The final rate to convergence is 
	\[
	T<t_{1}+t_{2}+t_{3+4}t_{r}
	\]
	\[
	\begin{array}{c}
	=\frac{\log\left(\frac{\left\Vert \mathbf{x}\right\Vert }{\zeta\sqrt{2}}\right)}{\log(1+\eta\left\Vert \mathbf{x}\right\Vert ^{2})}+\frac{\log(2)}{2\log(1+2c\eta\left\Vert \mathbf{x}\right\Vert ^{2})}\\
	+\frac{\log(2c)\log(\frac{4}{\sqrt{7}})}{\log\left(1-\left(1-2c\right)\eta\left\Vert \mathbf{x}\right\Vert ^{2}\right)\log(1+2c\eta\left\Vert \mathbf{x}\right\Vert ^{2})}
	\end{array}
	\]
	
	\subsection{Probability of the bound holding}
	
	The bound applies to an initialization with $\zeta \geq \zeta_0$, hence in $S_{1}\backslash\overline{Q}_{\zeta_{0}}$. Assuming uniform initialization in $S_{1}$, the set $\overline{Q}_{\zeta_{0}}$ is simply
	a band of width $2\zeta_0$ around the equator of the ball $B_{\left\Vert \mathbf{x}\right\Vert /\sqrt{2}}$
	(in $\mathbb{R}^{2n}$, using the natural identification of $\mathbb{C}^n$ with $\mathbb{R}^{2n}$). This volume can be calculated by integrating
	over $2n-1$ dimensional balls of varying radius.
	
	Denoting $r=\frac{\zeta_0\sqrt{2}}{\left\Vert \mathbf{x}\right\Vert }$
	and by $V(n)=\frac{\pi^{n/2}}{\frac{n}{2}\Gamma(\frac{n}{2})}$ the
	hypersphere volume, the probability of initializing in $S_{1}\cap\overline{Q}_{\zeta_{0}}$ (and thus in a region that
	feeds into small gradient regions around saddle points) is 
	
	\[
	\mathbb{P}(\mathrm{fail})=\frac{\mathrm{Vol}(\overline{Q}_{\zeta_{0}})}{\mathrm{Vol}(B_{\left\Vert \mathbf{x}\right\Vert /\sqrt{2}})}
	\]
	\[
	=\frac{V(2n-1)\underset{-r}{\overset{r}{\int}}(1-x^{2})^{\frac{2n-1}{2}}dx}{V(2n)}
	\]
	\[
	\leq\frac{V(2n-1)\underset{-r}{\overset{r}{\int}}e^{-\frac{2n-1}{2}x^{2}}dx}{V(2n)}
	\]
	\[
	=\frac{1}{\sqrt{n-\frac{1}{2}}}\frac{n}{n-\frac{1}{2}}\frac{\Gamma(n)}{\Gamma(\frac{2n-1}{2})}\mathrm{erf}(\sqrt{\frac{2n-1}{2}}r)
	\]
	\[
	\leq\sqrt{\frac{8}{\pi}}\mathrm{erf}(\sqrt{n}r)
	\]
	. For small $\zeta$ we again find that $\mathbb{P}(\mathrm{fail})$
	scales linearly with $\zeta$, as was the case for the previous problems considered. 
\end{proof}

\section{Auxiliary Lemmas} %%%
\subsection{Separable objective}\lbl{derivs}
\[
\frac{\partial g_{s}(\mb w)}{\partial w_{i}}=\tanh\left(\frac{w_{i}}{\mu}\right)-\tanh\left(\frac{q_{n}}{\mu}\right)\frac{w_{i}}{q_{n}}
\]
\[
\frac{\partial^{2}g_{s}(\mb w)}{\partial w_{i}\partial w_{j}}=\left[\frac{1}{\mu}\mathrm{sech}^{2}\left(\frac{w_{i}}{\mu}\right)-\tanh\left(\frac{q_{n}}{\mu}\right)\frac{1}{q_{n}}\right]\delta_{ij}
\]
\[
+\left[\frac{1}{\mu}\mathrm{sech}^{2}\left(\frac{q_{n}}{\mu}\right)\frac{1}{q_{n}^{2}}-\tanh\left(\frac{q_{n}}{\mu}\right)\frac{1}{q_{n}^{3}}\right]w_{i}w_{j}
\]
\subsection{Dictionary Learning}
\[
\nabla_{\mb w}g_{DL}^{pop}(\mb w)=\mathbb{E} \left[ \tanh\left(\frac{\mb q^{\ast}(\mb w) \mb x}{\mu}\right)\left(\overline{x}-\frac{x_{n}}{q_{n}(\mb w)}w\right) \right]
\]
\subsection{Properties of $\mathcal{C}_{\zeta}$}

\begin{proof}[\bf{Proof of Lemma \ref{vollem}:} (Volume of $\mathcal{C}_{\zeta}$)] \label{vollemp}
	
	We are interested in the relative volume $\frac{\mathrm{Vol}(\mathcal{C}_{\zeta})}{\mathrm{Vol}(\mathbb{S}^{n-1})} \equiv V_\zeta$.
	Using the standard solid angle formula, it is given by 
	
	%\[
	%\underset{\varepsilon\rightarrow0}{\lim}\frac{1}{\varepsilon^{n/2}}\underset{L_{\frac{\pi}{2}-\zeta}^{\infty}(\alpha)}{\int}e^{-\frac{\pi}{\varepsilon}x^{2}}dx
	%\]
	
	\[
	V_\zeta=\underset{\varepsilon\rightarrow0}{\lim}\frac{1}{\varepsilon^{n/2}}\underset{0}{\overset{\infty}{\int}}e^{-\frac{\pi}{\varepsilon}x_{1}^{2}}\underset{i=2}{\overset{n}{\Pi}}\underset{-x_{1}/(1+\zeta)}{\overset{x_{1}/(1+\zeta)}{\int}}e^{-\frac{\pi}{\varepsilon}x_{i}^{2}}dx_{i}dx_1
	\]
	
	\[
	=\underset{\varepsilon\rightarrow0}{\lim}\frac{1}{\sqrt{\varepsilon}}\underset{0}{\overset{\infty}{\int}}e^{-\frac{\pi}{\varepsilon}x^{2}}\left[\text{erf}(\frac{x}{\left(1+\zeta\right)}\sqrt{\frac{\pi}{\varepsilon}})\right]^{n-1}dx
	\]
	
	changing variables to $\tilde{x}=\sqrt{\frac{\pi}{\varepsilon}}\frac{x}{\left(1+\zeta\right)}$
	
	\[
	V_\zeta=\frac{\left(1+\zeta\right)}{\sqrt{\pi}}\underset{0}{\overset{\infty}{\int}}e^{-\left(1+\zeta\right)^2x^{2}}\text{erf}^{n-1}(x)dx
	\]
	
	This integral admits no closed form solution but one can construct
	a linear approximation around small $\zeta$ and show that it is convex.
	Thus the approximation provides a lower bound for $V_\zeta$ and an upper bound on the failure probability. 
	
	From symmetry considerations the zero-order term is $V_0=\frac{1}{2n}$. The first-order term is given by 
	
	\[
	\frac{\partial V_\zeta}{\partial\zeta}_{\zeta=0}=\frac{1}{n}-\frac{2}{\sqrt{\pi}}\underset{0}{\overset{\infty}{\int}}x^{2}e^{-x^{2}}\text{erf}^{n-1}(x)dx
	\]
	
	We now require an upper bound for the second integral since we are interested in a lower bound for $V_\zeta$. We can express it
	in terms of the second moment of the $L^{\infty}$ norm of a Gaussian
	vector as follows:
	
	\[
	\frac{1}{\sqrt{\pi}}\underset{0}{\overset{\infty}{\int}}x^{2}e^{-x^{2}}\mathrm{erf}^{n-1}(x)=\frac{1}{\sqrt{\pi}}\underset{0}{\overset{\infty}{\int}}x^{2}e^{-x^{2}}\underset{i}{\Pi}\frac{1}{\sqrt{\pi}}\underset{-x}{\overset{x}{\int}}e^{-t_{i}^{2}}dt_{i}dx
	\]
	\[
	=\frac{1}{\sqrt{2\pi}}\underset{0}{\overset{\infty}{\int}}\frac{x^{2}}{2}e^{-x^{2}/2}\underset{i}{\Pi}\frac{1}{\sqrt{2\pi}}\underset{-x}{\overset{x}{\int}}e^{-t_{i}^{2}/2}dt_{i}dx
	\]
	\[
	=\frac{1}{4n}\int\left\Vert \mb X\right\Vert _{\infty}^{2}d\mu(\mb X)
	\]
	
	\[
	=\frac{1}{4n}\left(\mathrm{Var}\left[\left\Vert \mb X\right\Vert _{\infty}\right]+\left(\mathbb{E}\left[\left\Vert \mb X\right\Vert _{\infty}\right]\right)^{2}\right)
	\]
	
	where $\mu( \mb X)$ is the Gaussian measure on the vector $\mb X\in\mathbb{R}^{n}$. We can bound the first term
	using 
	
	\[
	\mathrm{Var}\left[\left\Vert \mb X\right\Vert _{\infty}\right]\leq\underset{i}{\max}\mathrm{Var}\left[\left\vert X_{i}\right\vert \right]=\mathrm{Var}\left[\left\vert X_{i}\right\vert \right]<\mathrm{Var}\left[X_{i}\right]=1
	\]
	
	To bound the second term, we use the fact that for a standard Gaussian
	vector $\mb X$ ($X_{i}\sim\mathcal{N}(0,1)$) and any $\lambda>0$ we
	have 
	
	\[
	\exp\left(\lambda\mathbb{E}\left[\left\Vert \mb X\right\Vert _{\infty}\right]\right)\leq\mathbb{E}\left[\exp\left(\lambda\underset{i}{\max}\left\vert X_{i}\right\vert \right)\right]
	\]
	\[
	\leq\mathbb{E}\left[\underset{i}{\sum}\exp\left(\lambda\left\vert X_{i}\right\vert \right)\right]=n\mathbb{E}\left[\exp\left(\lambda\left\vert X_{i}\right\vert \right)\right]
	\]
	
	(using convexity and non-negativity of the exponent respectively)
	
	\[
	n\mathbb{E}\left[\exp\left(\lambda\left\vert X_{i}\right\vert \right)\right]=2n\underset{0}{\overset{\infty}{\int}}\exp\left(\lambda X_{i}\right)d\mu(X_{i})
	\]
	\[
	\leq2n\mathbb{E}\left[\exp\left(\lambda X_{i}\right)\right]=2n\exp\left(\frac{\lambda^{2}}{2}\right)
	\]
	
	taking the log of both sides gives 
	
	\[
	\mathbb{E}\left[\underset{i}{\max}\left\vert X_{i}\right\vert \right]\leq\frac{\log(2n)}{\lambda}+\frac{\lambda}{2}
	\]
	
	and the bound is minimized for $\lambda=\sqrt{2\log(2n)}$ giving 
	
	\[
	\mathbb{E}\left[\underset{i}{\max}\left\vert X_{i}\right\vert \right]\leq\sqrt{2\log(2n)}\sim\sqrt{2\log(n)}
	\]
	
	Combining these bounds, the leading order behavior of the gradient
	is 
	
	\[
	\frac{\partial V_\zeta}{\partial\zeta}_{\zeta=0}\geq \frac{3-4\log(2n)}{4n} \geq -\frac{\log(n)}{n}.
	\]
	
	This linear approximation is indeed a lower bound, since using integration by parts twice we have
	
	\[
	\frac{\partial^{2}V_\zeta}{\partial\zeta^{2}}=\frac{1}{\sqrt{\pi}}\underset{0}{\overset{\infty}{\int}}e^{-(1+\zeta)^{2}x^{2}}\left(\begin{array}{c}
	-6(1+\zeta)x^{2}\\
	+4(1+\zeta)^{3}x^{4}
	\end{array}\right)\text{erf}^{n-1}(x)dx
	\]
	
	\[
	=-\frac{2(n-1)}{\pi}\underset{0}{\overset{\infty}{\int}}e^{-(1+\zeta)^{2}x^{2}}\left(1-2(1+\zeta)^{2}x^{2}\right)e^{-x^{2}}\text{erf}^{n-2}(x)dx
	\]
	
	\[
	=\frac{4(n-1)(n-2)(1+\zeta)}{\pi^{3/2}}\underset{0}{\overset{\infty}{\int}}e^{-((1+\zeta)^{2}+2)x^{2}}\text{erf}^{n-3}(x)dx>0
	\]
	
	where the last inequality holds for any $n>2$ since the integrand
	is non-negative everywhere. This gives 
	
	\[
	V_\zeta \geq \frac{1}{2n} -\frac{\log(n)}{n} \zeta
	\]
	
\end{proof}

\begin{lemma} \label{czlemma}
	$B_{s(\zeta)}^{\infty}(0)\subseteq \mathcal{C}_{\zeta}\subseteq B_{\sqrt{n-1}s(\zeta)}^{2}(0)$
	where $s(\zeta)=\frac{1}{\sqrt{(2+\zeta)\zeta+n}}$. $B_{s(\zeta)}^{\infty}(0)$
	is the largest $L^{\infty}$ ball contained in $\mathcal{C}_{\zeta}$, and $B_{\sqrt{n-1}s(\zeta)}^{2}(0)$
	is the smallest $L^{2}$ ball containing $\mathcal{C}_{\zeta}$ (where these balls are defined in terms of the $\mb w$ vector). All three intersect
	only at the points where all the coordinates of $\mb w$ have equal magnitude. Additionally, $\mathcal{C}_{\zeta}\subseteq B_{1/\sqrt{2+\zeta}}^{\infty}(0)$ and this is the smallest $L^{\infty}$ ball containing $\mathcal{C}_{\zeta}$. 
\end{lemma}

\begin{proof}
	Given the surface of some $L^{\infty}$ ball for $\mb w$ , we can
	ask what is the minimal $\zeta$ such that $\partial C_{\zeta_{m}}$
	intersects this surface. This amounts to finding the minimal $q_{n}$
	given some $\left\Vert \mb w\right\Vert _{\infty}$. Yet this is clearly
	obtained by setting all the coordinates of $w$ to be equal to $\left\Vert \mb w\right\Vert _{\infty}$
	(this is possible since we are guaranteed $q_{n}\geq\left\Vert \mb w\right\Vert _{\infty}\Rightarrow\left\Vert \mb w\right\Vert _{\infty}\leq\frac{1}{\sqrt{n}}$),
	giving 
	
	\[
	\frac{\sqrt{1-(n-1)\left\Vert \mb w\right\Vert _{\infty}^{2}}}{\left\Vert \mb w\right\Vert _{\infty}}=1+\zeta_{m}
	\]
	
	\[
	\left\Vert \mb w\right\Vert _{\infty}=\frac{1}{\sqrt{(1+\zeta_{m})^{2}+n-1}}
	\]
	
	thus, given some $\zeta$, the maximal $L^{\infty}$ ball that is
	contained in $\mathcal{C}_{\zeta}$ has radius $\frac{1}{\sqrt{(2+\zeta)\zeta+n}}$. The minimal $L^{\infty}$ norm containing $\mathcal{C}_{\zeta}$ can be shown by a similar argument to be $B_{1/\sqrt{1+(1+\zeta)^2}}^{\infty}(0)$, where one instead maximizes $q_{n}$ with some fixed $\left\Vert \mb w\right\Vert _{\infty}$.
	
	Given some surface of an $L^{2}$ ball, we can ask what is the minimal
	$\mathcal{C}_{\zeta}$ such that $\mathcal{C}_{\zeta}\subseteq B_{r}^{2}(0)$. This is
	equivalent to finding the maximal $\zeta_{M}$ such that $\partial C_{\zeta_{M}}$
	intersects the surface of the $L^{2}$ ball. Since $q_{n}$ is fixed,
	maximizing $\zeta$ is equivalent to minimizing $\left\Vert \mb w\right\Vert _{\infty}$.
	This is done by setting $\left\Vert \mb w\right\Vert _{\infty}=\frac{\left\Vert \mb w\right\Vert }{\sqrt{n-1}}$,
	which gives 
	
	\[
	\frac{\sqrt{1-\left\Vert \mb w\right\Vert ^{2}}}{\left\Vert \mb w\right\Vert }\sqrt{n-1}=1+\zeta_{M}
	\]
	
	\[
	\sqrt{\frac{n-1}{(2+\zeta_M)\zeta_M+n}}=\left\Vert \mb w\right\Vert 
	\]
	The statement in the lemma follows from combining these results.
\end{proof}

\begin{lemma}[Geometric Increase in $\zeta$] \label{zetalem}
	For $\mathbf{w}\in C_{\zeta_{0}}\backslash B_{b}^{\infty}$ (where
	$\zeta\equiv\frac{q_{n}}{\left\Vert \mathbf{w}\right\Vert _{\infty}}-1$),
	assume\textbf{ $\left\vert w_{i}\right\vert >r\Rightarrow\mathbf{u}^{(i)\ast}\mathrm{grad}[f](\mathbf{q}(\mathbf{w}))\geq c(\mb w)\zeta$
	}where $\mathbf{u}^{(i)}$ is defined in \ref{ueq} and $1>b>r$. Then if
	$\left\Vert \mathrm{grad}[f](\mathbf{q}(\mathbf{w}))\right\Vert <M$
	and we define
	\[
	\mathbf{q}'\equiv\exp_{\mathbf{q}}(-\eta\mathrm{grad}[f](\mathbf{q}))
	\]
	for $\eta<\frac{b-r}{3M}$, defining $\zeta'$ in an analogous way
	to $\zeta$ we have 
	\[
	\zeta'\geq\zeta\left(1+\frac{\sqrt{n}}{2}\eta c(\mb w)\right)
	\]
\end{lemma}
Proof: \ref{zetalemp}

\begin{proof}[\bf{Proof of Lemma \ref{zetalem}:(Geometric Increase in $\zeta$)}] \label{zetalemp}
	Denoting $g\equiv\left\Vert \mathrm{grad}[f](\mathbf{q})\right\Vert $,
	we have
	
	\[
	\mathbf{q}'=\cos(g\eta)\mathbf{q}-\sin(g\eta)\frac{\mathrm{grad}[f](\mathbf{q})}{g}
	\]
	
	hence, using Lagrange remainder terms,
	
	\[
	\frac{q_{n}'}{w_{i}'}=\frac{\begin{array}{c}
		q_{n}-\eta\mathrm{grad}[f](\mathbf{q})_{n}-\underset{0}{\overset{g\eta}{\int}}\cos(t)(g\eta-t)dtq_{n}\\
		+\underset{0}{\overset{g\eta}{\int}}\sin(t)(g\eta-t)dt\frac{\mathrm{grad}[f](\mathbf{q})_{n}}{g}
		\end{array}}{\begin{array}{c}
		w_{i}-\eta\mathrm{grad}[f](\mathbf{q})_{i}-\underset{0}{\overset{g\eta}{\int}}\cos(t)(g\eta-t)dtw_{i}\\
		+\underset{0}{\overset{g\eta}{\int}}\sin(t)(g\eta-t)dt\frac{\mathrm{grad}[f](\mathbf{q})_{i}}{g}
		\end{array}}
	\]
	
	. We assume $w_i > 0$, and the converse case is analogous. From convexity of $\frac{1}{1+x}$
	
	\[
	\frac{q_{n}'}{w_{i}'}\geq\begin{array}{c}
	\frac{q_{n}}{w_{i}}+\left[\frac{\eta}{w_{i}}-\frac{\underset{0}{\overset{g\eta}{\int}}\sin(t)(g\eta-t)dt}{w_{i}g}\right]\\
	*\left(\mathrm{grad}[f](\mathbf{q})_{i}-\frac{w_{i}}{q_{n}}\mathrm{grad}[f](\mathbf{q})_{n}\right)
	\end{array}
	\]
	
	\[
	=\frac{q_{n}}{w_{i}}+\frac{\sin(g\eta)}{w_{i}g}\left(\mathrm{grad}[f](\mathbf{q})_{i}-\frac{w_{i}}{q_{n}}\mathrm{grad}[f](\mathbf{q})_{n}\right)
	\]
	
	\[
	=\frac{q_{n}}{w_{i}}+\frac{\sin(g\eta)}{w_{i}g} \mathbf{u}^{(i)\ast}\mathrm{grad}[f](\mathbf{q}(\mathbf{w}))
	\]
	
	We now use $\eta<\frac{b-r}{3M}<\frac{\pi}{2M}\Rightarrow g\eta<\frac{\pi}{2}\Rightarrow\sin(g\eta)\geq\frac{g\eta}{2}$
	and consider two cases. If $\left\vert w_{i}\right\vert >r$ we use the bound on the gradient projection in the lemma statement to obtain
	
	\[
	\frac{q_{n}'}{w_{i}'}\geq\frac{q_{n}}{w_{i}}+\frac{\eta}{2w_{i}}c(\mb w)\zeta\geq\frac{q_{n}}{w_{i}}+\frac{\sqrt{n}}{2}\eta c(\mb w)\zeta
	\]
	
	hence
	
	\begin{equation} \lbl{qweq}
	\frac{q_{n}'}{w_{i}'}-1\geq\frac{q_{n}}{\left\Vert \mathbf{w}\right\Vert _{\infty}}-1+\frac{\sqrt{n}}{2}\eta c(\mb w)\zeta=\zeta\left(1+\frac{\sqrt{n}}{2}\eta c(\mb w)\right)
	\end{equation}
	
	If $\left\vert w_{i}\right\vert <r$ we rule out the possibility that
	$\left\vert w_{i}'\right\vert =\left\Vert \mathbf{w}'\right\Vert _{\infty}$
	by demanding $\eta<\frac{b-r}{3M}$. Since $b(b-r)<1$ we have $1+\frac{1}{3}b(b-r)<\sqrt{1+b(b-r)}$
	hence the requirement on $\eta$ implies 
	
	\[
	\eta<\frac{\sqrt{1+b(b-r)}-1}{gb}=\frac{-2g+\sqrt{4g^{2}+4g^{2}b(b-r)}}{2g^{2}b}
	\]
	
	. If we now combine this with the fact that after a Riemannian gradient
	step $\cos(g\eta)q_{i}-\sin(g\eta)\leq q_{i}'\leq\cos(g\eta)q_{i}+\sin(g\eta)$,
	the above condition on $\eta$ implies the inequality $(*)$, which
	in turn ensures that $\left\vert w_{i}\right\vert <r\Rightarrow\left\vert w_{i}'\right\vert <\left\Vert \mathbf{w}'\right\Vert _{\infty}$: 
	
	\[
	\left\vert w_{i}'\right\vert <\left\vert w_{i}\right\vert +\sin(g\eta)<r+g\eta\underset{(*)}{<}(1-g^{2}\eta^{2})b-g\eta
	\]
	
	\[
	<\cos(g\eta)\left\Vert \mathbf{w}\right\Vert _{\infty}-\sin(g\eta)\leq\left\Vert \mathbf{w}'\right\Vert _{\infty}
	\]
	
	Due to the above analysis, it is evident that any $w_{i}'$ such that
	$\left\vert w_{i}'\right\vert =\left\Vert \mathbf{w}'\right\Vert _{\infty}$
	obeys $\left\vert w_{i}\right\vert >r$, from which it follows that
	we can use \ref{qweq} to obtain 
	
	\[
	\frac{q_{n}'}{\left\Vert \mathbf{w}'\right\Vert _{\infty}}-1=\zeta'\geq\zeta\left(1+\frac{\sqrt{n}}{2}\eta c(\mb w)\right)
	\]
\end{proof}

\bibliographystyle{plain}

\begin{thebibliography}{}

\end{thebibliography}


\begin{thebibliography}{10}
	
	\bibitem{anandkumar2014guaranteed}
	Animashree Anandkumar, Rong Ge, and Majid Janzamin.
	\newblock Guaranteed non-orthogonal tensor decomposition via alternating
	rank-$1 $ updates.
	\newblock {\em arXiv preprint arXiv:1402.5180}, 2014.
	
	\bibitem{balan2006signal}
	Radu Balan, Pete Casazza, and Dan Edidin.
	\newblock On signal reconstruction without phase.
	\newblock {\em Applied and Computational Harmonic Analysis}, 20(3):345--356,
	2006.
	
	\bibitem{bandeira2016low}
	Afonso~S Bandeira, Nicolas Boumal, and Vladislav Voroninski.
	\newblock On the low-rank approach for semidefinite programs arising in
	synchronization and community detection.
	\newblock In {\em Conference on Learning Theory}, pages 361--382, 2016.
	
	\bibitem{bertsekas1999nonlinear}
	Dimitri~P Bertsekas.
	\newblock {\em Nonlinear programming}.
	\newblock Athena scientific Belmont, 1999.
	
	\bibitem{boumal2016nonconvex}
	Nicolas Boumal.
	\newblock Nonconvex phase synchronization.
	\newblock {\em SIAM Journal on Optimization}, 26(4):2355--2377, 2016.
	
	\bibitem{bronstein2005blind}
	Michael~M Bronstein, Alexander~M Bronstein, Michael Zibulevsky, and Yehoshua~Y
	Zeevi.
	\newblock Blind deconvolution of images using optimal sparse representations.
	\newblock {\em IEEE Transactions on Image Processing}, 14(6):726--736, 2005.
	
	\bibitem{candes2015phase}
	Emmanuel~J Candes, Xiaodong Li, and Mahdi Soltanolkotabi.
	\newblock Phase retrieval via wirtinger flow: Theory and algorithms.
	\newblock {\em IEEE Transactions on Information Theory}, 61(4):1985--2007,
	2015.
	
	\bibitem{chen2018gradient}
	Yuxin Chen, Yuejie Chi, Jianqing Fan, and Cong Ma.
	\newblock Gradient descent with random initialization: Fast global convergence
	for nonconvex phase retrieval.
	\newblock {\em arXiv preprint arXiv:1803.07726}, 2018.
	
	\bibitem{conn2000trust}
	Andrew~R Conn, Nicholas~IM Gould, and Ph~L Toint.
	\newblock {\em Trust region methods}, volume~1.
	\newblock Siam, 2000.
	
	\bibitem{corbett2006pauli}
	John~V Corbett.
	\newblock The pauli problem, state reconstruction and quantum-real numbers.
	\newblock {\em Reports on Mathematical Physics}, 57:53--68, 2006.
	
	\bibitem{dauphin2014identifying}
	Yann~N Dauphin, Razvan Pascanu, Caglar Gulcehre, Kyunghyun Cho, Surya Ganguli,
	and Yoshua Bengio.
	\newblock Identifying and attacking the saddle point problem in
	high-dimensional non-convex optimization.
	\newblock In {\em Advances in neural information processing systems}, pages
	2933--2941, 2014.
	
	\bibitem{du2017gradient}
	Simon~S Du, Chi Jin, Jason~D Lee, Michael~I Jordan, Barnabas Poczos, and Aarti
	Singh.
	\newblock Gradient descent can take exponential time to escape saddle points.
	\newblock {\em arXiv preprint arXiv:1705.10412}, 2017.
	
	\bibitem{elad2006image}
	Michael Elad and Michal Aharon.
	\newblock Image denoising via sparse and redundant representations over learned
	dictionaries.
	\newblock {\em IEEE Transactions on Image processing}, 15(12):3736--3745, 2006.
	
	\bibitem{ge2015escaping}
	Rong Ge, Furong Huang, Chi Jin, and Yang Yuan.
	\newblock Escaping from saddle points?online stochastic gradient for tensor
	decomposition.
	\newblock In {\em Conference on Learning Theory}, pages 797--842, 2015.
	
	\bibitem{goldfarb1980curvilinear}
	Donald Goldfarb.
	\newblock Curvilinear path steplength algorithms for minimization which use
	directions of negative curvature.
	\newblock {\em Mathematical programming}, 18(1):31--40, 1980.
	
	\bibitem{goodfellow2014qualitatively}
	Ian~J Goodfellow, Oriol Vinyals, and Andrew~M Saxe.
	\newblock Qualitatively characterizing neural network optimization problems.
	\newblock {\em arXiv preprint arXiv:1412.6544}, 2014.
	
	\bibitem{hardt2016gradient}
	Moritz Hardt, Tengyu Ma, and Benjamin Recht.
	\newblock Gradient descent learns linear dynamical systems.
	\newblock {\em arXiv preprint arXiv:1609.05191}, 2016.
	
	\bibitem{hardt2015train}
	Moritz Hardt, Benjamin Recht, and Yoram Singer.
	\newblock Train faster, generalize better: Stability of stochastic gradient
	descent.
	\newblock {\em arXiv preprint arXiv:1509.01240}, 2015.
	
	\bibitem{jain2013low}
	Prateek Jain, Praneeth Netrapalli, and Sujay Sanghavi.
	\newblock Low-rank matrix completion using alternating minimization.
	\newblock In {\em Proceedings of the forty-fifth annual ACM symposium on Theory
		of computing}, pages 665--674. ACM, 2013.
	
	\bibitem{jin2017escape}
	Chi Jin, Rong Ge, Praneeth Netrapalli, Sham~M Kakade, and Michael~I Jordan.
	\newblock How to escape saddle points efficiently.
	\newblock {\em arXiv preprint arXiv:1703.00887}, 2017.
	
	\bibitem{keshavan2010matrix}
	Raghunandan~H Keshavan, Andrea Montanari, and Sewoong Oh.
	\newblock Matrix completion from a few entries.
	\newblock {\em IEEE Transactions on Information Theory}, 56(6):2980--2998,
	2010.
	
	\bibitem{kreutz2009complex}
	Ken Kreutz-Delgado.
	\newblock The complex gradient operator and the cr-calculus.
	\newblock {\em arXiv preprint arXiv:0906.4835}, 2009.
	
	\bibitem{lee2017first}
	Jason~D Lee, Ioannis Panageas, Georgios Piliouras, Max Simchowitz, Michael~I
	Jordan, and Benjamin Recht.
	\newblock First-order methods almost always avoid saddle points.
	\newblock {\em arXiv preprint arXiv:1710.07406}, 2017.
	
	\bibitem{lee2016gradient}
	Jason~D Lee, Max Simchowitz, Michael~I Jordan, and Benjamin Recht.
	\newblock Gradient descent only converges to minimizers.
	\newblock In {\em Conference on Learning Theory}, pages 1246--1257, 2016.
	
	\bibitem{lee2013near}
	Kiryung Lee, Yihong Wu, and Yoram Bresler.
	\newblock Near optimal compressed sensing of sparse rank-one matrices via
	sparse power factorization.
	\newblock {\em arXiv preprint}, 2013.
	
	\bibitem{mairal2014sparse}
	Julien Mairal, Francis Bach, Jean Ponce, et~al.
	\newblock Sparse modeling for image and vision processing.
	\newblock {\em Foundations and Trends{\textregistered} in Computer Graphics and
		Vision}, 8(2-3):85--283, 2014.
	
	\bibitem{miao2002high}
	Jianwei Miao, Tetsuya Ishikawa, Bart Johnson, Erik~H Anderson, Barry Lai, and
	Keith~O Hodgson.
	\newblock High resolution 3d x-ray diffraction microscopy.
	\newblock {\em Physical review letters}, 89(8):088303, 2002.
	
	\bibitem{nesterov2013introductory}
	Yurii Nesterov.
	\newblock {\em Introductory lectures on convex optimization: A basic course},
	volume~87.
	\newblock Springer Science \& Business Media, 2013.
	
	\bibitem{nicolaescu2011invitation}
	Liviu Nicolaescu.
	\newblock {\em An invitation to Morse theory}.
	\newblock Springer Science \& Business Media, 2011.
	
	\bibitem{olshausen1996emergence}
	Bruno~A Olshausen and David~J Field.
	\newblock Emergence of simple-cell receptive field properties by learning a
	sparse code for natural images.
	\newblock {\em Nature}, 381(6583):607, 1996.
	
	\bibitem{qu2014finding}
	Qing Qu, Ju~Sun, and John Wright.
	\newblock Finding a sparse vector in a subspace: Linear sparsity using
	alternating directions.
	\newblock In {\em Advances in Neural Information Processing Systems}, pages
	3401--3409, 2014.
	
	\bibitem{ravishankar2011mr}
	Saiprasad Ravishankar and Yoram Bresler.
	\newblock Mr image reconstruction from highly undersampled k-space data by
	dictionary learning.
	\newblock {\em IEEE transactions on medical imaging}, 30(5):1028--1041, 2011.
	
	\bibitem{shechtman2015phase}
	Yoav Shechtman, Yonina~C Eldar, Oren Cohen, Henry~Nicholas Chapman, Jianwei
	Miao, and Mordechai Segev.
	\newblock Phase retrieval with application to optical imaging: a contemporary
	overview.
	\newblock {\em IEEE signal processing magazine}, 32(3):87--109, 2015.
	
	\bibitem{spielman2012exact}
	Daniel~A Spielman, Huan Wang, and John Wright.
	\newblock Exact recovery of sparsely-used dictionaries.
	\newblock In {\em Conference on Learning Theory}, pages 37--1, 2012.
	
	\bibitem{sun2015complete}
	Ju~Sun, Qing Qu, and John Wright.
	\newblock Complete dictionary recovery over the sphere.
	\newblock In {\em Sampling Theory and Applications (SampTA), 2015 International
		Conference on}, pages 407--410. IEEE, 2015.
	
	\bibitem{sun2015nonconvex}
	Ju~Sun, Qing Qu, and John Wright.
	\newblock When are nonconvex problems not scary?
	\newblock {\em arXiv preprint arXiv:1510.06096}, 2015.
	
	\bibitem{sun2016geometric}
	Ju~Sun, Qing Qu, and John Wright.
	\newblock A geometric analysis of phase retrieval.
	\newblock In {\em Information Theory (ISIT), 2016 IEEE International Symposium
		on}, pages 2379--2383. IEEE, 2016.
	
	\bibitem{sun2017complete}
	Ju~Sun, Qing Qu, and John Wright.
	\newblock Complete dictionary recovery over the sphere i: Overview and the
	geometric picture.
	\newblock {\em IEEE Transactions on Information Theory}, 63(2):853--884, 2017.
	
	\bibitem{venturi2018neural}
	Luca Venturi, Afonso Bandeira, and Joan Bruna.
	\newblock Neural networks with finite intrinsic dimension have no spurious
	valleys.
	\newblock {\em arXiv preprint arXiv:1802.06384}, 2018.
	
	\bibitem{zhang2016understanding}
	Chiyuan Zhang, Samy Bengio, Moritz Hardt, Benjamin Recht, and Oriol Vinyals.
	\newblock Understanding deep learning requires rethinking generalization.
	\newblock {\em arXiv preprint arXiv:1611.03530}, 2016.
	
	\bibitem{zhang2017global}
	Yuqian Zhang, Yenson Lau, Han-wen Kuo, Sky Cheung, Abhay Pasupathy, and John
	Wright.
	\newblock On the global geometry of sphere-constrained sparse blind
	deconvolution.
	\newblock In {\em Proceedings of the IEEE Conference on Computer Vision and
		Pattern Recognition}, pages 4894--4902, 2017.
	
\end{thebibliography}
\end{document}